\documentclass[onecolumn,12pt]{IEEEtran}
\usepackage{graphicx} %
\usepackage{subcaption}
\usepackage[utf8]{inputenc}
\usepackage{amsmath,amsfonts,amssymb,amsthm,epsfig,epstopdf,titling}
\usepackage{nicefrac}
\usepackage{url,array,mathrsfs,mathtools,dsfont,xargs}
\usepackage{stmaryrd}
\usepackage{bbold,xcolor}
\usepackage{graphicx}
\usepackage{scalerel}
\usepackage{hyperref}
\usepackage{geometry}
\usepackage{enumitem}
\usepackage[authoryear]{natbib}
\usepackage{floatrow}

\theoremstyle{plain}
\newtheorem{thrm}{Theorem}
\newtheorem{prop}{Proposition}
\newtheorem{lemma}{Lemma}
\newtheorem{coro}{Corollary}

\theoremstyle{definition}
\newtheorem{dftn}{Definition}

\theoremstyle{remark}

\newcommand\R{\mathbb R}
\newcommand\C{\mathbb C}
\newcommand\N{\mathbb N}
\newcommand\X{\mathbb X}
\newcommand\Pb{\mathbb P}
\newcommand\E{\mathbb E}

\newcommand\Var{\operatorname{Var}}
\newcommand\Leb{\operatorname{Leb}}
\newcommand\Card{\operatorname{Card}}
\newcommand\diam{\operatorname{diam}}
\newcommand\sign{\operatorname{sign}}

\newcommand\argmin{\operatorname{argmin}}
\DeclarePairedDelimiter\floor\lfloor\rfloor

\newcommand\s{\sigma}
\newcommand\la{\lambda}
\newcommand\1[1]{\mathbf{1}\mathopen{}\left(#1\right)}
\newcommand\norm[1]{\lVert #1 \rVert}
\newcommand\va[1]{\left\lvert #1 \right\rvert}
\newcommand\sva[1]{\lvert #1 \rvert}
\newcommand\vab{\sva{V^d}}
\newcommand\meanm{\frac{1}{m}\sum_{j=1}^m}
\newcommand\meann{\frac{1}{n}\sum_{i=1}^n}

\newcommandx\ik[1][1=k]{\hat i_{#1}(x)}
\newcommand\ta{\hat{\tau}_k}
\newcommand\tal{\hat{\tau}_{\ell}(x)}

\newcommandx\ieu[1][1=n]{\llbracket 1; #1 \rrbracket}
\newcommandx\cvg[3][1=, 2=n, 3=\infty]{\xrightarrow[#2 \to #3]{#1}}
\newcommandx\cvl[1][1=n]{\xRightarrow[#1\to\infty]{\mathcal L}}

\newcommand\pinf{\underline p}
\newcommand\psup{\bar p}
\newcommand\qsup{\bar q}

\newcommand{\lr}[3]{\left #1 #3 \right #2}  %
\newcommand{\mtlr}[3]{\mathopen{}\left #1 #3 \right #2}  %

\newcounter{subsubfigure}

\begin{document}

\title{Convergence rate for Nearest Neighbour matching: geometry of the domain and higher-order regularity}

\author{
Simon Viel\thanks{Univ Rennes, Ensai, CNRS, CREST -- UMR 9194, F-35000 Rennes, France, (email: 
\texttt{simon.viel@ensai.fr}).},\quad
Lionel Truquet\thanks{Univ Rennes, Ensai, CNRS, CREST -- UMR 9194, F-35000 Rennes, France (email: 
\texttt{lionel.truquet@ensai.fr}).},\quad
Ikko Yamane\thanks{Univ Rennes, Ensai, CNRS, CREST -- UMR 9194, F-35000 Rennes, France (email: 
\texttt{ikko.yamane@ensai.fr}).}
}

\maketitle
\begin{abstract}
\noindent

Estimating some mathematical expectations from partially observed data and in particular missing outcomes is a
central problem encountered in numerous fields such as transfer learning, counterfactual analysis or causal 
inference.  
\emph{Matching estimators}, estimators based on $k$-nearest neighbours, are widely used in this context.
It is known that the variance of such estimators can converge to zero at a parametric rate, but their bias can have a
slower rate when the dimension of the covariates is larger than $2$. This makes analysis of this bias particularly 
important. In this paper, we provide higher order properties of the bias. In contrast to the existing literature related to this problem, 
we do not assume that the support of the target distribution of the covariates is strictly included in that of the 
source, and we analyse two geometric conditions on the support that avoid such boundary bias problems. We show 
that these conditions are much more general than the usual convex support assumption, leading to an improvement of 
existing results. Furthermore, we show that the matching estimator studied by Abadie and Imbens (2006) for the 
average treatment effect can be asymptotically efficient when the dimension of the covariates is less than $4$, a
result only known in dimension $1$.
\end{abstract}
\begin{IEEEkeywords}
$k$-nearest neighbour estimators, transfer learning, average treatment effect, boundary bias
\end{IEEEkeywords}

\section{Introduction}\label{intro}

Estimating the expectation of a function of variables from a sample with some missing observations has attracted 
much effort in the statistical literature and $k$-nearest neighbours ($k$-NN) estimators are often useful in this 
context. In many interesting situations, these expectations concern pairs of random variables, covariates and a 
label, for which some labels are missing or partially missing in the data set. In the field of transfer learning 
or domain adaptation, the so-called covariate shift~\citep{shimodaira2000improving} arises in many real-world
situations such as natural language recognition \citep{bickel2006dirichlet}; \citep{jiang2007instance}, 
brain-computer interfaces \citep{sugiyama2007covariate}; \citep{li2010application}, pharmaceutics 
\citep{klarner2023drug}, and econometrics \citep{heckman1979sample}. In a simple but versatile formulation, a 
central task common in these applications is to estimate the expectation of a random variable of the form 
$Z^*:=h(X^*,Y^*)$ using only a sample of the unlabelled covariates $X^*$ and another training sample of the 
labelled pair $(X,Y)$, for which the two conditional distributions of $Y\vert X$ and $Y^*\vert X^*$ are identical 
while the covariates have different probability distributions, i.e., $\Pb_X\neq\Pb_{X^*}$. An important case 
arises in supervised learning when $h(x,y):=l(f(x),y)$ for a hypothesis function $f$ and a loss function $l$. In 
this case, our problem becomes risk estimation, which is central for empirical risk minimisation. When $P\equiv 
\Pb_X$ and $Q\equiv\Pb_{X^*}$ are absolutely continuous with respect to a reference measure, an expectation of 
$Z^*$ under $Q$ can be written as an expectation of $Z:=h(X,Y)$ under $P$ weighted by the density ratio $q/p$, 
where $p$ and $q$ denote respectively the densities of $P$ and $Q$ with respect to the reference measure.
\citet{shimodaira2000improving} considered such a reweighting for parametric maximum likelihood estimation in a 
mispecified case and some classical references discuss the use of linear-in-parameter models or the Reproducing 
Kernel Hilbert Space (RKHS) framework to estimate $q/p$. See for example \citet{gretton2009covariate}, 
\citet{kanamori2009least}, or \citet{sugiyama2008direct}.

Another important problem is the estimation of treatment effects in biostatistics or econometrics.
In these domains, a sample of triplets of random variables $(X,Y,W)\in\R^d\times\R\times\{0,1\}$ is observed. 
Using hidden random variables $Y(0)$ and $Y(1)$ called \emph{potential outcomes}, the observed response $Y$ is 
assumed to be either distributed as $Y(1)$ (if $W=1$, indicating that the corresponding individual is in the 
treated group) or distributed as $Y(0)$ (if $W=0$, indicating that the corresponding individual is in the control 
group). Only one of the two potential outcomes is observed for a given individual. The aim is then to estimate the 
expectation of the difference $Y(1)-Y(0)$, called the average treatment effect (ATE), conditional or unconditional 
on the treatment. The use of $k$-NN estimators has been widely studied in this context. Typically, the unobserved 
response $Y(1)$ or $Y(0)$ is replaced by an average of the outcomes corresponding to the $k$-NN covariate vectors 
in the other group. See for example Abadie and Imbens (\citeyear{abadie2006large}, \citeyear{abadie2008failure},
\citeyear{abadie2011bias}, \citeyear{abadie2012martingale}), or \citet{lin2023estimation}.
When $k$ is moderate, these estimators are often referred to as matching estimators.
This $k$-NN approach can be naturally translated to covariate shift adaptation and interpreted 
as density ratio estimators. This connection between matching and density ratio estimation has recently been
discussed in \citet{lin2023estimation}.

In a series of papers addressing the estimation of \emph{average treatment effects}, a problem related to ours,
Abadie and Imbens (\citeyear{abadie2006large}, \citeyear{abadie2008failure}, \citeyear{abadie2011bias}, 
\citeyear{abadie2012martingale}) introduced nearest-neighbour matching estimators and established several 
asymptotic results. \citet{lin2023estimation} further studied the properties of their matching estimators, and 
they showed, by letting the number of neighbours $k$ grow as the sample size increases under some relation, that 
the normalised matching estimator was consistent for the estimation of the density ratio.
Although we do not restrict ourselves to $k$ following this relation---for instance we allow the parameter $k$ to
remain constant---and therefore we lose the consistency result for the density ratio estimation in
\citet{lin2023estimation}, we show in this paper that the final risk estimator using NN-matching is still consistent.

Recently, $k$-NN estimators going beyond the approach of density ratio estimation have been proposed in the context 
of covariate shift adaptation~\citep{loog2012nearest, portier2023scalable, holzmann2024multivariate}.
The $k$-NN method that appeared in \citet{loog2012nearest} does not produce an estimate of the density ratio, but it 
was shown to have properties similar to the matching estimators~\citep{portier2023scalable}.
\citet{holzmann2024multivariate} extended the $k$-NN matching estimators by giving weights that can incorporate 
higher order smoothness.

Although our notation focuses mainly on covariate shift adaptation, both problems can be described in a similar 
framework. To estimate the expectation $e(h):=\E[h(X^*,Y^*)]$, we observe a sample $(X_i,Y_i),\,1\leq i\leq n$, 
of i.i.d. random vectors taking values in $\R^d\times\R$ as well as an additional sample of unlabelled i.i.d.
random variables $X^*_j,\,1\leq j\leq m$. We denote by $\X$ the measurable subset of $\R^d$ on which $X$ and $X^*$
take their values and consider the Euclidean norm $\norm\cdot$ on $\R^d$. We will discuss some statistical
properties of two estimators. The first one, introduced in \citet{portier2023scalable}, produces labels 
$\hat Y_j^*$ which are independent conditionally on the $(X_i,Y_i)$'s and the $X_j^{*}$'s and with a conditional 
distribution estimating $Y_j^{*}\vert X_j^{*}$. More precisely, denote by $\mathcal F_{m,n}=\sigma\left(X_i,Y_i,
X_j^*; 1\leq i\leq n,1\leq j\leq m\right)$ the sigma-algebra generated by those random variables. Their estimator 
is \begin{equation}\label{first}
\hat e_1(h):=\frac 1m\sum_{j=1}^mh\left(X_j^*,\hat Y_j^*\right),\quad\Pb\left(\hat Y_j^*=Y_i\mid\mathcal F_{m,n}
\right):=k^{-1}\1{\norm{X_i-X_j^*}\leq\ta(X_j^*)},
\end{equation}
where $\ta(x)$ is the $k$-nearest neighbour radius, that is, the $k$th-order statistic of the sample 
$\left(\norm{X_i-x}\right)_{1\leq i\leq n}$ for $x\in\X$ and $1\leq k\leq n$.
The second, which is similar to the matching estimators used to infer average treatment effects, is defined by 
\begin{equation}\label{second}
\hat e_2(h):=\frac 1n\sum_{i=1}^n\frac{nM_k^*(X_i)}{mk}\,h\left(X_i,Y_i\right),\quad M_k^*(x):=\sum_{j=1}^m
\1{\norm{x-X_j^*}\leq\ta(X_j^*)}.
\end{equation}

The estimator~\eqref{second} has recently been extended by \citet{holzmann2024multivariate} to a higher-order 
local polynomial estimator of the regression function $g\colon\X\to\R$ defined by
$$g(x):=\E[h(X,Y)\mid X=x]=\E[h(X^*,Y^*)\mid X^*=x].$$
It can be observed that $e(h)=\int_\X g(x)\,\mathrm dQ(x)$, which is simply the expectation of the regression 
function $g$ with respect to the target probability measure $Q$. Averaging the values of non-parametric estimates 
of the regression function evaluated with an additional sample of random variables with probability distribution 
$Q$ is then a natural idea. Note that the estimators~\eqref{first} and \eqref{second} are not very different. 
The estimator~\eqref{first} is the empirical mean $m^{-1}\sum_{j=1}^mh(X_j^{*},\hat Y_j^*)$ with $\hat Y_j^*$
following the conditional distribution
$$\hat\Pb_{Y\vert X}\left(\mathrm dy\mid X_j^*\right)=\frac 1k\sum_{i=1}^n\1{\norm{X_i-X_j^*}\leq\ta(X_j^*)}
\delta_{Y_i}(\mathrm dy),$$
while the estimator~\eqref{second} equals $m^{-1}\sum_{j=1}^m\hat g(X_j^*)$, where $\hat g_n(x)=\int_\X z\,
\hat\Pb_{Z\vert X}(\mathrm dz\vert x)$ and
$$\hat\Pb_{Z\vert X}\left(\mathrm dz\vert X_j^*\right)=\frac 1k\sum_{i=1}^n\1{\norm{X_i-X_j^*}\leq\ta(X_j^*)}
\delta_{h(Y_i,X_i)}(\mathrm dz).$$
Notice that $\hat g_n(x)$ is the standard $k$-NN nonparametric regression estimator. As we will see, the 
bias/variance of the estimators~\eqref{first} and \eqref{second} have the same convergence rates under quite 
similar regularity assumptions.
 
As is now well known in this literature, estimators constructed from an averaging of $k$-NN estimators of a 
regression function often have a variance that converges to $0$ at a parametric rate even when $k$ is small, $k=1$ 
being a standard choice. Then the regression function needs not to be consistently estimated to get this result.
This rate for the variance has also been obtained in some references that focus on estimating some expectations of 
a regression function in a different context. For example, \citet{Dev1} and \citet{Dev2} studied the estimation of 
some mean squared errors in non-parametric estimation using two identically distributed samples, from which an 
averaging of $1$-NN estimators leads to a variance converging to $0$ at a parametric rate. Estimators of some 
integrals of a probability density, useful for estimating entropies or divergences, have also been studied in 
\citet{sricharan2012estimation} and \citet{singh2016finite}, with a similar convergence rate for the variance.

The quality of these matching-type estimators then depends on their bias, the convergence rate of which generally
depends on the dimension of the covariate vector. If the regression function $g$ has Lipschitz properties on a 
bounded domain $\X$, the bias of the estimators (\ref{first}) or (\ref{second}) can be bounded by $k^{1/d}/
n^{1/d}$, up to a constant. However, if $g$ is two times continuously differentiable, one could expect the rate 
$k^{2/d}/n^{2/d}$, but as we will see, obtaining this rate requires a geometric analysis of the boundary of the 
domain $\X$ to avoid boundary bias problems. In the literature, this faster bias rate has been obtained by 
\citet[Theorem~2]{abadie2006large}, or using the local polynomial extension of the estimator~\eqref{second} 
recently studied by \citet{holzmann2024multivariate}, but assuming that the support of the measure $Q$ is 
contained in a compact subset of the interior of a convex domain $\X$, a simple condition that avoids boundary 
bias problems. In information theory, such a condition means that the relative entropy of $P$ with respect to $Q$, 
measured by the Kullback-Leibler divergence, is automatically infinite, and it excludes two probability measures 
with the same support. For the study of standard ATE, the two distributions $P$ and $Q$ of the covariates 
conditional on treatment or conditional on non-treatment are assumed to be mutually absolutely continuous, and 
such a condition on the support is not satisfied. 
 
The contribution of this paper is threefold. \begin{enumerate}
\item First, we show that under some suitable conditions on the domain $\X$ and smoothness conditions, the 
estimators (\ref{first}) and (\ref{second}) have a bias of order $k^{2/d}/n^{2/d}$. This shows that the parameter 
$e(h)$ can be estimated at a parametric rate when $d\leq 4$, using only $k=1$, a choice that leads to fast 
computation for the numerical implementation of these estimators. For ATE, we show that the matching estimators 
studied in \citet{abadie2006large} can be efficient in a semi-parametric sense when $d\leq 3$. We then obtain an 
extension of Corollary $4.1$ in \citet{lin2023estimation}, which was only stated for the univariate case $d=1$.
\item
We analyse two geometric conditions that avoid the bias boundary problem. These conditions are generally satisfied
if the boundary of the domain $\X$ is sufficiently smooth or for a finite union of well-shaped subsets of $\R^d$ 
such as convex sets.
\item We show that such geometric conditions are also sufficient to avoid the boundary bias problem of the $k$-NN
local polynomial estimator of \citet{holzmann2024multivariate}, who recently studied a very nice extension of 
matching estimators for obtaining parametric convergence rates for estimating $g$, under suitable smoothness 
conditions on $g$. Such an extension provides an alternative to the doubly robust estimators studied in 
\citet{lin2023estimation}, for which additional hyperparameters have to be tuned.
\end{enumerate}
The paper is organised as follows. In Section~\ref{S2}, we analyse the second-order properties of the bias for the
estimators~\eqref{first} and \eqref{second} using appropriate geometric conditions on the support of the 
covariates. We also give some consequences for the estimation of ATE in Section~\ref{ate}. A discussion of 
various sufficient conditions for checking our geometric conditions is given in Section~\ref{geometry}. In 
Section~\ref{local} we show that such conditions are sufficient to control higher order properties of the bias 
for a local polynomial extension of the previous estimates. Numerical experiments are given in
Section~\ref{numerical}. A conclusion is given in Section \ref{conclusion} and proofs of our results are given in 
Section \ref{proofs}. Finally, some technical lemmas are collected in an Appendix Section.

\section{Second-order bias property of matching estimators}\label{S2}

\subsection{Notation and assumptions}\label{assumptions}

For any $r\geq 0$ and $x\in\R^d$, let $B(x,r):=\left\{x'\in\R^d:\norm{x'-x}\leq r\right\}$ be the closed ball 
of centre $x$ and radius $r,\,S(x,r)$ be the sphere of center $x$ and radius $r$, $S^{d-1}:=S(0,1)$ be the unit
sphere, and for a given Borel subset $A$ of $\R^d$ we denote by $\1 A$ the indicator function which takes the 
value $1$ on $A$ and $0$ on its complement. The conditional bias of estimators (\ref{first}) and (\ref{second}) is
defined as $$B_{i,n}:=\E[\hat e_i(h)\mid X]-e(h),$$ 
for $i=1,2$ respectively. We first describe the bias problem due to boundary issues. For simplicity, we only 
consider the estimator~\eqref{first} (the case $i=1$). The bias for the second estimator behaves in a similar way. 
For $\ell=1,\dots,k$, let us denote by $\ik[\ell]$ the index of the $\ell$th-nearest neighbour of $x\in\X$ in 
the sample $(X_1,\dots,X_n)$. Let us also set $\Delta(x,u):=\E[h(x,Y_1)\vert X_1=u]$. We then have $\Delta(x,x)=
g(x)$ and assuming that the mapping $\Delta$ is two times continuously differentiable with respect to its second 
argument, we could expect that
$$B_{1,n}=k^{-1}\sum_{\ell=1}^k\int_\X\left[\Delta\left(x,X_{\ik[\ell]}\right)-g(x)\right]\,\mathrm dQ(x)
\approx k^{-1}\sum_{\ell=1}^k\int_\X\nabla_2\Delta(x,x)^\top Z_\ell(x)\,\mathrm dQ(x),$$
where $Z_\ell(x):=X_{\hat i_\ell(x)}-x$ and $\nabla_2\Delta(x,x)$ denotes the gradient of the mapping $u\mapsto
\Delta(x,u)$ with respect to its second argument evaluated at point $x$.
Note that $\norm{Z_\ell(x)}=\norm{X_{\ik[\ell]}-x}$ can be always bounded by $\ta(x)$, which is of order 
$(k/n)^{1/d}$, using Lemma~\ref{moments_tau}. This leads to the same rate for the bias.

However, one can get a better rate for the expectation of the vector $Z_\ell(x)$ if we restrict it to the event
$\{\tal\leq\delta(x)\}$, where $\delta(x)$ denotes the distance between $x$ and the complement $\X^c$ of $\X$. On
this event, the ball centered at $x$ with radius $\tal$ is included in $\X$. Moreover, conditionally on the radius
$\tal$, the normalized vector $Z_\ell(x)/\tal$ has a density, supported on $S^{d-1}$, given by $\theta\mapsto
p(x+\tal\theta)$, up to a normalizing constant denoted here by $\hat C_\ell(x)$. We then deduce that 
\begin{multline*}
\E[Z_\ell(x)\1{\tal\leq\delta(x)}] \\ =\E\left[\int_{S^{d-1}}\hat C_\ell(x)\,\tal\,\theta\,(p(x+\tal\theta)-p(x))
\,\mathrm d\s(\theta)\,\1{\tal\leq\delta(x)}\right],
\end{multline*}
where $\s$ denotes the Haar measure on the unit sphere. Assuming that the density $p$ is Lipschitz and positive on 
the compact set $\X$, one can get, for a generic constant $C>0$,
$$\norm{\E[Z_\ell(x)]}\leq C\,\E[\ta(x)^2+\ta(x)\,\1{\ta(x)>\delta(x)}].$$
Therefore, if $\nabla_2\Delta$ is bounded, we gain an order in the expectation of the bias as soon as
$$\int_\X\E[\ta(x)\,\1{\ta(x)>\delta(x)}]\,\mathrm dQ(x)\leq C(k/n)^{2/d}.$$
Thanks to Lemma \ref{censored_tau}, this property holds true as soon as the following condition holds true. 
\begin{enumerate}[label=(A), wide=0.5em, leftmargin=*]
\item\label{cond:regA} $\underset{L>0}\sup\,\{L^{1/d}\int_\X\exp(-L\,\delta(x)^d)\;\mathrm dQ(x)\}<\infty$.
\end{enumerate}
These derivations will be formalized in our proofs, with an upper-bound for $\E[B_{i,n}^2]$, which is necessary 
for bounding the mean squared error of $\hat e_i(h)$.

When the support of $Q$ is included in the interior of $\X$, i.e. $\{x\in\X:\delta(x)\geq\varepsilon\}$ for some 
$\varepsilon>0$, condition \ref{cond:regA} is trivially satisfied. However, such an assumption is quite 
restrictive and cannot be used for studying the standard average treatment effect. See Section~\ref{ate} for 
details. Without such a restriction on the target measure $Q$, condition \ref{cond:regA} is related to
the geometry of the boundary of $\X$. We defer the reader to Section \ref{geometry} for a discussion on this 
condition.

Though our work is mainly devoted to the control of the bias, we will also analyse the variance term $\hat e_i(h)-
\E[\hat e_i(h)\mid X]$ from an original exponential bound for the tail distribution function of the volume
$Q(A_k(z))$, where $A_k(z):=\{x\in\X\,:\,\norm{z-x}\leq\ta(x)\}$ is called the \emph{catchment area} of point $x$,
see \citet{abadie2006large} or \citet{lin2023estimation}. The results will be presented in Subsection
\ref{rate_variance}.

In the sequel, we will refer to the following assumptions. 
\begin{enumerate}[label=(X\arabic*), wide=0.5em, leftmargin=*]
\item\label{cond:reg1} $\X$ is a compact subset of $\R^d$.
\item\label{cond:reg2} There exists a constant $c>0$ such that for all $x\in\X$ and all $0\le r\le\diam(\X)$, 
$\va{B(x,r)\cap\X}\geq c\va{B(x,r)}$, where $\va B$ denotes the Lebesgue measure of a Borel set $B$ in $\R^d$.
\item\label{cond:reg3} The distributions $P$ and $Q$ are absolutely continuous with respect to the Lebesgue 
measure on $\R^d$, with respective densities $p$ and $q$. Furthermore, there exist $\pinf,\,\qsup>0$ such that for
all $x\in\X,\,p(x)\geq\pinf$ and $q(x)\leq\qsup$.
\item\label{cond:reg4} The function $g$ is Lipschitz-continuous on $\X$.
\end{enumerate}

\textbf{Notes}: \begin{enumerate}
\item Assumption \ref{cond:reg2} requires that the intersection of a ball with the probability support $\X$ should 
have a volume comparable to that of the ball when the center is inside the support. Such an assumption has been 
widely used in previous studies. See \citet{Gadat}, \citet{lin2023estimation}, or \citet{portier2023scalable} 
among others. As for Assumption \ref{cond:regA}, such a condition can be seen as a geometric condition for the 
boundary of the probability support $\X$. Sufficient conditions ensuring the validity of Assumption 
\ref{cond:reg2} or Assumption \ref{cond:regA} will be studied in Section \ref{geometry}, as well as some examples 
showing that these two conditions are not equivalent in general. Assumptions \ref{cond:reg1} and \ref{cond:reg3} 
are also classical regularity conditions for studying $k$-NN type estimators. 
\item Assumption \ref{cond:reg4} is a standard regularity condition for some conditional expectation and which 
is often compatible with a bounded probability support $\X$.
\end{enumerate}

\subsection{Rate for the bias}\label{rate_bias}
\bigskip

Simply assuming that the regression function $g$ is Lipschitz-continuous and using the bounds on the moments of 
the $k$-NN radius as expressed in Lemma \ref{moments_tau}, we can obtain the first control of the conditional bias
$B_{2,n}:=\E[\hat e_2(h)\mid X]-e_2(h)$. For $B_{1,n}$, the same control has been already obtained in Proposition
$3$ in \citet{portier2023scalable} under similar assumptions.

\begin{prop}\label{bias}
Under Assumptions \ref{cond:reg1} to \ref{cond:reg4}, for all $\la\ge 1$, we have $$\E[\va{B_{2,n}}^\la]\leq 
C_{\la,d,P,h}\;\left(\frac k{n+1}\right)^{\la/d},$$ where $C_{\la,d,P,h}:=L^\la\,2\,\Gamma(2+\floor{\la/d})\,
(c\pinf\,\vab)^{-\la/d}$ and $L$ is the Lipschitz constant for $g$. Here, $\Gamma$ denotes Euler's gamma function 
and $\floor a$ denotes the integer part of $a\in\R$.
\end{prop}

We will now prove that under stronger conditions on the function $g$ and the distribution $P$, we can derive 
better rates for the second-order bias $\E[B_{i,n}^2]$ for $i=1,2$. In what follows, for any $x\in\X$, we denote 
by $\nabla^2_2\Delta(x,y)$ the Hessian matrix of $z\mapsto \Delta(x,z)$ at point $y\in\X$. The Hessian matrix of 
$g$ at point $x$ will be simply denoted by $\nabla^2 g(x)$. We will also denote by $\norm\cdot$ the operator norm 
on the space of $d\times d$ matrices with real coefficients, corresponding to the norm $\norm\cdot$ used on $\R^d$.
We will therefore consider the following two assumptions. In what follows, we denote by $U$ an open convex and
bounded set containing $\X$ and by $\mbox{co}(\X)$ the convex hull of $\X$. We recall that $\mbox{co}(\X)$ is 
compact if $\X$ itself is compact.
\begin{enumerate}[label=(X5), wide=0.5em, leftmargin=*]
\item\label{cond:reg5} The density $p$ is Lipschitz-continuous on $\X$.
\end{enumerate}
\begin{enumerate}[label=(X6-\arabic*), wide=0.5em, leftmargin=*]
\item\label{cond:reg61} For any $x\in\X$, the mapping $\Delta(x,\cdot)$ is the restriction to $\X$ of a mapping 
from $U$ to $\R$, still denoted by $\Delta(x,\cdot)$, two-times continuously differentiable on $U$ and such that 
$$\sup_{(x,y)\in\X\times \mbox{co}(\X)}\left\{\norm{\nabla_2\Delta(x,y)}+\norm{\nabla_2^2\Delta(x,y)}\right\}<
\infty.$$
\item\label{cond:reg62} The mapping $g$ is a restriction to $\X$ of a two-times continuously differentiable mapping
from $U$ to $\R$, still denoted by $g$.
\end{enumerate}

Allowing the extension of the mappings $\Delta$ or $g$ to differentiable mappings defined on a convex 
neighbourhood of $\X$ is only a technical assumption which ensures Lipschitz properties for these mappings and
their gradients restricted to the set $\X$. Such properties are not necessarily valid for differentiable mappings 
only defined on the set $\X$. It also avoids the choice of a definition for continuously differentiable functions 
on the boundary. See \citet{boundary} for a discussion.

\begin{thrm}\label{better_bias}
Suppose that Assumptions \ref{cond:reg1} to \ref{cond:reg5} hold true, as well as Assumption \ref{cond:regA} and
Assumption \ref{cond:reg61} if $i=1$ or Assumption to \ref{cond:reg62} if $i=2$. Then there exist constants 
$C_{i,d,P,Q,h}$ such that for $i=1,2$, $$\E[B_{i,n}^2]\leq C_{i,d,P,Q,h}\,\left(\frac k{n+1}\right)^{\min\{4/d\,,
\,3\}}.$$ 
\end{thrm}

\paragraph{Note} When $d\geq 2$, one can note that the convergence rate for the bias is $(k/n)^{2/d}$ which is the
square of the rate obtained from Proposition \ref{bias}, which only assumes a Lipschitz regularity for the 
conditional expectation. When $d=1$, the upper-bound for the bias is $(k/n)^{3/2}$ whereas a Lipschitz condition 
only guarantees the rate $k/n$. In comparison, the standard deviation of our estimators is $1/\sqrt n+1/\sqrt m$, 
as discussed in the next paragraph. When  $k=O(1)$, we then conclude that the second-order regularity of the 
conditional expectation helps to improve the convergence rate for the RMSE only for $d\geq 2$, the bias being 
already negligible with respect to the standard deviation in the Lipschitz case when $d=1$.

We next give a second-order bias expansion which is valid with a bit more regularity on the mappings $g,\,\Delta$,
and $p$, and under a more restrictive geometric condition, the inclusion of the support of $Q$ in the interior of
$\X$. Though quite restrictive, this condition is very common in the literature and useful to avoid boundary bias
problems. Moreover, it automatically entails the validity of Assumption \ref{cond:regA}. In the context of ATE, a
similar bias expansion is given in \citet[Theorem~2]{abadie2006large}. However, our expansion is more general here 
since we also consider the expectation of the square of $B_{i,n}$, a quantity which is important since it appears 
in the MSE, and $k$ is not necessarily bounded.

\begin{thrm}\label{bias_expansion}
Suppose that Assumptions \ref{cond:reg1} to \ref{cond:reg5} and Assumption \ref{cond:reg61} if $i=1$ or Assumption
\ref{cond:reg62} if $i=2$ hold true. Suppose furthermore that there exists $\beta\in(0,1)$ such that $\nabla^2g$ 
is $\beta$-Hölder continuous on $\mbox{co}(\X)$ if $i=2$ or $$\sup_{\substack{(x,y,z)\in\X\times\mbox{co}(\X)^2 \\ 
y\neq z}}\frac{\norm{\nabla_2^2 \Delta(x,y)-\nabla_2^2\Delta(x,z)}}{\norm{y-z}^\beta}<\infty,$$
if $i=1$. Finally, assume that $k^{d+1}=o(n)$, that the source density $p$ is a restriction to $\X$ of an element
$\overline p\in C^1(U)$ with $\nabla\overline p$ is $\eta$-Hölder continuous on $co(\X)$ for some $\eta\in(0,1)$, 
and that the support of $Q$ lies in the interior of $\X$. 
Then, for $i=1,2$, we have the following expansion for the first order bias
$$\E[B_{i,n}]=C_i(k,P,Q,h)\,n^{-2/d}+o((k/n)^{2/d}),$$ where
$$C_i(k,P,Q,h):=\frac 1k\sum_{\ell=1}^k\frac{\Gamma(\ell+2/d)}{\Gamma(\ell)}\int_\X(p(x)\,\vab)^{-2/d}\Psi_i(x)\,
\mathrm dQ(x)\text{, and}$$ $$\Psi_1(x):=\frac 1{\s(S^{d-1})}\int_{S^{d-1}}\theta^\top\left(\nabla\Delta(x,x)^\top
\frac{\nabla p(x)}{p(x)}+\frac{\nabla_2^2\Delta(x,x)}2\right)\theta\,\mathrm d\s(\theta),$$ 
$$\Psi_2(x):=\frac 1{\s(S^{d-1})}\int_{S^{d-1}}\theta^\top\left(\nabla g(x)^\top\frac{\nabla p(x)}{p(x)}+
\frac{\nabla^2g(x)}2\right)\theta\,\mathrm d\s(\theta).$$ 
\bigskip

Similarly, when the dimension of the covariates $d$ is greater than $2$, for $i=1,2$, we have the
following expansion for the second order bias $$\E[B_{i,n}^2]=C_i(k,P,Q,h)^2\,n^{-4/d}+o((k/n)^{4/d}).$$
\end{thrm}

\bigskip

\paragraph{Note.} Under the assumptions of Theorem~\ref{bias_expansion}, the support of $Q$, which 
is defined as the smallest closed subset of $\R^d$ with measure $1$, is a compact subset of $\X$ and its distance 
to the boundary of $\X$ is then positive. When $\X=[0,1]$, the support of $Q$ should be included in 
$[\epsilon,1-\epsilon]$ for some $\epsilon\in (0,1/2)$, which excludes many natural probability distributions such 
as all beta distributions and then a simple uniform distribution. \\
If we relax the condition of inclusion of the supports and we replace it with the more general geometric condition
\ref{cond:regA}, as discussed in subsection \ref{assumptions}, then we no longer have access to this bias 
expansion but we still have an upper bound on the second-order bias.

\paragraph{Note.} As mentioned in the proof of Lemma $3$ in \citet{portier2023scalable}, one can find $L_1,L_2>0$ 
not depending on $\ell\geq 1$ such that $L_1\ell^{2/d}\leq\Gamma\left(\ell+2/d\right)/\Gamma\left(\ell\right)\leq 
L_2\ell^{2/d}$. We then observe that the leading term in the bias expansion is really of order $k^{2/d}/n^{2/d}$.

\subsection{Rate for the variance}\label{rate_variance}

Numerous studies have examined the variance term $V_{m,n}:=\hat e_i(h)-\E[\hat e(h)\mid X]$ and have proved that 
it converges at a parametric rate $\E[V_{m,n}^2]\leq C(m^{-1}+n^{-1})$. Interestingly, we can obtain this result 
as a consequence of a concentration inequality on the volume of catchment areas. Recall that the catchment
area of a point $z\in\X$ is defined as the set $A_k(z):=\{x\in\X\,:\,\norm{z-x}\leq\ta(x)\}$. 
\begin{thrm}\label{catchment}
Suppose that Assumptions \ref{cond:reg1} to \ref{cond:reg3} hold true, then for all $t>0$, we have the inequality
$$\Pb(nQ(A_k(X_1))\geq kt)\leq 3e^{1/4}\exp\left(-\frac{c\pinf}{12\qsup}\,t\right).$$
\end{thrm}

With the control on the bias and the above concentration inequality, we can now conclude on the mean squared error
of our estimator, adding the following assumption. 
\begin{enumerate}[label=(X7), wide=0.5em, leftmargin=*]
\item\label{cond:reg7} The random variable $\Var(h(X,Y)\mid X)$ is bounded above by some constant $\bar\s^2$.
\end{enumerate}

\begin{coro}\label{overall_speed}
Suppose that Assumptions \ref{cond:reg1} to \ref{cond:reg5} hold true, as well as Assumption \ref{cond:reg61} if 
$i=1$ or Assumption to \ref{cond:reg62} if $i=2$, Assumptions \ref{cond:reg7} and Assumption \ref{cond:regA}. Then
for all $m,n\in\N^*$ and all $1\le k\le n$, we have $$\E[(\hat e_i(h)-e(h))^2]\leq C\left(\left(\frac k{n+1}
\right)^{\min\{4/d\,,\,3\}}+\frac 1m+\frac 1n\right),$$ where $C>0$ is a constant depending on the dimension $d$, 
the function $h$ and the distributions $P$ and $Q$.
\end{coro}

\paragraph{Note} As a consequence of Corollary \ref{overall_speed}, if we choose $k=1$, our estimator converges at
a parametric rate when the dimension of the covariates $d$ does not exceed $4$. In contrast, as shown by 
Proposition \ref{bias}, a Lipschitz regularity for the conditional expectation only guarantees a parametric rate 
of convergence when $d$ does not exceed $2$. Such an improvement, assuming second order regularity for the 
conditional expectation 
as well as the additional geometric condition \ref{cond:regA} on the support $\X$ of the covariates, is one of the 
main result of the paper.

\section{Application to average treatment effects}\label{ate}

In the setting of average treatment effect, as described in \citet{abadie2006large}, $N$ units $i=1,\dots,N$, each
described by a vector of covariates $X_i$, are given either an active treatment $W_i=1$ or a control treatment
$W_i=0$. For each unit, out of the two potential outcomes $Y_i(0)$ and $Y_i(1)$, only one is observed, namely 
$Y_i=Y_i(W_i)$. The quantities of interest are the global average treatment effect $\mu:=\E[Y(1)-Y(0)]$ and the
average treatment effect on the treated $\mu^t:=\E[Y(1)-Y(0)\mid W=1]$. We will assume that the random vectors 
$\left(W_i,X_i,Y_i(0),Y_i(1)\right)$ are i.i.d.\@ for $1\leq i\leq N$ and are some replications of a generic random 
vector denoted $\left(W,X,Y(0),Y(1)\right)$. While it is straightforward to estimate $\E\left[Y(w)\mid W=w\right]$ 
by a simple empirical average for $w=0,1$, the estimation of the two expectations $\E\left[Y(1-w)\mid W=w\right]$ 
for $w=0,1$ is not obvious since the outcome $Y(1-w)$ is never observed in the group assigned with $W = w$.
Related to our setup, the source or target probability measures $P$ or $Q$ equal to either the probability 
distribution $X\vert W=0$ or $X\vert W=1$, depending on the expectation $\E\left[Y(1)\mid W=0\right]$ or 
$\E\left[Y(0)\mid W=1\right]$ we have to estimate. Then, as in this paper, \citet{abadie2006large} replaced the 
missing outputs by an average on the observed outputs of the $k$ units having the nearest covariates. For $x\in
\X$, $w\in\{0,1\}$ and $\ell\in\ieu[k]$, set $\hat\tau_{\ell,w}(x)$ and $\hat i_{\ell,w}(x)$ respectively the 
distance between $x$ and its $\ell$th-nearest neighbour among the $X_i,\,i=1,\dots,N$ such that $W_i=w$, and the 
corresponding index. Setting 
$$M_k^*(X_i)=\sum_{j=1}^N\1{W_j=1-W_i,\norm{X_i-X_j}\leq\hat\tau_{k,W_i}(X_j)},$$ it led them to the two 
estimators $$\hat\mu:=\frac 1N\sum_{i=1}^N(2W_i-1)\left(1+\frac{M_k^*(X_i)}k\right)Y_i\mbox{ and }\hat\mu^t:=
\frac 1{N_1}\sum_{i=1}^N\left(W_i-(1-W_i)\frac{M_k^*(X_i)}k\right)Y_i$$ for the average treatment effect (ATE) and 
the average treatment effect on the treated (ATT), respectively, where $N_1$ is the number of units that were 
given the active treatment. In what follows, we only focus on the ATE, but similar results are valid for the ATT.  
One can decompose the difference $\hat\mu-\mu$ as $\hat\mu-\mu=E+B+V$, where \begin{align*}
E & :=\frac 1N\sum_{i=1}^N(2W_i-1)\left(1+\frac{M_k^*(X_i)}k\right)(Y_i-g_{W_i}(X_i)), \\ B & :=\frac 1N
\sum_{i=1}^N(2W_i-1)\left(g_{1-W_i}(X_i)-\frac 1k\sum_{\ell=1}^kg_{1-W_i}\left(X_{\hat i_{\ell,1-W_i}(X_i)}\right)
\right),\text{ and} \\ V & :=\frac 1N\sum_{i=1}^N(g_1(X_i)-g_0(X_i))-\mu,
\end{align*}
where $g_w(x):=\E[Y(w)\mid X=x]$ and the nearest neighbours are always taken from the other subpopulation. They 
showed that both the variance term $V$ and the residual term $E$ are of order $N^{-1/2}$ while the conditional 
bias is of order at least $N^{-1/d}$---they ask $k$ to be constant---and similarly for the average treatment 
effect on the treated. Furthermore, under extra regularity assumptions and the additional condition that
\begin{itemize}
\item the support $\X_1$ of $X$ given $W=1$ is a compact subset of the interior of the support $\X_0$ of $X$ given
$W=0$,
\end{itemize}
they proved in their Theorem $2$ that the unconditional bias on the treated $\E[B^t]$ was of order $N^{-2/d}$.
However, because the above condition cannot be asked to be symmetric in terms of treatment $w$, in other words we 
cannot simultaneously have $\X_1\subset\X_0^\circ$ and $\X_0\subset\X_1^\circ$, they could not extend this result 
to the bias $B$ of the global average treatment effect.

With our results, we can prove that under the next less restrictive geometric condition \ref{cond:T8} below,  the conditional bias term $B$ is of order
$(k/N)^{2/d}$, and $(k/N)^{3/2}$ instead if $d=1$. 

To state the theorem, we consider the following assumptions. 
\begin{enumerate}[label=(T\arabic*), wide=0.5em, leftmargin=*]
\item\label{cond:T1} The support $\X$ of $X$ is a compact subset of $\R^d$.
\item\label{cond:T2} There exists a constant $c>0$ such that for all $x\in\X$ and all $0\le r\le\diam(\X)$, 
$\va{B(x,r)\cap\X}\geq c\va{B(x,r)}$, where $\va B$ denotes the Lebesgue measure of a Borel set $B$ in $\R^d$.
\item\label{cond:T3} The random vector $X$ has a density $f$ with respect to the Lebesgue measure on $\R^d$ and
there exist $\underline f,\bar f>0$ such that for almost all $x\in\X,\,\underline f\leq f(x)\leq\bar f$.
\item\label{cond:T4}(\emph{Unconfoundedness}) For almost all $x\in\X,\,W$ is independent of $(Y(0),Y(1))$ 
conditionally on $X=x$.
\item\label{cond:T5} There exists $\eta>0$ such that for almost all $x\in\X,\,\eta<\Pb(W=1\mid X=x)<1-\eta$.
\item\label{cond:T6} The regression functions $g_0$ and $g_1$ are restrictions of two-times
continuously differentiable mappings on an open convex neighbourhood of $\X$.
\item\label{cond:T7} The probability distributions $X\vert W=0$ and $X\vert W=1$ have densities with respect to the
Lebesgue measure on $\R^d$, denoted respectively by $f_0$ and $f_1$ and which are Lipschitz-continuous on $\X$.
\item\label{cond:T8} $\underset{L>0,\,w\in\{0,1\}}\sup\left\{L^{1/d}\int_\X\exp(-L\,\delta(x,\X^c)^d)f_w(x)\;
\mathrm dx\right\}<\infty$.
\end{enumerate}

\begin{thrm}\label{gate_bias}
Suppose that Assumptions \ref{cond:T1} to \ref{cond:T8} hold true, then the conditional bias satisfies 
$$\E[B^2]\leq C\left(\frac kN\right)^{3\wedge 4/d},$$ where $C>0$ depends on $d,\X,\eta,g_0,g_1,f_0$, and $f_1$.
\end{thrm}

Improving on the order of the conditional bias means that there are more cases where it is negligible compared to
the other two terms than in \citet{lin2023estimation}. The following result is a consequence of Theorem 
\ref{gate_bias} and of the asymptotic normality of $\sqrt{N}(E+V)$. See Lemma C$1$ in \citet{lin2023estimation} for 
such a weak convergence.

\begin{coro}\label{ate_efficiency}
Suppose that Assumptions \ref{cond:T1} to \ref{cond:T8} hold true. If $k$ diverges and if we have one of the
following: \begin{itemize}
\item $d=1$ and $k^3/N^2\xrightarrow{}0$,
\item $d=2$ and $k^2/N\xrightarrow{}0$, or
\item $d=3$ and $k^4/N\xrightarrow{}0$,
\end{itemize}
then the average treatment effect matching estimator with no bias correction attains the semiparametric 
efficiency lower bound \citep{hahn1998role}. More precisely
$$\sqrt N\left(\hat\mu-\mu\right)\Rightarrow\mathcal N\left(0,\sigma^2\right),$$ where
$$\sigma^2:=\E\left[\frac{\sigma_1^2(X_1)}{e(X_1)}+\frac{\sigma_0(X_1)^2}{1-e(X_1)}+\left(g_1(X_1)-g_0(X_1)-\tau
\right)^2\right],$$ and for $(x,w)\in\X\times\{0,1\},\,e(x):=\E\left(W_1\mid X_1=x\right)$ and $\sigma_w(x)^2:=
\Var\left(Y_1(w)\mid X_1=x\right)$.
\end{coro}

\paragraph{Note.} When $d=4$ and $k$ is bounded, we still have $\sqrt NB=O_P(1)$ but the stochastic bias is no 
longer negligible and Corollary $3$ does not extend to this case, though the three terms $E,V,B$ are of the same
order.

\section{Analysis of two geometric conditions}\label{geometry}

This section focuses on the following two geometric conditions that address the boundary bias problem:
\begin{enumerate}[wide=0.5em, leftmargin=*]
\item[\ref{cond:reg2}] There exists a constant $c>0$ such that for all $x\in\X$ and all $0\le r\le\diam(\X)$, 
$\va{B(x,r)\cap\X}\geq c\va{B(x,r)}$.
\item[\ref{cond:regA}] $\underset{L>0}\sup\,\left\{L^{1/d}\int_\X\exp(-L\,\delta(x,\X^c)^d)\;\mathrm dQ(x)\right\}<
\infty$.
\end{enumerate}

To show that these two conditions are independent in that none of them imply the other, we consider the following 
two counter-examples. In the following result, $Q$ is defined as the uniform distribution on $\X$. 

\begin{prop}\label{assumption_counter}
\begin{enumerate}
\item Consider $\X:=\{(x,y)\in\R^2\,:\,x\in[0,1],\,y\in[0,x^2]\}$. Then $\X$ satisfies Condition~\ref{cond:regA}
but not Condition~\ref{cond:reg2}.
\item Consider $\X:=\bigcup_{k\in\N^*}\mathcal C_k$, where $\mathcal C_k:=\{x\in\R^2\,:\,a_k\leq\norm x\leq b_k\}$.
We choose $a_k:=(k+1)^{-1}$ and $b_k\in(a_k,a_{k-1})$ such that for all $k\ge 1,\,0<a_{k-1}-b_k\leq\delta/(k+1)^4$
with $\delta$ small enough. Then $\X$ satisfies Condition \ref{cond:reg2} but not Condition \ref{cond:regA}.
\end{enumerate}
\end{prop}

\begin{figure}[tbp]
\centering
\begin{subfigure}[tbp]{0.40\textwidth}
\includegraphics[width=\linewidth]{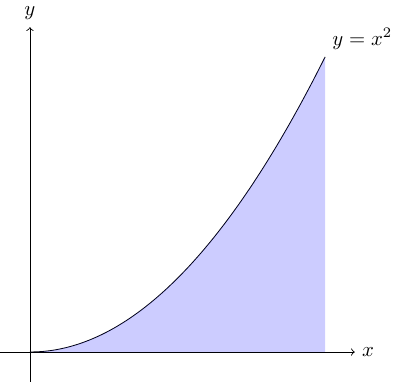}
\subcaption{Example~1}
\label{subfig:under_xsq}
\end{subfigure}
\hfill
\begin{subfigure}[tbp]{0.40\textwidth}
\includegraphics[width=\linewidth]{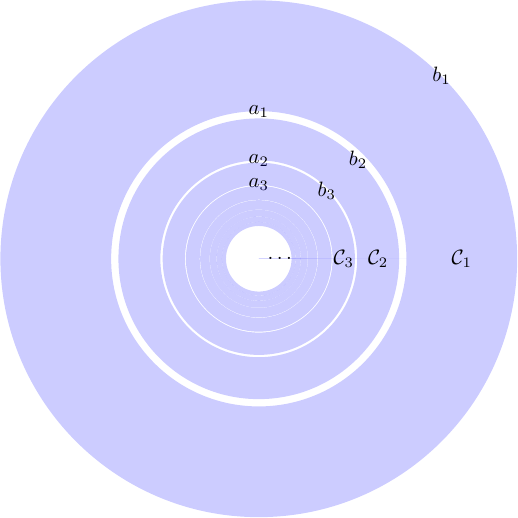}
\subcaption{Example~2}
\label{subfig:rings}
\end{subfigure}
\caption{Illustrations of Proposition~\ref{assumption_counter}}
\label{fig:illust_assumption_counter}
\end{figure}

For illustrations of Proposition~\ref{assumption_counter}, see Figure~\ref{fig:illust_assumption_counter}.

In the next theorem, we describe general settings in which the two conditions are simultaneously satisfied. 

\begin{thrm}\label{geometry_cases}
Suppose that the density $q$ is bounded above, then Conditions \ref{cond:reg2} and \ref{cond:regA} are 
simultaneously satisfied in either of the following two cases: \begin{itemize}
\item $\X$ is a compact convex subset of $\R^d$ with a non-empty interior; or
\item $\X$ is the closure of a bounded open set whose boundary $\partial\X$ is a $(d-1)$-dimensional submanifold 
of $\R^d$ of class $C^1$.
\end{itemize}
Furthermore, both conditions are stable by finite union.
\end{thrm}

Let us comment on these results. Under convexity or smooth boundary assumptions, as formulated in the statement of 
Theorem \ref{geometry_cases}, Condition \ref{cond:reg2} is satisfied since it is possible to show that the trace of 
a ball on $\X$ always includes a truncated cone with a fixed angle. The volume of this
cone is then proportional to that of the ball. See the proof of the theorem for details. In 
Figure~\ref{fig:illust_assumption_counter}, we see that when the boundary has a spike, it is no more possible to 
include such a cone in the ball and Condition \ref{cond:reg2} may be not satisfied. 

In the proof of Theorem \ref{geometry_cases}, we also show that the same assumptions entail condition 
\ref{cond:regA} but using a different argument. When a point $x\in\X$ is close to the boundary denoted by 
$\partial \X$, its distance to the boundary can be lower bounded by the distance to the boundary along one of the 
$d$ axis of $\R^d$ with a finite number of possible boundary points, i.e. $$\inf_{z\in\partial\X}\norm{x-z}\geq
\eta\min_{1\leq i\leq d}\min_{s=1,\dots,N}\va{x_i-g_s((x_j)_{j\neq i})},$$
for some real-valued mappings $g_1,\dots,g_N$ and a positive constant $\eta$. This is sufficient to upper-bound 
the integral and check Condition \ref{cond:regA}. However, there exist some pathological examples with infinitely 
many connected components, such as the rings example given in Proposition \ref{assumption_counter}, for which the 
previous argument cannot be used and Condition \ref{cond:regA} does not hold. 

Overall, for practical applications, both conditions seem to be satisfied under some sufficiently general 
regularity assumptions on the domain $\X$. 

Finally, we prove that condition \ref{cond:regA} is equivalent to a constraint on tube volumes for the boundary of
the domain $\X$.

\begin{prop}\label{equivalence}
Suppose that Assumption \ref{cond:reg1} holds true and set for $\epsilon>0$, 
$$A_\epsilon:=\left\{x\in\X:\delta(x,\X^c)\leq\epsilon\right\}.$$
Then condition \ref{cond:regA} is equivalent to $\varlimsup_{\epsilon\searrow 0}Q(A_\epsilon)/\epsilon<\infty$.
\end{prop}
From this result, one can deduce that condition \ref{cond:regA} is valid for any probability measure $Q$ with a
bounded density if and only if $\lambda_d\left(A_\epsilon\right)=O(\epsilon)$ where $\lambda_d$ denotes the Lebesgue 
measure on $\R^d$. Let us mention that there is a link between this latter condition and the so-called tube formula
of \citet{weyl1939volume} who expressed for a smooth submanifold $\partial\X=\X\setminus\mathring\X$ the volume of a 
tube of radius $\epsilon$ around $\partial\X$ as a polynomial in $\epsilon$. For a submanifold of dimension $d-1$,
the first term of this expression is proportional to $\epsilon$. The equivalence given in Proposition 
\ref{equivalence} can then be used to get another proof for checking \ref{cond:regA} in the submanifold case, a 
result already given in Theorem \ref{geometry_cases}.

\section{Extension to local polynomials}\label{local}

In this section, we discuss a more general estimator introduced by \citet{holzmann2024multivariate} to reduce the
bias when the regression function $g$ has more regularity. They use local polynomials of order $L\in\N$ to
approximate the regression function. In their theorem, \citet{holzmann2024multivariate} suppose that the supports
are convex and that the target support is inside the interior of the source support. We will show that in fact,
Condition \ref{cond:reg2} is all we need. \\ 
We set $\N_{\leq L}^d$ the set of $d$-tuples of non-negative integers whose sum does not exceed $L$. Let 
$\Psi\colon\ieu[K^*]\to\N_{\leq L}^d$ be a one-to-one and onto function with $\Psi(1)=(0,\dots,0)$, where $K^*$ is
the cardinality of $\N_{\le L}^d$. For $j=1,\dots,K^*$, we also write $\zeta_j(x,z):=(z-x)^{\Psi(j)}:=\prod_{i=1}^d
(z_i-x_i)^{(\Psi(j))_i}$, which is a local monomial in $z$ in $d$ coordinates, of degree at most $L$, and centered 
at point $x$. Finally, $\zeta(x,z)=(\zeta_j(x,z))_{1\le j\le K^*}$ will denote the vector in $\R^{K^*}$ composed 
of all such local monomials.

The local estimator $\hat g(x)$ of the regression function is the first coordinate of the least square estimator
$$\hat\gamma:=\underset{\gamma\in\R^{K^*}}\argmin\sum_{i=1}^n\left(h(X_i,Y_i)-\gamma^\top\zeta(x,X_i)\right)^2
\1{\norm{X_i-x}\le\ta(x)}.$$
In other words, $\hat g(x)=e_1^\top M(x)^{-1}\sum_{i=1}^nh(X_i,Y_i)\zeta(x,X_i)\,\1{\norm{X_i-x}\le\ta(x)})$,
where $e_1$ is the the first vector in the canonical basis of $\R^{K^*}$ and $M(x)$ is the matrix 
\begin{align*}
M(x) & :=\sum_{i=1}^n\zeta(x,X_i)\zeta(x,X_i)^\top\1{\norm{X_i-x}\le\ta(x)} \\ & =\left(\sum_{i=1}^n(X_i-x)^{
\Psi(j)+\Psi(j')}\,\1{\norm{X_i-x}\le\ta(x)}\right)_{1\le j,j'\le K^*}.    
\end{align*}

The final estimator is then $\hat e_L(h):=\frac 1m\sum_{j=1}^m\hat g(X_j^*)$. Note that when $L=0,\,\zeta(x,z)$ is
the constant vector $1\in\R,\,M(x)=k$, and $\hat g(x)$ coincides with the classical estimator $\hat g_n(x)$ 
defined in section \ref{intro}. In what follows, we will need the constant $D:=\sum_{i=1}^Li\binom i{d+i-1}$.
Adapting the nice proof technique of \citet{holzmann2024multivariate}, we obtain the following result. 

\begin{thrm}\label{local_bias_control}
Suppose that Assumptions \ref{cond:reg1} to \ref{cond:reg3} hold true, that the source density $p$ is upper 
bounded by $\psup$, and that the function $g$ have derivatives of order $l\in\N$ which are all H\"older-continuous 
of order $\beta\in(0,1]$ with constant $\Lambda>0$. Suppose also that $k\geq(2D+1)K^*+1$, then setting $L:=l$, 
there exists a constant $C>0$ depending on $d,l$, and $P$, such that the conditional bias $B_n:=\E[\hat e_L(h)-
e(h)\mid X]$ is controlled by $$\E[B_n^2]\leq C\Lambda^2\left(\frac kn\right)^{\frac{2(l+\beta)}d}.$$
\end{thrm}

The conditional variance of the estimator $\hat e_L(h)$ was already studied by \citet{holzmann2024multivariate},
Theorem 2, who showed it has the parametric rate $n^{-1}+m^{-1}$. We therefore obtain a parametric rate for
$\hat e_L(h)$ with fixed $k$ as soon as $l+\beta\geq d/2$.

\bigskip

Extending Theorem 2 from \citet{holzmann2024multivariate} to a situation where we did not assume any inclusion 
between the source and the target supports implies that their local polynomial estimator can also be used in the 
context of global average treatment effect (ATE) estimation. 
We keep the notation from Section \ref{ate} and we define, for $w\in\{0,1\}$ and $x\in\X$, \begin{align*}
\hat g_w(x) & :=e_1^\top\,\underset{\gamma\in\R^{K^*}}\argmin\sum_{W_i=w}\left(Y_i-\gamma^\top\zeta(x,X_i)\right)^2
\1{\norm{X_i-x}\le\hat\tau_{k,w}(x)},\text{ and} \\ \hat\mu & :=\frac 1N\sum_{i=1}^N(2W_i-1)(Y_i-\hat g_{
1-W_i}(X_i)).
\end{align*}
Define the conditional bias \begin{align*}
B & :=\frac 1N\sum_{i=1}^N(2W_i-1)(g_{1-W_i}(X_i)-\bar g_{1-W_i}(X_i)),\text{ where } \\ \bar g_w(x) & :=e_1^\top
M_w(x)^{-1}\sum_{W_i=w}g_w(X_i)\,\zeta(x,X_i)\,\1{\norm{X_i-x}\leq\hat\tau_{k,w}(x)},
\end{align*}
then proceeding as in the proof of Theorem \ref{gate_bias}, the previous theorem leads to the following result.

\begin{coro}\label{local_ate}
Suppose that Conditions \ref{cond:T1} to \ref{cond:T5} hold true, that the regression functions $g_0$ and 
$g_1$ are restrictions of mappings on an open convex neighbourhood of $\X$ having derivatives of order $l\in\N$ 
which are all H\"older continuous of order $\beta\in(0,1]$. Suppose further that $k\geq(2D+1)K^*+1$, then
there exists a constant $C>0$ such that the conditional bias of the estimator $\hat\mu$ for the average 
treatment effect satisfies $$\E[B^2]\le C\left(\frac kN\right)^{2(l+\beta)/d}.$$
\end{coro}

As stated above, the conditional variance $E+V$ (as defined in Section \ref{ate}) of the estimator $\hat\mu$ was
shown by \citet{holzmann2024multivariate} to have the parametric rate $N^{-1}$. Along with the previous corollary, 
it proves that when $l+\beta\geq d/2$, the estimator $\hat\mu$ with fixed $k$ has the rate of convergence $N^{-1}$.

\section{Numerical experiments}\label{numerical}

In this section, we conduct numerical experiments to study how the estimators perform with different dimensions 
and sample sizes.

\paragraph{Methods that we compare}
We use the estimators~\eqref{first} and \eqref{second} with $k=1$ as well as the local polynomial fitting 
estimator, proposed by \citet{holzmann2024multivariate}, of order $1$. We refer to those methods as the 
1-Nearest-Neighbour Conditional Sampling Adaptation (\textbf{1NN-CSA}), the 1-Nearest-Neighbour Weighting 
(\textbf{1NN-W}), and the $k$-Nearest-Neighbour Polynomial fitting (\textbf{kNN-Poly}), respectively.
kNN-Poly of order $1$ amounts to a local linear fitting with a bias term, meaning that $d+1$ parameters are 
fitted. For this method, we test two variants with different $k$'s.
\begin{description}
\item[kNN-Poly-LB:] kNN-Poly with even smaller $k=2d^2+3d+3$. This is the minimum $k$ satisfying the condition of 
\citet[Theorems~1-2]{holzmann2024multivariate} for the local polynomial fitting of order $1$. LB stands for 
``Lower Bound''. We choose the minimum of such $k$'s because \citet{holzmann2024multivariate} reported that 
smaller $k$'s tend to show better results.
\item[kNN-Poly-d+5:] kNN-Poly with $k=d+5$.
This choice of $k$ is very close to $k=d+1$, the smallest $k$ for which the polynomial fitting is possibly well-posed.
kNN-Poly-d+5 includes a few more points in the fitting for better stability.
\item[NoCorrection]: A na\"{i}ve estimate without covariate shift adaptation, given by the empirical average over the source sample $\frac 1n\sum_{i=1}^nh(X_i,Y_i)$.
\item[OracleY]: 
A hypothetical estimate with the oracle access to the hidden labels $Y_j^*$, given by
the empirical average over the target sample $\frac 1m\sum_{j=1}^mh(X_j^*,Y_j^*)$.
\end{description}

\paragraph{Setups}
We conduct experiments under two setups described as follows.
For both setups, we let each of $(X^{(2)},\dots,X^{(d)})$ and $(X^{*(2)},\dots,X^{*(d)})$ be distributed uniformly 
over $[-1,1]^{d-1}$,
where $(\cdot)^{(j)}$ denotes the $j$th coordinate of the given vector.
$X^{(1)},\dots,X^{(d)}$, $X^{*(1)},\dots, X^{*(d)}$ are all independent.
We define the function $h$ by $h(x,y):=(x^{(1)}+y)^2$.

The differences between the setups are the conditional distribution of $Y$ given $X$ (equivalently, that of $Y^*$ 
given $X^*$) and the distributions of $X^{(1)}$ and $X^{*(1)}$, as summarized in Table~\ref{tab:experiment_setups}
and Figure~\ref{fig:trnorm_data_visualization} and as detailed below.
\begin{description}
\item[Setup TN0.5-Cubic:]
We let $X^{(1)}$ follow $\operatorname{TN}(-0.5,0.5,[-1,1])$ and $X^{*(1)}$ follow \\ $\operatorname{TN}(0.5,0.5,
[-1,1])$, where  $\operatorname{TN}(\mu,\s,S)$ denotes the truncated normal distribution for $(\mu,\s)\in \R
\times(0,\infty)$ and $S\subset\R$, defined as the distribution of $V\1{V\in S}$, $V$ being a normal variable with 
mean $\mu$ and standard deviation $\s$. The response is generated according to $Y=\vert{X^{(1)}}\vert^3+
\varepsilon$, where $\varepsilon$ a normal noise variable with mean $0$ and standard deviation $0.1$, independent 
of all the other variables.

\item[Setup TN0.5-Cubic-Reversed:]
The conditional distribution of $Y$ given $X$ is the same as in Setup TN0.5-Cubic.
However, we switch the locations of the source and the target covariate distributions, that is, $X^{(1)}$ follows
$\operatorname{TN}(0.5, 0.5,[-1,1])$ and $X^{*(1)}$ follows $\operatorname{TN}(-0.5,0.5, [-1,1])$.
\end{description}
Note that the supports of the target distribution is not included in the interior of the source distribution as in 
the literature~\citep{holzmann2024multivariate}, but Theorem~\ref{geometry_cases} immediately ensures that they 
satisfy our new conditions~\ref{cond:reg2} and \ref{cond:regA} as well as the other conditions for 
Theorems~\ref{better_bias} and \ref{local_bias_control}.
Note that under these setups, the function $g(x)=(x^{(1)}+\vert{x^{(1)}}\vert^3)^2 + 0.01$ is two-times 
continuously differentiable~\ref{cond:reg62}.

\begin{table}[bp]
\centering
\small
\caption{Setups of experiments.}
\label{tab:experiment_setups}
\begin{tabular}{lccl}
Name of setup & Source distribution of $[X]_1$ & Target distribution of $[X^*]_1$ & Type of response \\
\hline
       TN0.5-Cubic  & $\operatorname{TN}(-0.5, 0.5, [-1, 1])$ & $\operatorname{TN}(0.5, 0.5, [-1, 1])$ & $Y = \va{[X]_1}^3 + \varepsilon$  \\
       TN0.5-Cubic-Reversed  & $\operatorname{TN}(0.5, 0.5, [-1, 1])$ & $\operatorname{TN}(-0.5, 0.5, [-1, 1])$ & $Y = \va{[X]_1}^3 + \varepsilon$
    \end{tabular}
\end{table}

\begin{figure}
    \centering
    \begin{subfigure}[t]{0.43\linewidth}
        \centering
        \includegraphics[width=\linewidth]{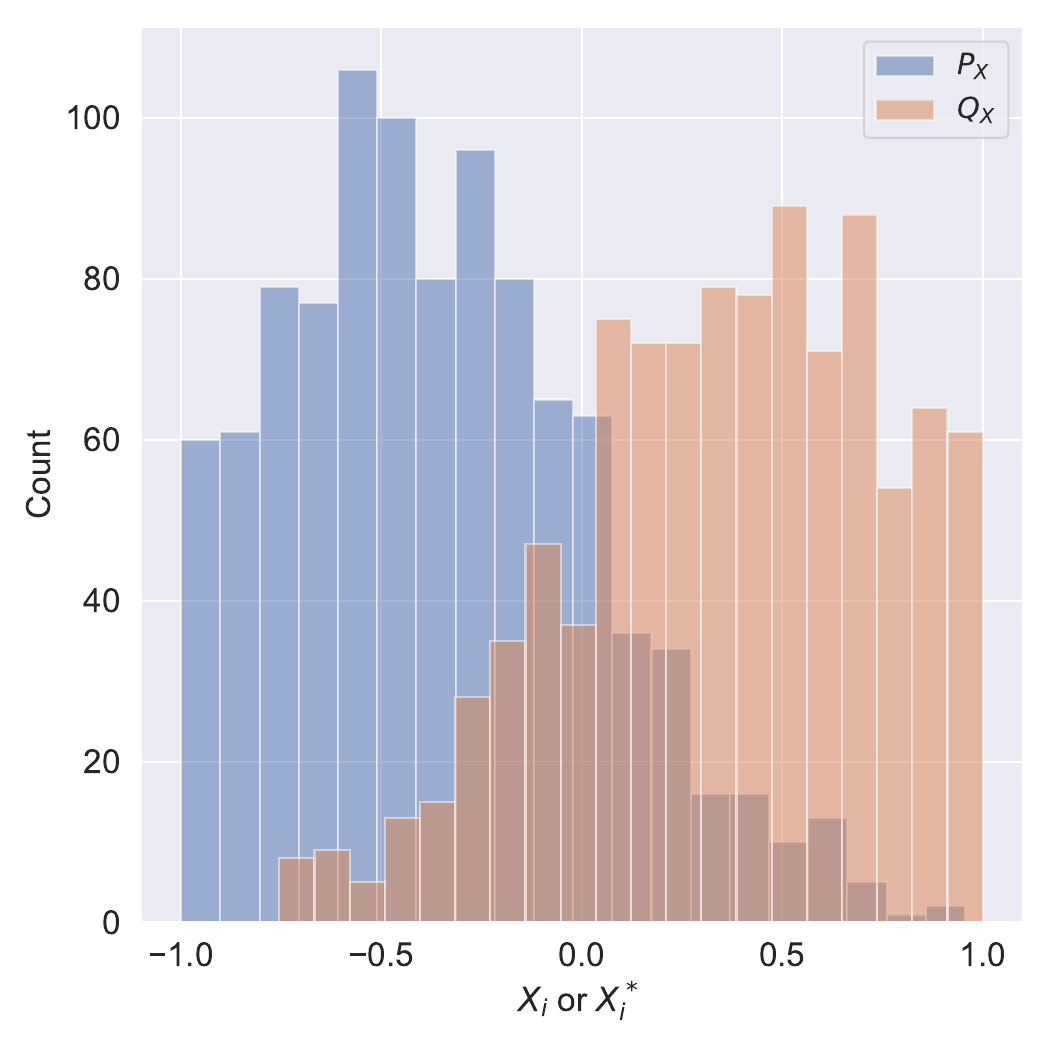}
        \caption{
        Setup TN0.5-Cubic.
        }
        \label{fig:trnorm0.5_data_visualization_x}
    \end{subfigure}
    \hfill
    \begin{subfigure}[t]{0.43\linewidth}
        \centering
        \includegraphics[width=\linewidth]{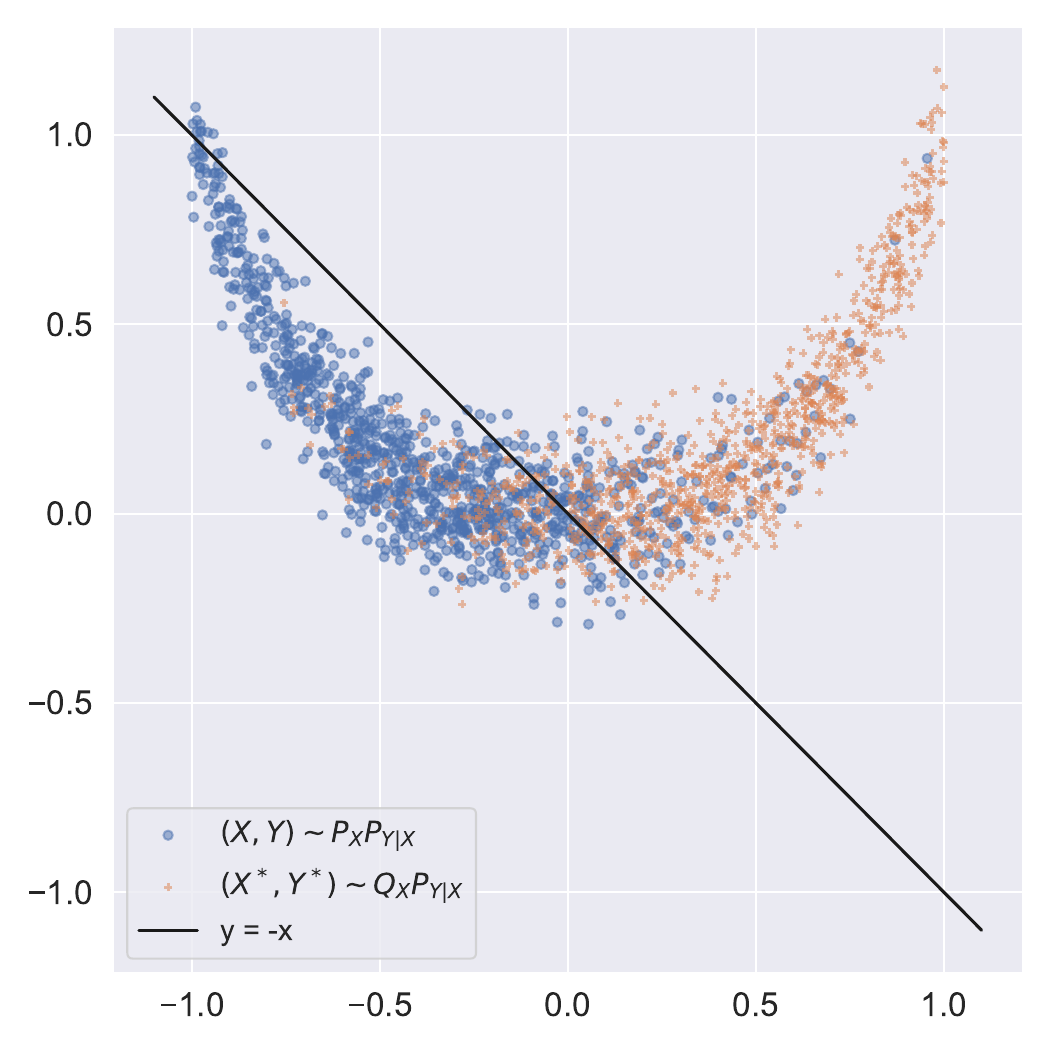}
        \caption{Setup TN0.5-Cubic.
        }
        \label{fig:trnorm0.5_data_visualization_xy}
    \end{subfigure}
    \begin{subfigure}[t]{0.43\linewidth}
        \centering
        \includegraphics[width=\linewidth]{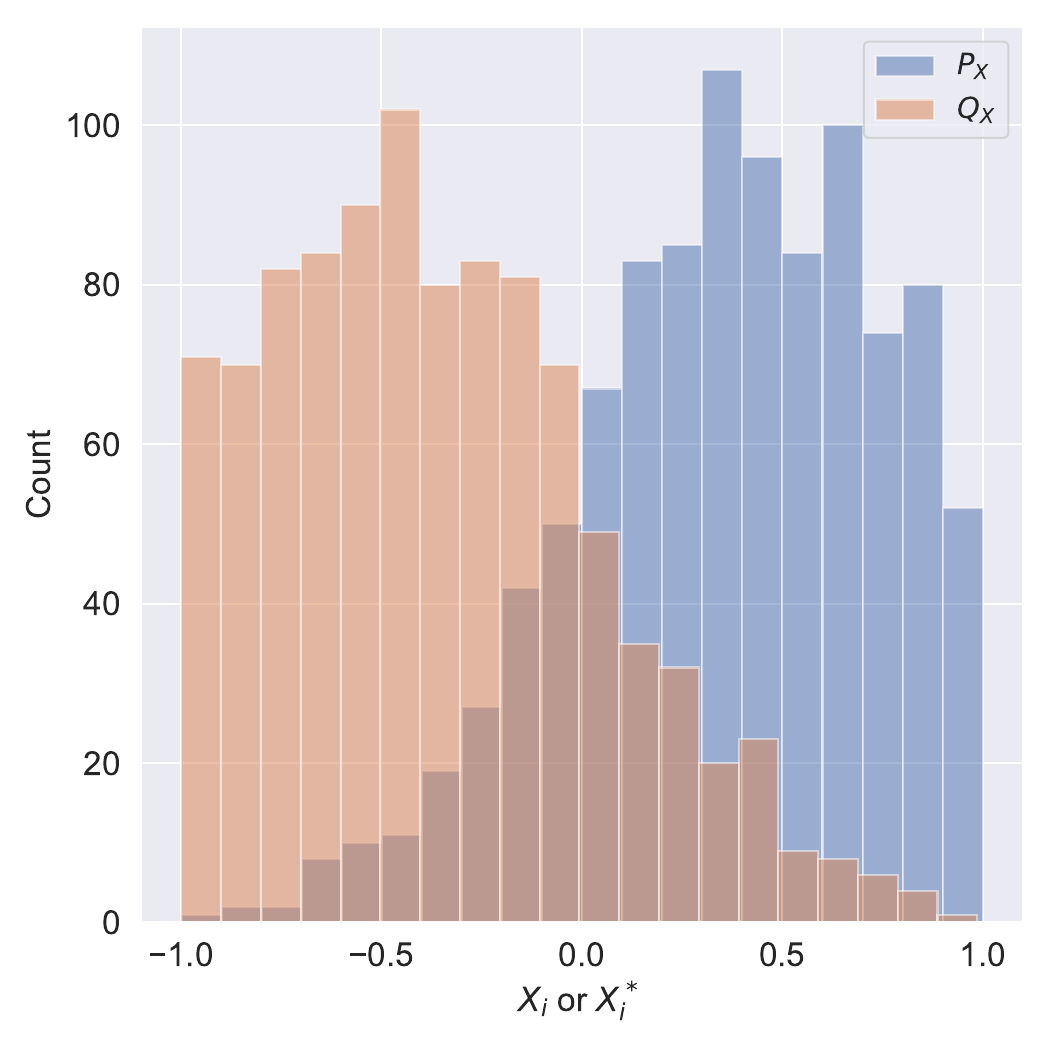}
        \caption{
        Setup TN0.5-Cubic-Reversed.
        }
        \label{fig:TN0.5_Cubed_Reversed_visualization_x}
    \end{subfigure}
    \hfill
    \begin{subfigure}[t]{0.43\linewidth}
        \centering
        \includegraphics[width=\linewidth]{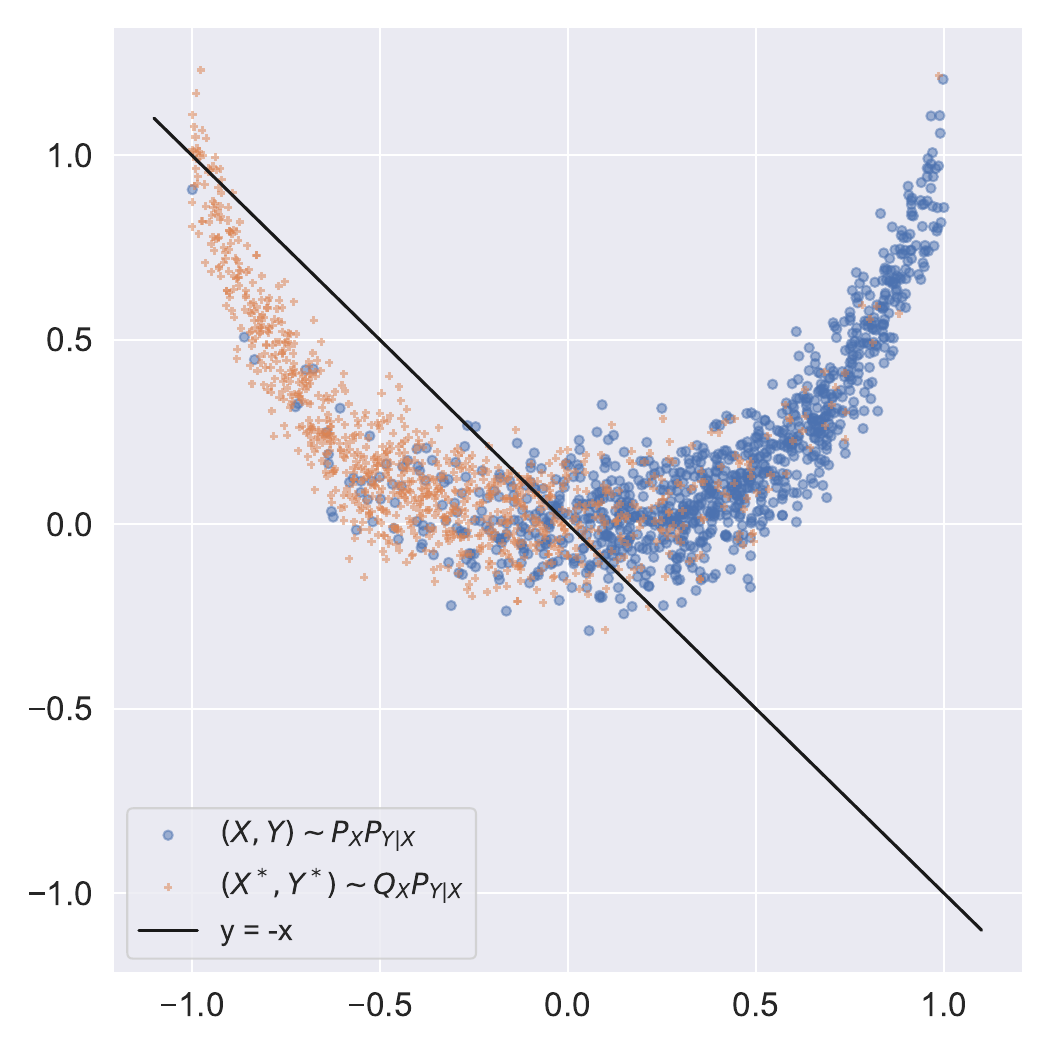}
        \caption{Setup TN0.5-Cubic-Reversed.}
        \label{fig:TN0.5_Cubed_Reversed_visualization_xy}
    \end{subfigure}
    \caption{Visualization of the data distributions used in the experiments.}
    \label{fig:trnorm_data_visualization}
\end{figure}

\begin{figure}
        \centering
        \begin{subfigure}{\linewidth}
            \centering
            \includegraphics[width=\linewidth]{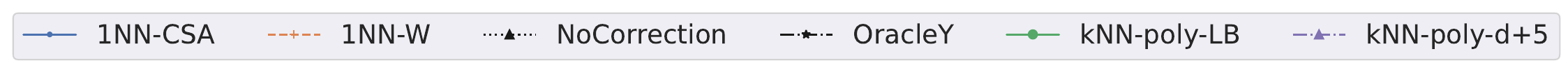}
        \end{subfigure}
        \begin{subfigure}{0.32\linewidth}
            \centering
            \includegraphics[width=\linewidth]{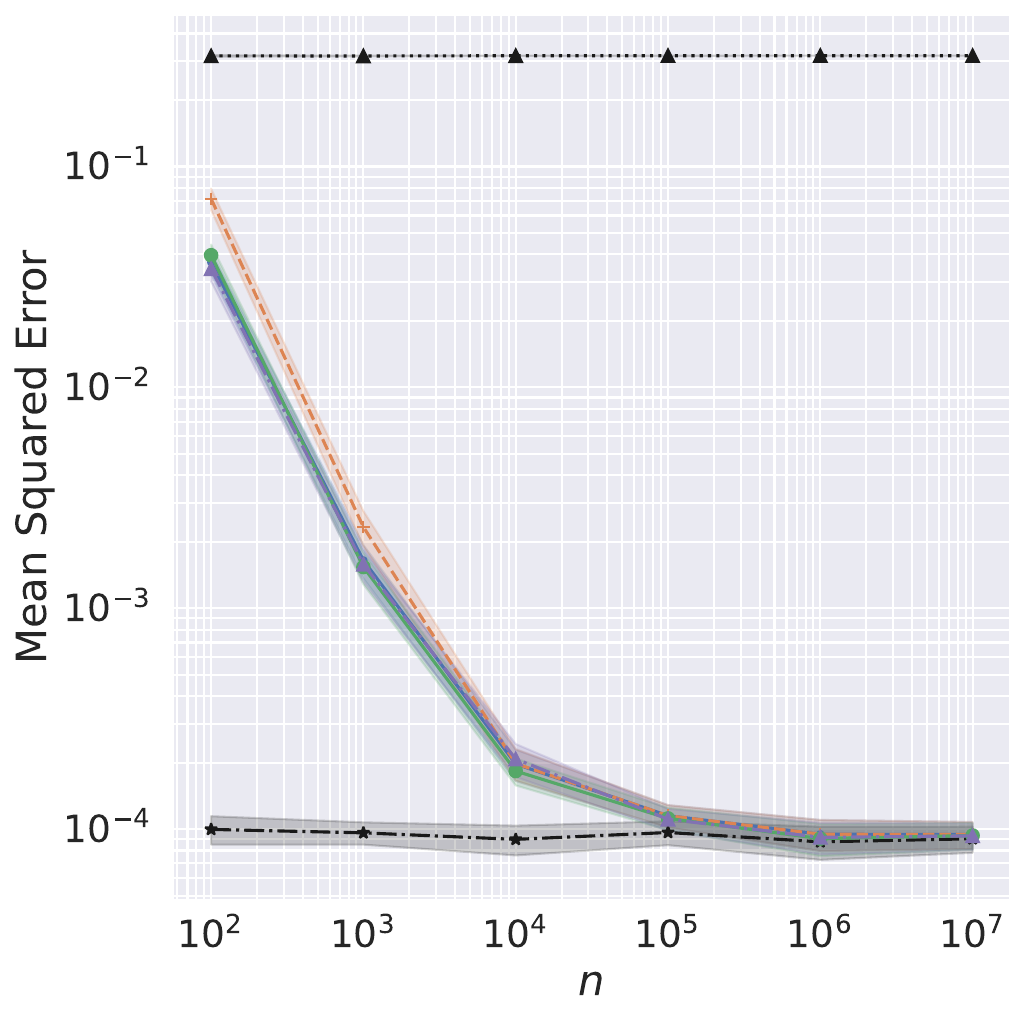}
            \caption{$d=1$}
            \label{fig:simu1_abs_cubed_trnorm_vs_trnorm_scale0.5_dim1}
        \end{subfigure}
        \hfill
        \begin{subfigure}{0.32\linewidth}
            \centering
            \includegraphics[width=\linewidth]{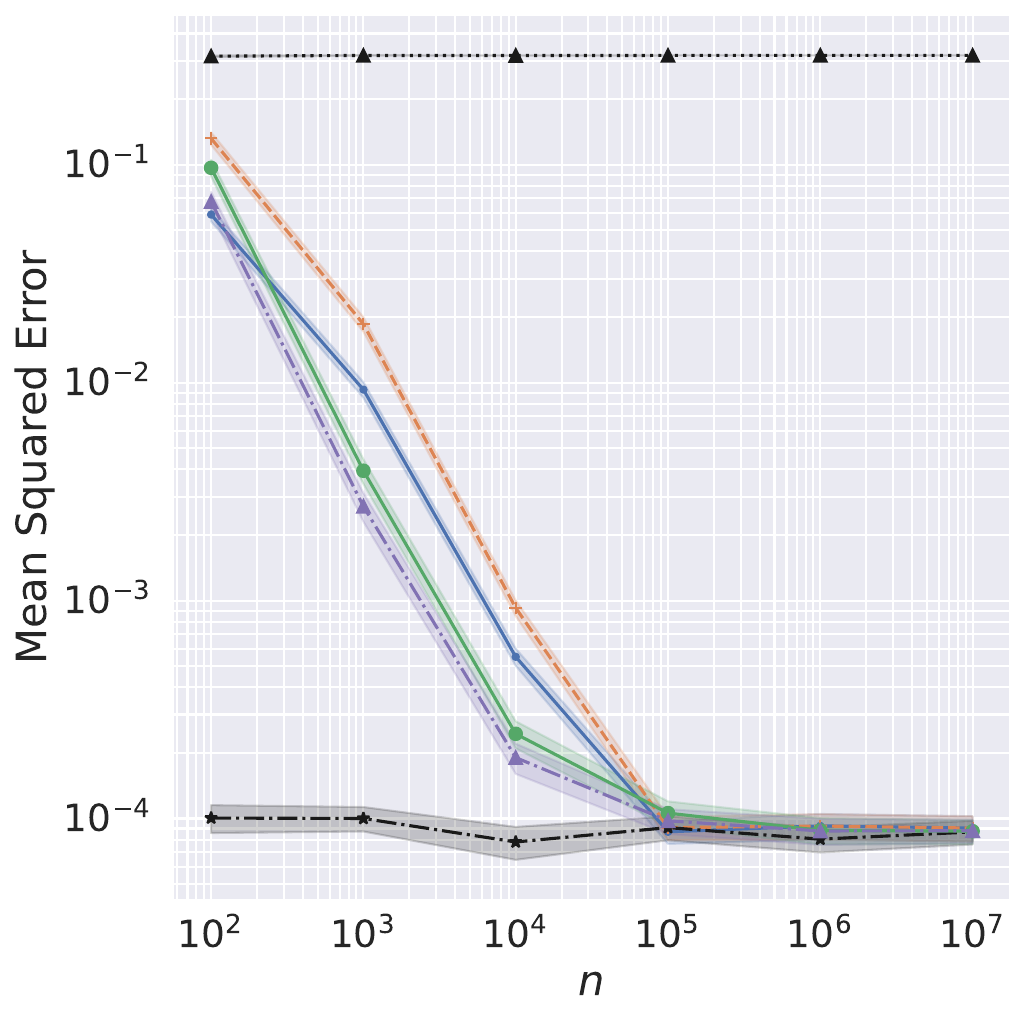}
            \caption{$d=2$}
            \label{fig:simu1_abs_cubed_trnorm_vs_trnorm_scale0.5_dim2}
        \end{subfigure}
        \hfill
        \begin{subfigure}{0.32\linewidth}
            \centering
            \includegraphics[width=\linewidth]{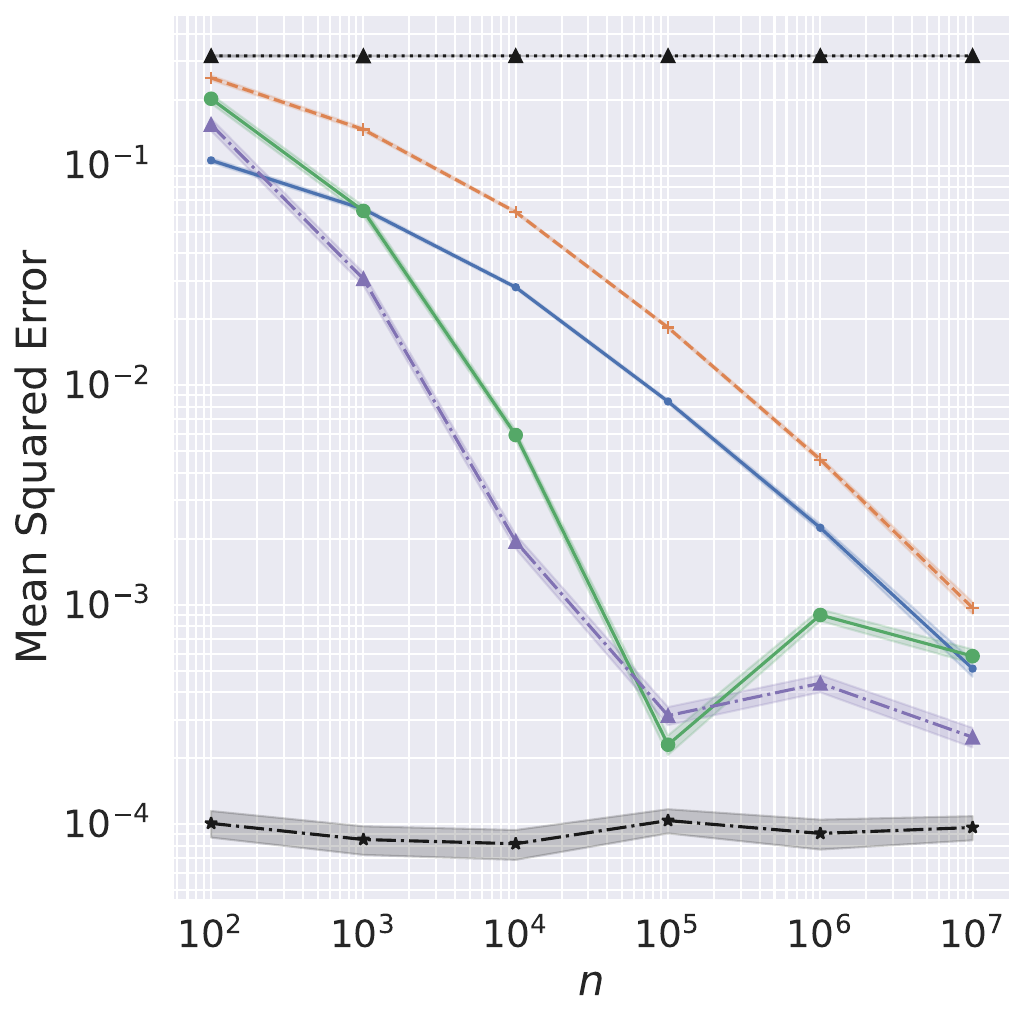}
            \caption{$d=5$}
            \label{fig:simu1_abs_cubed_trnorm_vs_trnorm_scale0.5_dim5}
        \end{subfigure}
        \caption{Results for Setup TN0.5-Cubic.}
        \label{fig:tn0.5cubic}
\end{figure}

\begin{figure}
        \centering
        \begin{subfigure}{\linewidth}
            \centering
            \includegraphics[width=\linewidth]{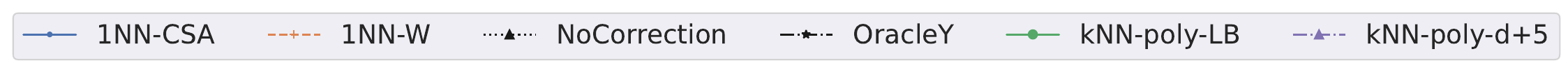}
        \end{subfigure}
        \begin{subfigure}{0.32\linewidth}
            \centering
            \includegraphics[width=\linewidth]{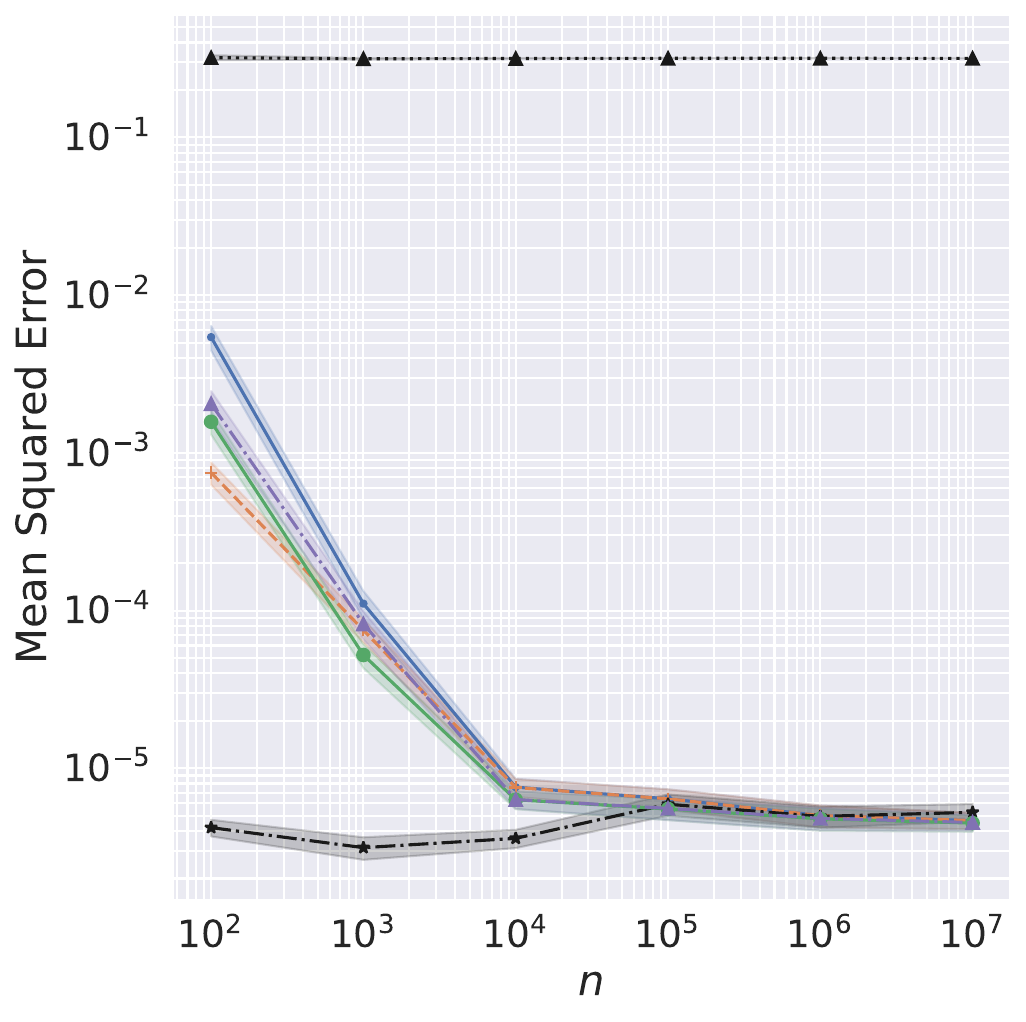}
            \caption{$d=1$}
            \label{fig:simu1_TN0.5_Cubic_Reversed_dim1}
        \end{subfigure}
        \hfill
        \begin{subfigure}{0.32\linewidth}
            \centering
            \includegraphics[width=\linewidth]{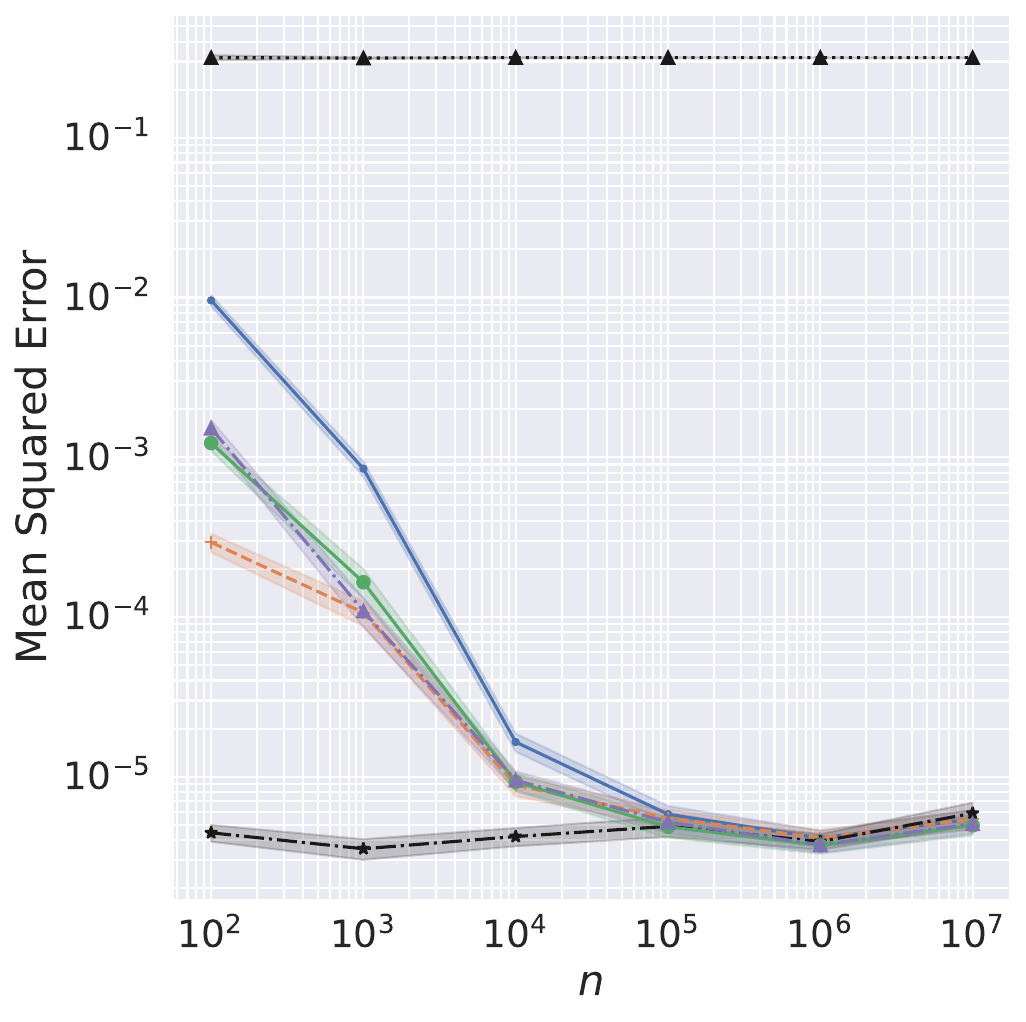}
            \caption{$d=2$}
            \label{fig:simu1_TN0.5_Cubic_Reversed_dim2}
        \end{subfigure}
        \hfill
        \begin{subfigure}{0.32\linewidth}
            \centering
            \includegraphics[width=\linewidth]{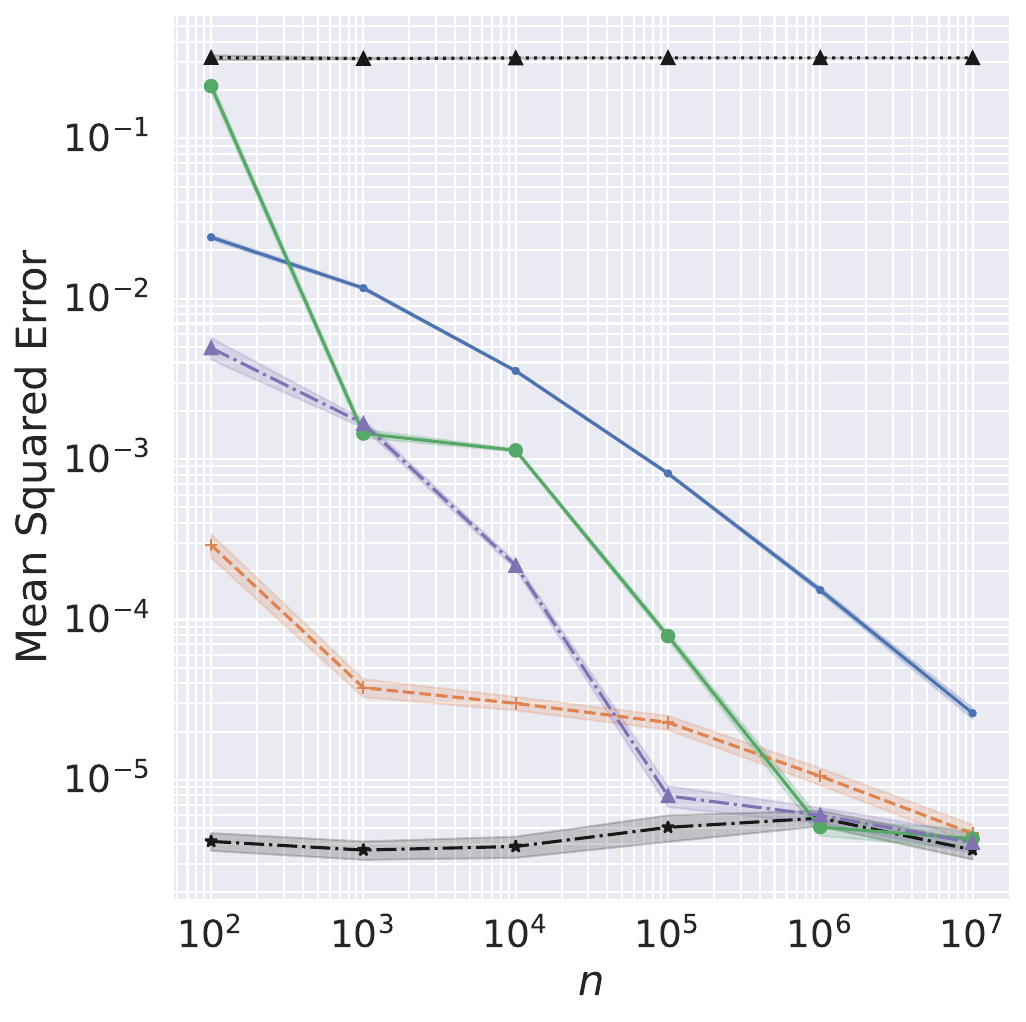}
            \caption{$d=5$}
            \label{fig:simu1_TN0.5_Cubic_Reversed_dim5}
        \end{subfigure}
        \caption{Results for Setup TN0.5-Cubic Reversed.}
        \label{fig:TN0.5_Cubic_Reversed}
\end{figure}

\paragraph{Results for the experiments}
The results for Setup TN0.5-Cubic are shown in Fig.~\ref{fig:tn0.5cubic}.
First, we can observe that NoCorrection, the estimator without covariate shift adaptation, performs poorly, and 
OracleY, the hypothetical estimator with the oracle access to the unobserved target labels, performs the best in almost all cases.
Other methods have performances between these two methods, and they tend to 
show smaller errors with larger sample sizes but increasing the covariate dimension slows down the rates of convergence, as the theory suggests.

For the kNN-poly methods of \citet{holzmann2024multivariate}, kNN-poly-d+5 often performs better than kNN-poly-LB.
However, $k$ being too small, there is no known theoretical guarantees for kNN-poly-d+5, unlike kNN-poly-LB~\citep[Theorems~1-2]{holzmann2024multivariate}.

When comparing 1NN-CSA, 1NN-W, and kNN-poly-d+5, we did not find a simple conclusion on which one would be the 
best recommended in practice. Note that the convergence rates for the three methods are the same under the assumption of second order differentiability.
Which method is preferable should only depend on some constants with for instance the covariate distribution which seems to have an impact.
For Setup TN0.5-Cubic, the kNN-poly methods tend to perform better than the others and 1NN-W is the worst (Figure~\ref{fig:tn0.5cubic}) while for Setup TN0.5-Cubic-Reversed, 1NN-W is often the best (Figure~\ref{fig:TN0.5_Cubic_Reversed}). Finally, let us mention that local polynomials of higher-order (with $L>1$) can have a better convergence rate but additional smoothness is required. In contrast, one cannot hope a better rate of convergence for 1NN-CSA or 1NN-W when additional smoothness is possible, as shown by Theorem \ref{bias_expansion}. 

\section{Conclusion}\label{conclusion}

We proposed a method for covariate shift using nearest-neighbour matching. We provided non-asymptotic error bounds 
and we applied our results to the study of average treatment effects. For future work, we could consider relaxing 
the boundedness assumptions on the support and on the distribution densities.

\section{Proofs}\label{proofs}

\subsection{Proof of Theorem \ref{better_bias}}\label{proof_better_bias}

\newcommand{\PhiOne}{\Phi^{(1)}_{\ell,\ell'}(x, y)}
\newcommand{\PhiTwo}{\Phi^{(2)}_{\ell,\ell'}(x, y)}
We prove the result when $i=2$, the case $i=1$ follows in the same way using uniform bounds, with respect to $x$, 
for the derivatives of $\Delta(x,\cdot)$.

Recall that $B_{2,n}=\E[\hat e_2(h)\mid X]-e(h)=\int_\X\frac 1k\sum_{\ell=1}^k(g(X_{\ik[\ell]})-g(x))\,
\mathrm dQ(x)$, we have a double integral \begin{align}
\E[B_{2,n}^2] & =\int_{\X^2}\frac 1{k^2}\sum_{\ell,\ell'=1}^k\E[(g(X_{\ik[\ell]})-g(x))(g(X_{\hat i_{\ell'}(y)})-
g(y))]\,\mathrm dQ^{\otimes 2}(x,y)\nonumber \\ & {}=\fbox{1-1}+\fbox{1-2},\label{eq:1+2}
\end{align}
where
\[
\fbox{1-1}:=k^{-2}\sum_{\ell,\ell'= 1}^k\int_{\X^2}\PhiOne\,\mathrm dQ^{\otimes 2}(x,y),\quad\fbox{1-2}:= k^{-2}
\sum_{\ell,\ell'=1}^k\int_{\X^2}\PhiTwo\,\mathrm dQ^{\otimes 2}(x,y),
\]
with \begin{align*}
\PhiOne & :=\E[(g(X_{\ik[\ell]})-g(x))(g(X_{\hat i_{\ell'}(y)})-g(y))\,\1{\tal+\hat\tau_{\ell'}(y)<\norm{y-x}}]
\text{ and } \\ \PhiTwo & :=\E[(g(X_{\ik[\ell]})-g(x))(g(X_{\hat i_{\ell'}(y)})-g(y))\,\1{\tal+\hat\tau_{\ell'}(y)
\ge\norm{y-x}}].
\end{align*}
It is sufficient to show that for appropriate constants $C_1$ and $C_2$, it holds
\begin{equation}
\fbox{1-1}\le C_1\left(\frac k{n + 1}\right)^{4/d}\text{\quad and \quad}
\fbox{1-2}\le C_2\left(\frac k{n + 1}\right)^{1+2/d}.\label{eq:bound1and2}
\end{equation}

Fix $(x,y)\in\X^2$ and $(\ell,\ell')\in\ieu[k]^2$.
If {$\tal+\hat\tau_{\ell'}(y) < \norm{y-x}$} holds, then the NN balls $B(x,\tal)$ and $B(y,\hat\tau_{\ell'}(y))$ do
not intersect and the actual positions of the observations $X_{\ik[\ell]}$ and $X_{\hat i_{\ell'}(y)}$ are 
conditionally independent given $\tal$ and $\hat\tau_{\ell'}(y)$.
Then the indicator $\1{\tal+\hat\tau_{\ell'}(y)<\norm{y-x}}$ ensures that there is negative correlation between
$\tal$ and $\hat\tau_{\ell'}(y)$.
Precisely, for $R_x,\,R_y\ge 0$ such that $R_x+R_y<\norm{y-x}$, Corollary \ref{ind_cond} gives
\begin{equation}
\E[(g(X_{\ik[\ell]})-g(x))(g(X_{\hat i_{\ell'}(y)})-g(y))\mid\tal=R_x,\hat\tau_{\ell'}(y)=R_y]=I(x,R_x)I(y,R_y),
\label{eq:gxgy_factorize}
\end{equation}
where $$I(a,R_a):=\frac 1{\int_{S^{d-1}}p(a+R_a\theta)\,\mathrm d\s(\theta)}\left(\int_{S^{d-1}}(g(a+R_a\theta)-
g(a))\,p(a+R_a\theta)\,\mathrm d\s(\theta)\right).$$
Applying Taylor's expansion on the function $g$, we get $\va{g(a+r\theta)-g(a)-r\nabla g(a)^\top\theta}\leq\frac{
\norm{\nabla^2g}_\infty}2r^2$. If $r\leq d(a,\X^c)$, then the density $p$ is $K$-Lipschitz continuous on $B(a,r)$ 
and we can write \begin{equation}
\va{\int_{S^{d-1}}r\theta\,p(a+r\theta)\,\mathrm d\s(\theta)}=r\va{\int_{S^{d-1}}(p(a+r\theta)-p(a))
\theta\,\mathrm d\s(\theta)}\leq K\s(S^{d-1})\,r^2.  \label{eq:gain_one_exponent}
\end{equation}

It shows that $\va{I(a,R_a)}\leq T(a,R_a):=(CR_a^2+\norm{\nabla g}_\infty R_a\,\1{R_a>d(a,\X^c)})$, where $C:=\norm{
\nabla g}_\infty K(c\pinf)^{-1}+\frac 12\norm{\nabla^2g}_\infty$.
It means that it remains to bound the quantity
\[
\E[T(x,\tal)\,T(y,\hat\tau_{\ell'}(y))\,\1{\tal+\hat\tau_{\ell'}(y)<\norm{y-x}}].
\]
Since $T(a,\cdot)$ is a non-decreasing function we can apply Corollary \ref{negative_correlation_NN}
on the negative correlation to obtain $$\E[T(x,\tal)\,T(y,\hat\tau_{\ell'}(y))\,\1{\tal+\hat\tau_{\ell'}(y)<
\norm{y-x}}]\leq\E[T(x,\tal)]\,\E[T(y,\hat\tau_{\ell'}(y))].$$
It means that \begin{align*}
\fbox{1-1} & =\int_{\X^2}\E[(g(X_{\ik[\ell]})-g(x))(g(X_{\hat i_{\ell'}(y)})-g(y))\;\1{\tal+\hat\tau_{\ell'}(y)<
\norm{y-x}}]\,\mathrm dQ^{\otimes 2}(x,y) \\ & \leq\left(\int_\X\E[C\tal^2+\norm{\nabla g}_\infty\tal\;
\1{\tal>\delta(x,\X^c)}]\,\mathrm dQ(x)\right)^2. 
\end{align*}
Using Lemmas \ref{moments_tau} and \ref{censored_tau}, and Assumption \ref{cond:regA}, we establish the bound
$$\fbox{1-1}\leq(C'+C'')^2\left(\frac k{n+1}\right)^{4/d},$$
where we have set $C':=\left(2\norm{\nabla g}_\infty K(c\pinf)^{-1}+\norm{\nabla^2g}_\infty\right)\Gamma(2+\floor{
2/d})(c\pinf\,\vab)^{-2/d}$ and $C'':=\frac{2^{1+3/d}\norm{\nabla g}_\infty}{(c\pinf\,\vab)^{2/d}}\;\underset{L>0}
\sup\left\{L^{1/d}\int_\X\exp(-L\,\delta(x,\X^c)^d)\;\mathrm dQ(x)\right\}$.

\bigskip

We are now reduced to analyse the case where the $\ell$-NN ball and the $\ell'$-NN ball intersect. More 
precisely, we set $$\PhiTwo:=\E[(g(X_{\ik[\ell]})-g(x))(g(X_{\hat i_{\ell'}(y)})-g(y))\;\1{\tal+\hat\tau_{
\ell'}(y)\geq\norm{y-x}}].$$ Using the mean value theorem, the inequality $ab\,\1{a+b\geq\delta}\leq a^2\1{a\geq
\delta/2}+b^2\1{b\geq\delta/2}$, and applying Lemma \ref{censored_tau} now yields \begin{align*}
\fbox{1-2} & =\int_{\X^2}\PhiTwo\,\mathrm dQ^{\otimes 2}(x,y) \\ & \leq\norm{\nabla g}_{\infty}^2\int_{\X^2}
\E\left[\tal\hat\tau_{\ell'}(y)\1{\tal+\hat\tau_{\ell'}(y)\geq\norm{y-x}}\right]\,\mathrm dQ^{\otimes 2}(x,y) \\ 
&\leq\norm{\nabla g}_{\infty}^2\int_{\X^2}\E\left[\tal^2\1{\tal\geq\norm{y-x}/2}+\hat\tau_{\ell'}(y)^2\1{\hat\tau_{
\ell'}(y)\geq\norm{y-x}/2}\right]\,\mathrm dQ^{\otimes 2}(x,y) \\ & \leq 2C\left(\frac k{n+1}\right)^{2/d}
\int_{\X^2}\exp\left(-L\,\frac{n+1}k\,\norm{y-x}^d\right)\,\mathrm dQ^{\otimes 2}(y,x) \\ & \leq 2C\qsup\,\left(
\frac k{n+1}\right)^{2/d}\s(S^{d-1})\int_{\R_+}\exp\left(-L\,\frac{n+1}k\,r^d\right)\,r^{d-1}\mathrm dr \\ & \leq
2C\qsup\,\left(\frac k{n+1}\right)^{2/d}\,\vab\int_{\R_+}\exp\left(-L\,\frac{n+1}k\,s\right)\,\mathrm ds \\ & 
\leq\frac{2C\qsup\,\vab}L\,\left(\frac k{n+1}\right)^{1+2/d},   
\end{align*}
where $C:=2e^{1/4}\norm{\nabla g}_\infty^2\Gamma(2+\floor{2/d})(c\pinf\,\vab)^{-2/d}$ and $L:=\frac{c\pinf\,
\vab}{2^{d+3}}$. In the series of inequalities above, we have used the equality $\sigma\left(S^{d-1}\right)=
d\,\vab$, and we also used the fact that for all $a\ge 0,\,\Pb(\tal\geq a)=\Pb(\tal>a)$ since the distribution 
$P$ is absolutely continuous with respect to the Lebesgue measure on $\R^d$.

\subsection{Proof of Theorem \ref{bias_expansion}}\label{proof_bias_expansion}

As for Theorem \ref{better_bias}, we only consider the case $i=2$. The case $i=1$ is similar, using uniform controls 
for the derivatives of $y\mapsto g_x(y):=\Delta(x,y)$ with respect to $x$ and $y$. 

\textbf{Proof of the expansion of $\E[B_{2,n}]$:}
The first-order bias term can be written as
$$\E[B_{2,n}]=\frac 1k\sum_{\ell=1}^k\int_\X\E\left[g(X_{\ik[\ell]})-g(x)\right]\,\mathrm dQ(x).$$
Fix $1\le\ell\le k$ and $x\in\X$ such that $x$ is in the support of $Q$. We consider the Taylor expansion of the
function $g\in C^{2+\beta}(\X)$. We have $$g(X_{\ik[\ell]})-g(x)=\nabla g(x)^\top(X_{\ik[\ell]}-x)+\frac 12
(X_{\ik[\ell]}-x)\nabla^2g(x)(X_{\ik[\ell]}-x)+O(\tal^{2+\beta}).$$
For the rest of the proof, we set $\varphi_x(z):=\nabla g(x)^\top z+\frac 12z^\top\nabla^2g(x)z$. From
Lemma~\ref{moments_tau}, we get 
\begin{equation}
\E[g(X_{\ik[\ell]})-g(x)]=\E[\varphi_x(X_{\ik[\ell]}-x)]+o((k/n)^{2/d}). \label{eq:FirstBiasAfterLemFour}
\end{equation}
Because we assumed that the support of $Q$ lies inside the interior of $\X$, there exists $\varepsilon>0$ such 
that for $Q$-almost all $x\in\X,\,B(x,2\varepsilon)\subset\X$. Then by Lemma \ref{censored_tau}, there exist 
positive constants $C,L>0$ such that $$\va{\E[\varphi_x(X_{\ik[\ell]}-x)\1{\tal>\varepsilon}]}\le C\exp\left(-L
\frac nk\right).$$ The previous bound holds independently from $x$ thanks to the continuity of $\nabla g$ and
$\nabla^2g$ on the compact set $\X$, and because $k=o(n^{1/2})$, this term is negligible compared to
$(k/n)^{2/d}$. We know from Lemma \ref{single_distribution} that the density of the random vector $X_{\ik[\ell]}-x$
is given by $$f_\ell(z):=np(x+z)\binom{n-1}{\ell-1}P(B(x,\norm z))^{\ell-1}P(\X\backslash B(x,\norm z))^{n-\ell}.$$
Then setting $F_x(r):=P(B(x,r))$, we have \begin{align*}
\fbox{2-1} & :=\E[\nabla g(x)^\top(X_{\ik[\ell]}-x)\1{\tal\leq\varepsilon}] \\ & =n\binom{n-1}{\ell-1}
\int_0^\varepsilon F_x(r)^{\ell-1}(1-F_x(r))^{n-\ell}\int_{S^{d-1}}\nabla g(x)^\top(r\theta)p(x+r\theta)\,
\mathrm d\s(\theta)\,r^{d-1}\,\mathrm dr \\ & =n\binom{n-1}{\ell-1}\int_0^\varepsilon F_x(r)^{\ell-1}(1-F_x(r))^{
n-\ell}r^d\int_{S^{d-1}}\nabla g(x)^\top\theta(p(x+r\theta)-p(x))\,\mathrm d\s(\theta)\,\mathrm dr.
\end{align*}
But since $B(x,\varepsilon)\subset\X$ and $p\in C^{1+\eta}(\X)$, we have $p(x+r\theta)=p(x)+r\nabla p(x)^\top
\theta+O(r^{1+\eta})$, so that \begin{align*}
\fbox{2-1} & =n\binom{n-1}{\ell-1}\int_0^\varepsilon F_x(r)^{\ell-1}(1-F_x(r))^{n-\ell}r^d\int_{S^{d-1}}(r\nabla 
g(x)^\top\theta\nabla p(x)^\top\theta+O(r^{1+\eta}))\,\mathrm d\s(\theta)\,\mathrm dr \\ & =n\binom{n-1}{\ell-1}
\int_0^\varepsilon F_x(r)^{\ell-1}(1-F_x(r))^{n-\ell}r^{d+1}\left(\int_{S^{d-1}}\theta^\top\nabla g(x)\nabla 
p(x)^\top\theta\,\mathrm d\s(\theta)+O(r^\eta)\right)\,\mathrm dr.
\end{align*}
It means that setting $C_{\ell,n}:=n\binom{n-1}{\ell-1}$ and $\Psi_2(x):=\frac 1{\s(S^{d-1})}\int_{S^{d-1}}
\theta^\top\left(\frac{\nabla g(x)^\top\nabla p(x)}{p(x)}+\frac{\nabla^2g(x)}2\right)\theta\,\mathrm d\s(\theta)$,
we have \begin{align*}
\E[\varphi_x(X_{\ik[\ell]}-x)\1{\tal\leq\varepsilon}]=C_{\ell,n}\,p(x)\s(S^{d-1})\int_0^\varepsilon(\Psi_2(x)+
O(r^\eta))F_x(r)^{\ell-1}(1-F_x(r))^{n-\ell}\,r^{d+1}\,\mathrm dr.
\end{align*}
Now, because $p$ is Lipschitz-continuous on $\X$, one can show that $F_x(r)=p(x)\,\vab\,r^d(1+O(r))$, and setting
$C_p:=\norm{\nabla p}_\infty/\pinf$, we obtain by the binomial expansion and $\binom{\ell-1}j\le\ell^j$,
\begin{equation}
\va{F_x(r)^{\ell-1}-(p(x)\,\vab\,r^d)^{\ell-1}}\leq(p(x)\,\vab\,r^d)^{\ell-1}\sum_{j=1}^{\ell-1}(C_p\ell r)^j.
\label{eq:Diff_F_and_rd}
\end{equation}
Similarly, we also have $F_x(r)\ge p(x)\,\vab\,r^d(1-C_pr)\ge p(x)\,\vab\,r^d(1-C_p\varepsilon)$.
Applying Lemma \ref{censored_tau} and using $k^{d+1}=o(n)$, we know that \begin{align*}
& C_{\ell,n}\,p(x)\s(S^{d-1})\int_{\varepsilon/k}^\varepsilon r^2F_x(r)^{\ell-1}(1-F_x(r))^{n-\ell}\,r^{d-1}\,
\mathrm dr=\E[\tal^2\1{\varepsilon/k\leq\tal\leq\varepsilon}]\\ &\le C\left(\frac kn\right)^{2/d}\exp\left(-C'\,
\frac nk\,(\varepsilon/k)^d\right)=o((k/n)^{2/d}),
\end{align*}
which means that the part of the integral outside the range $[0,\varepsilon/k]$ is negligible.
Using the inequality $1-t\leq\exp(-t)$ for $t\geq 0$, we get \begin{align*}
\fbox{2-2} & :=C_{\ell,n}\int_0^{\varepsilon/k}\va{F_x(r)^{\ell-1}-(p(x)\,\vab\,r^d)^{\ell-1}}(1-F_x(r))^{n-\ell}\,
r^{d+1}\,\mathrm dr \\ & \le C_{\ell,n}\sum_{j=1}^{\ell-1}(C_p\ell)^j\int_0^{\varepsilon/k}(p(x)\,\vab\,r^d)^{
\ell-1}r^j\exp(-(n-\ell)F_x(r))\,r^{d+1}\,\mathrm dr \\ & \le C_{\ell,n}\sum_{j=1}^{\ell-1}(C_p\ell)^j\int_{\R_+}
(p(x)\,\vab\,r^d)^{\ell-1}r^{2+j}\exp(-L_{\ell,k,n}np(x)\,\vab\,r^d)\,r^{d-1}dr,
\end{align*}
where $L_{\ell,k,n}:=\frac{n-\ell}n(1-C_p\varepsilon/k)>0$. We can redefine the constant $\varepsilon$ small enough
to ensure that $1/4\ge C_p\varepsilon$. Furthermore, $\ell=o(n)$, so that $L_{\ell,k,n}\ge L_k:=1-(2k)^{-1}$ for all
sufficiently large $n$, and thus \begin{align*}
\fbox{2-2}\le C_{\ell,n}\sum_{j=1}^{\ell-1}(C_p\ell)^j\int_{\R_+}(p(x)\,\vab\,r^d)^{\ell-1}r^{2+j}\exp(-nL_kp(x)\,
\vab\,r^d)\,r^{d-1}dr,
\end{align*}
to which we can apply Lemma~\ref{lem:negligible_integral_simple_simplified} to obtain \begin{align*}
\fbox{2-2} & =C_{\ell,n}\sum_{j=1}^{\ell-1}(C_p\ell)^j\,o((k/n)^{(2+j)/d})=o\left(\left(\frac kn\right)^{2/d}\;
\sum_{j=1}^\infty\left(\frac{\ell^{d+1}}n\right)^{j/d}\right) \\ & =o\left(\left(\frac kn\right)^{2/d}\left(
\frac{\ell^{d+1}}n\right)^{1/d}\left(1-\lr(){\frac{\ell^{d+1}}n}^{1/d}\right)^{-1}\right)=o((k/n)^{2/d}).
\end{align*}
Similarly, through Lemma~\ref{lem:negligible_integral_simple_simplified}, we have
$$C_{\ell,n}\int_0^{\varepsilon/k}r^\eta F_x(r)^{\ell-1}(1-F_x(r))^{n-\ell}\,r^{d+1}\,\mathrm dr= O((k/n)^{\frac{2+
\eta}d})=o((k/n)^{2/d}).$$
This far, we have shown that $\E[g(X_{\ik[\ell]})-g(x)]$ equals \begin{equation*}
C_{\ell,n}\,p(x)\s(S^{d-1})\Psi_2(x)\int_0^{\varepsilon/k}(p(x)\,\vab\,r^d)^{\ell-1}(1-F_x(r))^{n-\ell}\,r^{d+1}\,
\mathrm dr+o((k/n)^{2/d}).    
\end{equation*}
Next, we claim \begin{align}
& C_{\ell,n}\int_0^{\varepsilon/k}(p(x)\,\vab\,r^d)^{\ell-1}\,(1-F_x(r))^{n-\ell}\,r^{d+1}\,\mathrm dr\nonumber \\
& =C_{\ell,n}\int_0^{\varepsilon/k}(p(x)\,\vab\,r^d)^{\ell-1}\,(1-F_x(r))^n\,r^{d+1}\,\mathrm dr+o((k/n)^{2/d}).
\label{eq:lead1}
\end{align}
To see this, set $r_1:=r,\,r_2:=0$, and $\ell':=0$ in Lemma~\ref{lem:CanChangeExponentOfOneMinusFExtendedCorrected}
to obtain \begin{equation*}
\va{(1-F_x(r))^{n-\ell}-(1-F_x(r))^n}=O\mtlr(){\ell\exp(-nL_kp(x)\,\vab\,r^d)\,r^d},
\end{equation*}
for $r\in[0,\epsilon/k]$, and apply Lemma~\ref{lem:negligible_integral_simple_simplified}:
\[
C_{\ell,n}\int_0^{\varepsilon/k}(p(x)\,\vab\,r^d)^{\ell-1}\ell{\exp(-nL_kp(x)\,\vab\,r^d)r^{2d+1}}\,\mathrm dr
=O(k^{1+(2+d)/d}n^{-(2+d)/d}).
\]
The order of this quantity is $o((k/n)^{2/d})$ by the assumption $k^{\gamma} = o(n)$ with $\gamma > d + 1$.
Moreover, setting $r_2:=0$ in Lemma~\ref{lem:PolyOneMinusFNearExpLnExtended} and then applying 
Lemma~\ref{lem:negligible_integral_simple_simplified} yields \begin{align*}
& \bigg\vert C_{\ell,n}\int_0^{\varepsilon/k}(p(x)\,\vab\,r^d)^{\ell-1}\,(1-F_x(r))^n\,r^{d+1}\,\mathrm dr \\
& \phantom{{\bigg\vert{}}}-C_{\ell,n}\int_0^{\varepsilon/k}(p(x)\,\vab\,r^d)^{\ell-1}\,\exp(-np(x)\,\vab\,r^d)
r^{d+1}\,\mathrm dr\bigg\vert \\ & \le 2np(x)\vert V^d\vert C_{\ell,n}\int_0^{\varepsilon/k} (p(x)\,\vab\,r^d)^{
\ell-1}\,O\mtlr(){r^{d+1}\,\exp(-np(x)\,\vab\,r^d\,L_k)r^{d+1}}\,\mathrm dr \\ & =O(n(k/n)^{(2+d+1)/d})=
o((k/n)^{2/d}),
\end{align*}
where we used $k^{d+1}=o(n)$.
The leading term of the bias is now reduced to the integral \begin{equation}
\fbox{2-3}:=C_{\ell,n}\,p(x)\s(S^{d-1})\Psi_2(x)\int_0^{\varepsilon/k}(p(x)\,\vab\,r^d)^{\ell-1}\exp(-np(x)\,\vab\,
r^d)\,r^{d+1}\,\mathrm dr.\label{eq:lead5}
\end{equation}
To extend the range of the integral to $\R_+$, we show that the following part is negligible:
$$\fbox{2-4}:=C_{\ell,n}p(x)\s(S^{d-1})\Psi_2(x)\int_{\varepsilon/k}^\infty(p(x)\,\vab\,r^d)^{\ell-1}\exp(-np(x)
\,\vab\,r^d)r^{d+1}\,\mathrm dr.$$
We make the change of variables $s:=np(x)\,\vab\,r^d$, which yields \begin{align*}
\fbox{2-4} & =C_{\ell,n}n^{-\ell}\Psi_2(x)(p(x)\,\vab\,n)^{-2/d}\int_{np(x)\vab(\varepsilon/k)^d}^\infty 
s^{\ell+2/d-1}e^{-s}\,\mathrm ds \\ & \le\frac C{\Gamma(\ell)}\Psi_2(x)(p(x)\,\vab\,n)^{-2/d}\exp\left(
-\frac n{2k^d}p(x)\,\vab\,\varepsilon^d\right)\int_{\R_+}s^{\ell+2/d-1}e^{-s/2}\,\mathrm ds \\ & \le C\Psi_2(x)
\left(\frac 2{p(x)\,\vab}\right)^{2/d}n^{-2/d}\,\frac{\Gamma(\ell+2/d)}{\Gamma(\ell)}\;2^\ell\,\exp\left(-\frac n{
2k^d}p(x)\,\vab\,\varepsilon^d\right).
\end{align*}
Above we use the identities $\s(S^{d-1})=d\,\vab$ and $C_{\ell,n}n^{-\ell}=n^{1-\ell}\binom{n-1}{\ell-1}=\frac 1{
\Gamma(\ell)}(1+o(1))$.
However, because $$2^\ell\exp\left(-\frac n{2k^d}p(x)\,\vab\,\varepsilon^d\right)=\exp\left(\ell\ln 2-\frac n{2k^d}
p(x)\,\vab\,\varepsilon^d\right)=o(1),$$ using $k^{d+1}=o(n)$, we find that $\fbox{2-4}=o((k/n)^{2/d})$.

\bigskip

This means that we can extend the integral in Eq.~\eqref{eq:lead5} over $\R_+$.
Now again the change of variables $s:=np(x)\,\vab\,r^d$ yields \begin{align*}
\fbox{2-3} & =C_{\ell,n}p(x)\sigma(S^{d-1})\Psi_2(x)\times\int_{\R_+}(p(x)\,\vab\,r^d)^{\ell-1}r^2\exp(-np(x)\,\vab
\,r^d)\,r^{d-1}\mathrm dr \\ & =C_{\ell,n}p(x)\sigma(S^{d-1})\Psi_2(x)\times d^{-1}n^{-\ell-2/d}\,(p(x)\,\vab)^{
-2/d-1}\,\Gamma(\ell+2/d).
\end{align*}
Finally, we showed that \begin{equation*}
\E\left[g(X_{\ik[\ell]})-g(x)\right]=n^{-2/d}\frac{\Gamma(\ell+2/d)}{\Gamma(\ell)}(p(x)\,\vab)^{-2/d}
\Psi_2(x)+o((k/n)^{2/d}),$$ and therefore, $$\E[B_{2,n}]=\frac{n^{-2/d}}k\sum_{\ell=1}^k\frac{\Gamma(\ell+2/d)}{
\Gamma(\ell)}\left(\int_\X(p(x)\,\vab)^{-2/d}\,\Psi_2(x)\,\mathrm dQ(x)\right)+o((k/n)^{2/d}).
\end{equation*}

\bigskip

\textbf{Proof of the expansion of $\E[B_{2,n}^2]$:}

It is sufficient to show $\E[B^2_{2,n}]-\E[B_{2,n}]^2=o((k/n)^{4/d})$, because in the first part of the proof, we 
have shown that $\E[B_{2,n}]^2$ has the same leading term. However, setting for $\varepsilon>0$ to be fixed later,
$$B_{2,n,\varepsilon}:=k^{-1}\sum_{\ell=1}^k\int_\X\left\{g\left(X_{\hat i_\ell(x)}\right)-g(x)\right\}\1{
\tal\leq\varepsilon/k}\,\mathrm dQ(x),$$
we first note that 
$$\va{B_{2,n}-B_{2,n\varepsilon}}\leq K'\int_\X\hat\tau_k(x)\1{\hat\tau_k(x)>\varepsilon/k}\,\mathrm dQ(x),$$
where $K'$ is the Lipschitz constant of $g$. Using Lemma \ref{censored_tau}, it is not difficult to show that
$$\E[B^2_{2,n}]-\E[B_{2,n}]^2=\E[B^2_{2,n,\varepsilon}]-\E[B_{2,n,\varepsilon}]^2+o\left(\left(\frac kn\right)^{
4/d}\right).$$
For simplicity of notations, we will denote $B_{2,n,\varepsilon}$ by $B_{2,n}$ in what follows. 
We can write the the leading term as \begin{equation*}
\E[B_{2,n}]^2=\frac 1{k^2}\sum_{(\ell,\ell')\in\ieu[k]^2}\int_{\X^2}B_\ell(x)B_{\ell'}(y)\,\mathrm dQ^{\otimes 2}
(x,y),
\end{equation*}
where
\[
B_\ell(a):=\E\lr[]{\lr(){g(X_{\ik[\ell]})-g(a)}\,\1{\hat\tau_\ell(a)\le\varepsilon/k}}.
\]

On the other hand, from Eqs.~\eqref{eq:1+2} and \eqref{eq:bound1and2} of the proof of Theorem~\ref{better_bias} in 
Appendix~\ref{proof_better_bias}, because we work in dimension $d\geq 3$, the leading term in the second-order bias 
is
\[
\E[B_{2,n}^2]=\frac 1{k^2}\sum_{\ell,\ell'=1}^k\int_{\X^2}B_{\ell,\ell'}(x,y)\,\mathrm dQ^{\otimes 2}(x,y)+o((k/n)^{
4/d}),
\]
where
\[
B_{\ell,\ell'}(x,y):=\E[(g(X_{\ik[\ell]})-g(x))(g(X_{\hat i_{\ell'}(y)})-g(y))\;\1{\tal+\hat\tau_{\ell'}(y)<
\norm{y-x}}].
\]
Using a similar argument as before, we can replace $B_{\ell,\ell'}(x,y)$ by the same expectation but restricted to
the event $\left\{\tal\leq\varepsilon/k\right\}\cap\left\{\hat\tau_{\ell'}(y)\leq\varepsilon/k\right\}$ and for 
simplicity of notation, we will still denote by $B_{\ell,\ell'}(x,y)$ this restricted version.
All we need to do is to compare $B_\ell(x)B_{\ell'}(y)$ and $B_{\ell,\ell'}(x,y)$.

Because the support of $Q$ is included by the interior of $\X$ and $\X$ is compact, we can find $\varepsilon>0$ such 
that for all $x$ on the support of $Q$, $B(x,\varepsilon)\subset\X$.
Then, for any point $a$ on the support of $Q$ and any $R_a\leq\varepsilon$, we can express the following expectation 
using the polar coordinates:
\begin{equation}
\E[(g(X_{\hat i_\ell(a)})-g(a))\mid\tal=R_a]=I(a,R_a),\label{eq:EgaIa}
\end{equation}
where $$I(a,R_a):=\frac 1{\int_{S^{d-1}}p(a+R_a\theta)\,\mathrm d\s(\theta)}\left(\int_{S^{d-1}}
(g(a+R_a\theta)-g(a))\,p(a+R_a\theta)\,\mathrm d\s(\theta)\right).$$
Moreover, from Eq.~\eqref{eq:gxgy_factorize} of Proof of Theorem~\ref{better_bias}, for any pair of points $(x,y)$ 
on the support of $Q$ and $R_x+R_y<\norm{x-y}$ and $R_x,R_y\leq\varepsilon$, we have \begin{equation}
\E[(g(X_{\ik[\ell]})-g(x))(g(X_{\hat i_{\ell'}(y)})-g(y))\mid\tal=R_x,\hat\tau_{\ell'}(y)=R_y]=I(x,R_x)I(y,R_y),
\label{eq:IxIy}
\end{equation}
The $\beta$-H\"older continuity of $\nabla^2g$ and the $\eta$-H\"older continuity of $\nabla p$ imply that
\begin{align*}
g(a+R_a\theta)-g(a) & =\nabla g(a)^\top z+z^\top\nabla^2g(a)z+O\mtlr(){R_a^{2+\beta}},\text{ and } \\
p(a+R_a\theta)-p(a) & =\nabla p(a)^\top\theta R_a+O\mtlr(){R_a^{1+\eta}}.
\end{align*}
Thus, the integral in the definition of $I(a, R_a)$ can be expressed as
\begin{align*}
\fbox{2-5} & :=\int_{S^{d-1}}(g(a+R_a\theta)-g(a))p(a+R_a\theta)\,\mathrm d\s(\theta) \\ & =\int_{S^{d-1}}\lr(){
\nabla g(a)^\top\theta R_a+\frac 12\theta^\top\nabla^2g(a)\theta R_a^2+O\mtlr(){R_a^{2+\beta}}} \\ & \phantom{={} 
\int_{S^{d-1}}\bigg(}\times(p(a)+\nabla p(a)^\top\theta R_a+O\mtlr(){R_a^{1+\eta}})\,\mathrm d\s(\theta) \\
& =\int_{S^{d-1}}\lr(){\nabla g(a)^\top\theta\,\nabla p(x)^\top\theta R_a^2+\frac 12\theta^\top\nabla^2g(a)\theta
R_a^2 p(a)}\,\mathrm d\s(\theta)+O\mtlr(){R_a^{2+\kappa}} \\ & =p(a)\Psi_2(a)\s(S^{d-1})R_a^2+O\mtlr(){R_a^{
2+\kappa}},
\end{align*}
where $\kappa:=\min(\eta,\beta)$. We also used $\int_{S^{d-1}}\theta\,\mathrm d\s(\theta)=0$. Here,
\[
\Psi_2(a):=\frac 1{\s(S^{d-1})}\int_{S^{d-1}}\theta^\top\left(\frac{\nabla g(a)^\top\nabla p(a)}{p(a)}+\frac{\nabla^2
g(a)}2\right)\theta\,\mathrm d\s(\theta),
\]
as defined in the statement of the theorem. By the same expansion of $p$ and the fact that $\int_{S^{d-1}}\theta\,
\mathrm d\s(\theta)=0$, we can express the denominator of $I(a,R_a)$ as \begin{equation}
{\int_{S^{d-1}}p(a+R_a\theta)\mathrm d\s(\theta)}=p(a)\s(S^{d-1})+O(R_a^{1+\eta}).\label{eq:PaPlusIsSigma}
\end{equation}
Here, we choose $\varepsilon>0$ small enough in such a way the previous quantity is lower bounded by a positive 
constant, which is always possible since $p$ is lower bounded on $\X$. Plugging this result into
Eqs.~\eqref{eq:EgaIa} and \eqref{eq:IxIy}, we conclude \begin{align*}
& \E[(g(X_{\hat i_\ell(a)})-g(a))\mid\hat\tau_\ell(a)=R_a]=\Psi_2(a)R_a^2+O(R_a^{2+\kappa}),\text{ and } \\
& \E[(g(X_{\ik[\ell]})-g(x))(g(X_{\hat i_{\ell'}(y)})-g(y))\mid\tal=R_x,\hat\tau_{\ell'}(y)=R_y] \\ & =\Psi_2(x)
\Psi_2(y)R_x^2R_y^2+O(R_x^{2+\kappa}R_y^{2+\kappa}).
\end{align*}
As a consequence, using Lemma \ref{tail_tau}, we obtain \begin{align*}
B_{\ell}(x) B_{\ell'}(y) & =\Psi_2(x)\Psi_2(y)J_\ell^{(1)}(x)J_{\ell'}^{(1)}(y)+o\left((k/n)^{4/d}\right),
\text{ and} \\ B_{\ell,\ell'}(x,y) &=\Psi_2(x)\Psi_2(y)J_{\ell,\ell'}^{(2)}(x,y)+o\left((k/n)^{4/d}\right),
\end{align*}
where for $a\in\X$ and $1\leq j\leq n$, $J_j^{(1)}(a):=\E[\hat\tau_j(a)^2\1{\hat\tau_j(a)\le\varepsilon/k}]$ and
$$J_{\ell,\ell'}^{(2)}(x,y):=\E[\hat\tau_{\ell'}(x)^2\hat\tau_{\ell'}(y)^2\;\1{\tal+\hat\tau_{\ell'}(y)<\norm{y-x}}
\1{\hat\tau_\ell(x)\le\varepsilon/k}\1{\hat\tau_{\ell'}(y)\le\varepsilon/k}].$$
It remains to show that the difference $D:=J_\ell^{(1)}(x)J_{\ell'}^{(1)}(y)-J_{\ell,\ell'}^{(2)}(x,y)$ is
negligible. Denoting by $f_{\ell,x}$ the probability distribution of $\tal$, we first observe that 
$$J_\ell^{(1)}(x)J_{\ell'}^{(1)}(y)=J^{(1)}_{\ell,\ell'}(x,y)+J^{(3)}_{\ell,\ell'}(x,y),$$
with \begin{eqnarray*}
J^{(3)}_{\ell,\ell'}(x,y) & := & \int_{[0,\varepsilon/k]^2}r_1^2r_2^2\1{r_1+r_2\geq\norm{x-y}}f_{\ell,x}(r_1)
f_{\ell',y}(r_2)\,\mathrm d(r_1,r_2) \\ & \leq & \E\left[\tal^2\1{\tal\geq\norm{x-y}/2}\right]\cdot \E\left[
\hat\tau_{\ell'}^2(y)\right] \\ & + & \E\left[\tal^2\right]\cdot\E\left[\hat\tau_{\ell'}^2(y)\1{\hat\tau_{\ell'}(y)
\geq\norm{x-y}/2}\right] \\ & \leq & C(k/n)^4\exp\left(-L(n/k)\norm{x-y}^d/2^d\right), 
\end{eqnarray*}
where the last bound is obtained from Lemma \ref{moments_tau} and Lemma \ref{censored_tau} with suitable positive 
constants $C,L$. By integrating with respect to $x$ and $y$, we get $$\int_{\X^2}J^{(3)}_{\ell,\ell'}(x,y)\,
\mathrm dQ^{\otimes 2}(x,y)\leq 2C\overline q(k/n)^{4/d}\int_0^\infty r^{d-1}\exp\left(-L(n/k)r^d/2^d\right)\,
\mathrm dr=o\left((k/n)^{4/d}\right).$$
It then remains to show that
\[
\fbox{2-6}:=\va{J_{\ell,\ell'}^{(1)}(x,y)-J_{\ell,\ell'}^{(2)}(x,y)}
\]
is negligible compared to the leading term. Using the distributions from Lemmas~\ref{pair_ind_distribution} and
\ref{single_distribution}, and then by Eq.~\eqref{eq:PaPlusIsSigma}, we can show that
\[
\fbox{2-6}\le\int_{[0,\varepsilon/k]^2}F_x(r_1)^{\ell-1}F_y(r_2)^{\ell'-1}O\mtlr(){(r_1r_2)^{d+1}}\va{D_{x,y}}
(r_1,r_2)\,\mathrm d(r_1,r_2),
\]
where \begin{align*}
D_{x,y}(r_1,r_2) & :=C_{n,\ell}C_{n,\ell'}(1-F_x(r_1))^{n-\ell}(1-F_y(r_2))^{n-\ell'} \\ & \phantom{:=}-n(n-1) 
\binom{n-2}{\ell-1}\binom{n-\ell-1}{\ell'-1} \\ & \phantom{:=} \times\1{r_1+r_2<\norm{y-x}}(1-F_x(r_1)-F_y(r_2))^{
n-\ell-\ell'} \\ & =C_{n,\ell}C_{n, \ell'}\1{r_1+r_2<\norm{y-x}} \\ & \times\left\{(1-F_x(r_1))^{n-\ell}
(1-F_y(r_2))^{n-\ell'}-(1+o(1))(1-F_x(r_1)-F_y(r_2))^{n-\ell-\ell'}\right\},
\end{align*}
where $C_{\ell,n}:=n\binom{n-1}{\ell-1}$ and $C_{\ell',n}:=n\binom{n-1}{\ell'-1}$ and we used
$$n(n-1)\binom{n-2}{\ell-1}\binom{n-\ell-1}{\ell'-1}=C_{\ell,n}C_{\ell',n}(1+o(1)).$$

Therefore, we are left to analyse the following difference for $(r_1,r_2)\in[0,\varepsilon/k]^2$ such that $r_1+r_2
<\norm{x-y}$: $$\va{(1-F_x(r_1))^{n-\ell}(1-F_y(r_2))^{n-\ell'}-(1+o(1))(1-F_x(r_1)-F_y(r_2))^{n-\ell-\ell'}}.$$
From Lemma \ref{lem:CanChangeExponentOfOneMinusFExtendedCorrected}, applied successively with $r_1=0$ or $r_2=0$, 
and Lemma \ref{lem:PolyOneMinusFNearExpLnExtended}, one can show that
$(1-F_x(r_1))^{n-\ell}(1-F_y(r_2))^{n-\ell'}$ equals $$\exp(-n\,\vab\,(p(x)r_1^d +p(y)r_2^d))+O\left(\mathcal P
(n,\ell,\ell',r_1,r_2)\right)\exp(-n\,\vab\,(p(x)r_1^d+p(y)r_2^d)L_k),$$
for a suitable polynomial $\mathcal P$ and $L_k:=1-1/(2k)$. 
Similarly, by Lemmas~\ref{lem:CanChangeExponentOfOneMinusFExtendedCorrected} and
\ref{lem:PolyOneMinusFNearExpLnExtended}, $(1-F_x(r_1)-F_y(r_2))^{n-\ell-\ell'}$ equals to 
$$\exp(-n\,\vab\,(p(x)r_1^d-p(y)r_2^d))+O\left(\mathcal P'(n,\ell,\ell',r_1,r_2)\right)\exp(-n\,\vab\,(p(x)r_1^d+
p(y)r_2^d)L_k),$$
for another suitable polynomial $\mathcal P'$.

Applying Lemma~\ref{lem:negligible_integral_simple_simplified}, it is possible to show that \fbox{2-6} is negligible 
compared to $(k/n)^{4/d}$, uniformly in $(x,y,\ell,\ell')$, when $k^{d+1}=o(n)$. Before applying 
Lemma~\ref{lem:negligible_integral_simple_simplified}, one can note that
for $0\leq r\leq\varepsilon/k,\,1\leq \ell\leq k$, and $x\in\widetilde\X$, we have 
$$F_x(r)^{\ell-1}\leq \left(p(x)\,\vab\right)^{\ell-1}\left(1+O(\varepsilon/k)\right)^{k-1}$$
and the second factor is bounded in $k$.

\subsection{Proof of Theorem \ref{catchment}}

We use a general result stating that for any integrable non-negative random variable $Z$ and any $t\geq 0$,
$$\E[(Z-\E[Z])\1{Z\ge t}]\ge 0.$$ It is clear if $t\ge\E[Z]$, and if $t<\E[Z]$, then we can write 
$$\E[(Z-\E[Z])\1{Z\ge t}]=\E[(Z-\E[Z])(1-\1{Z<t})]=-\E[(Z-\E[Z])\1{Z<t}]\ge 0.$$
It also means that for any non-negative random variable $R$, we have \begin{align*}
\Pb(Z\ge t) & \le\frac{\E[Z\1{Z\ge t}]}{\E[Z]} \\ & \le\frac{\E[Z\1{Z\ge t}]}{\E[R\1{Z\ge t}]}\,\frac{\E[R\1{
Z\ge t}]}{\E[Z]} \\ & \le\frac{\E[Z\1{Z\ge t}]}{\E[Z\1{Z\ge t}]-\E[(Z-R)\1{Z\ge t}]}\,\frac{\E[R]}{\E[Z]}.    
\end{align*}
Setting $Z:=\frac nkQ(A_k(X_1)),\,R:=\frac nk\int_\X\1{\norm{X_1-x}\le\ta(x)}\1{\norm{X_1-x}>a}\,\mathrm dQ(x)$,
and $a^d:=\frac{2kt}{3n\qsup\,\vab}$, we first have $$Z-R\le\frac nk\int_\X\1{\norm{X_1-x}\le a}\,\mathrm dQ(x)\le
\frac nk\qsup\vab a^d\le\frac{2t}3.$$ Therefore, $$\frac{\E[Z\1{Z\ge t}]}{\E[Z\1{Z\ge t}]-\E[(Z-R)\1{Z\ge t}]}
\le\frac{\E[Z\1{Z\ge t}]}{\E[Z\1{Z\ge t}]-2t\Pb(Z\ge t)/3},$$ and because for $\alpha\ge 0$, the mapping
$x\mapsto\frac x{x-\alpha}$ is non-increasing on $(\alpha,\infty)$, it gives $$\frac{\E[Z\1{Z\ge t}]}{\E[Z\1{
Z\ge t}]-\E[(Z-R)\1{Z\ge t}]}\le\frac{t\Pb(Z\ge t)}{t\Pb(Z\ge t)-2t\Pb(Z\ge t)/3}=3.$$
Second we have $\E[Z]=1$ and because $\ta(x)$ is independent of the event $(\norm{X_1-x}\le\ta(x))=\left(
\hat\pi^{-1}(1)\leq k\right)$ by Lemma \ref{order_statistics}, where $\hat{\pi}$ denotes the permutation giving 
the random indices of the order statistics of the sample $\left(\norm{X_i-x}\right)_{1\leq i\leq n}$, applying 
Lemma \ref{tail_tau} yields $$\E[R]\le\frac nk\int_\X\Pb(\norm{X_1-x}\le\ta(x))\,\Pb(\ta(x)>a)\,\mathrm dQ(x)\le
e^{1/4}\exp\left(-\frac nk\frac{c\pinf\,\vab}8a^d\right).$$
Putting these two results together yields the statement of Theorem \ref{catchment}.

\subsection{Proof of Corollary \ref{overall_speed}}

We first consider the case $i=2$. Using the bias-variance decomposition $\hat e_2(h)-e(h)=B_{2,n}+(F_{m,n}+E_n)$, 
where \begin{align*}
B_{2,n} & :=\E[\hat e_2(h)\mid X]-e(h)=\int_\X\frac 1k\sum_{i=1}^n(g(X_i)-g(x))\1{\norm{X_i-x}\le\ta(x)}\,
\mathrm dQ(x), \\ F_{m,n} & :=\hat e_2(h)-\E[\hat e_2(h)\mid X,Y]=\meanm\hat g_n(X_j^*)-\int_\X\hat g_n(x)\,
\mathrm dQ(x)\text{, and} \\ E_n & :=\E[\hat e_2(h)\mid X,Y]-\E[\hat e_2(h)\mid X]=\int_\X\frac 1k\sum_{i=1}^n
(h(X_i,Y_i)-g(X_i))\1{\norm{X_i-x}\le\ta(x)}\,\mathrm dQ(x), 
\end{align*} 
and using the fact that $\E[F_{m,n}\mid X,Y]=0$ whereas $E_n$ and $B_n$ are $(X,Y)$-measurable, and $\E[E_n\mid X]=0$
whereas $B_n$ is $X$-measurable, we find that $$\E[(\hat e_2(h)-e(h))^2]=\E[B_{2,n}^2]+\E[F_{m,n}^2]+\E[E_n^2].$$
The control of the second-order bias was established in Theorem \ref{better_bias}. For the variance term, note 
that, conditionally on $(X,Y)$, $F_{m,n}$ is an average of i.i.d. centered random variables, so that 
\begin{align*} 
\E[F_{m,n}^2] & =\frac 1m\E[\Var(\hat g_n(X^*)\mid X,Y)] \\ & \leq\frac 1m\,\E\left[\int_\X\hat g_n^2(x)\,
\mathrm dQ(x)\right] \\ & \leq\frac 1m\int_\X\E\left[\frac{1}{k}\sum_{i=1}^n h^2(X_i,Y_i)\1{\norm{X_i-x}\le\ta(x)}
\right]\,\mathrm dQ(x)\text{ by Cauchy-Schwarz inequality,} \\ & =\frac 1m\int_\X\E[\Var(h(X,Y)\mid X)+g(X)^2] \\ &
\leq\frac{\bar\s^2+\norm g_\infty^2}m.
\end{align*}
Now observe that $E_n=\meann\left(\frac nkQ(A_k(X_i))\right)(h(X_i,Y_i)-g(X_i))$, and that conditionally on $X$, 
the random variables $Q(A_k(X_i))(h(X_i,Y_i)-g(X_i)),\,i=1,\dots,n$ are centered and independent. Thus, using the
boundedness assumption \ref{cond:reg7} on the conditional variance, we have $$\E[E_n^2]\leq\frac{\bar\s^2}n
\E\left[\left(\frac nkQ(A_k(X_1))\right)^2\right].$$
Now, applying Theorem \ref{catchment} yields \begin{align*}
\E\left[\left(\frac nkQ(A_k(X_1))\right)^2\right] & =\int_0^\infty 2t\,\Pb\left(\frac nkQ(A_k(X_1))\geq t\right)\,
\mathrm dt \\ & \leq 6e^{1/4}\,\int_0^\infty t\exp\left(-\frac{c\pinf}{12\qsup}\,t\right)\,\mathrm dt \\ & \leq
6e^{1/4}\left(\frac{12\qsup}{c\pinf}\right)^2. 
\end{align*}
It means that setting $K_{P,Q,h}:=\frac{2^53^3e^{1/4}\qsup^2\bar\s^2}{c^2\pinf^2}$, we have $$\E[E_n^2]\leq 
K_{P,Q,h}\;\frac 1n.$$ This concludes the proof for the case $i=2$.

If now $i=1$, the variance part $V_{1,n}:=\hat e_1(h)-e(h)-B_{1,n}$ has been already studied in 
\citet{portier2023scalable}, Theorem $1$. In particular, we have $$\E[V_{1,n}^2]=O\left(m^{-1}+n^{-1}\right).$$
The proof then follows from Theorem \ref{better_bias}.

\subsection{Proof of Theorem \ref{gate_bias}}

For the global average treatment effect, we can write $$B=\frac 1N\sum_{\underset{W_i=1}{i=1}}^N\left(
\frac 1k\sum_{\ell=1}^k\left(g_0(X_i)-g_0\left(X_{\hat i_\ell(X_i)}\right)\right)\right)-\frac 1N\sum_{\underset{
W_i=0}{i=1}}^N\left(\frac 1k\sum_{\ell=1}^k\left(g_1(X_i)-g_1\left(X_{\hat i_\ell(X_i)}\right)\right)\right).$$ 
We work conditionally on the treatment $W$ and we apply twice Theorem \ref{better_bias}, first by setting $P$ the
distribution of $X$ given $W=0$ and $Q$ the distribution of $X$ given $W=1$, and second by inverting the roles of
$W=0$ and $W=1$. Assumptions \ref{cond:T3} and \ref{cond:T5} along with Bayes' theorem ensure that both conditional
densities $f_0$ and $f_1$ are bounded away from zero and from infinity. Applying Theorem \ref{better_bias}, the 
two terms in the bias are of order $\frac{N_1}N\left(\frac k{N_0}\right)^{\min\{3/2,2/d\}}$ and $\frac{N_0}N
\left(\frac k{N_1}\right)^{\min\{3/2,2/d\}}$, respectively, where it is assumed that $N_0$ or $N_1$ do not vanish.
However, using Assumption \ref{cond:T5}, one can integrate this bounds and get the required rates. 
Indeed, setting $a:=\min\{3,4/d\}$ and $p_w=\Pb(W=w)$ for $w\in\{0,1\}$, we have for some $\delta\in (0,1)$,
\begin{eqnarray*}
\E\left(N^aN_w^{-a}\1{N_w\geq 1}\right) & \leq & \E\left(N^aN_w^{-a}\1{N_w\leq(1-\delta)Np_w}\right)+N^a\left(
(1-\delta)Np_w\right)^{-a} \\ & \leq & N^a\Pb\left(N_w\leq(1-\delta)Np_w\right)+(1-\delta)^{-a}p_w^{-a} \\
& \leq & N^a\exp\left(-\frac{\delta^2Np_w}2\right)+(1-\delta)^{-a}p_w^{-a},
\end{eqnarray*}
which is bounded in $N$. The last inequality is based on the Chernoff concentration bound on binomial random 
variables, see for instance \citet{boucheron2013concentration}. The previous bounds complete the proof.

\subsection{Proof of Proposition \ref{assumption_counter}}

\begin{enumerate}
\item For $z\in[0,1]$, let $(z,z^2),\,(1,z)$ and $(z,0)$ be three possible points of the boundary of $\X$. We give 
a lower bound of the distance between $(x,y)\in\X$ and the point $(z,z^2)$ using the point $(x,x^2)$. We have 
$$(x^2-y)^2\leq 2(x^2-z^2)^2+2(z^2-y)^2\leq 8(x-z)^2+2(z^2-y)^2,$$ which yields $$\norm{(x,y)-(z,z^2)}^2\geq
\frac 18(x^2-y)^2.$$ 
On the other hand, we have $$\norm{(x,y)-(1,z)}^2\geq(1-x)^2\text{ and }\norm{(x,y)-(z,0)}^2\geq y^2.$$ 
We then obtain $\delta((x,y),\X^c)^2\geq\min\{(x^2-y)^2/8\,,\,(1-x)^2\,,\,y^2\}$, so that $$\exp(-L\delta((x,y),
\X^c)^2)\leq\exp(-L(x^2-y)^2/8)+\exp(-L(1-x)^2)+\exp(-Ly^2).$$
Since $$\sqrt L\int_\X\left(\exp(-L(1-x)^2)+\exp(-Ly^2)\right)\,\mathrm d(x,y)\leq 2\int_{\R_+}\exp(-z^2)\,
\mathrm dz,$$ and $$\sqrt L\int_\X\exp(-L(x^2-y)^2/8)\,\mathrm d(x,y)\leq\sqrt L\int_0^1\int_0^{x^2}\exp(-Lz^2/8)
\,\mathrm dz\mathrm dx\leq\int_{\R_+}\exp(-z^2/8)\,\mathrm dz,$$ we conclude that Assumption \ref{cond:regA} is 
verified.

On the other hand, if $x\leq\varepsilon/2$, $$B((x,y),\varepsilon/2)\cap\X\subset\{(x,y)\in\R^2\,:\,x\in
[0,\varepsilon],\,y\in[0,x^2]\},$$ whose volume is upper bounded by $$\int_0^\varepsilon x^2\,\mathrm dx=
\frac{\varepsilon^3}3.$$ Condition \ref{cond:reg2} is not verified since a lower bound by $\varepsilon^2$, up to 
a positive constant, is impossible.

\item Ignoring the normalisation constant $1/\va\X$ and setting for $k\ge 1,\,c_k:=b_k-a_k$, we have 
\begin{align*}
\sqrt L\int_\X\exp(-L\delta(x,\X^c)^2)\,\mathrm dx & =\sqrt L\sum_{k=1}^\infty\int_{\mathcal C_k}\exp(-L
\delta(x,\X^c)^2)\,\mathrm dx \\ & \geq\sqrt L\sum_{k=1}^\infty\int_{a_k\leq\norm x\leq\frac{a_k+b_k}2}\exp(-L
(\norm x-a_k)^2)\,\mathrm dx \\ & =\sqrt L\,\sigma(S^1)\sum_{k=1}^\infty\int_{a_k}^{\frac{a_k+b_k}2}\exp(-L(r-
a_k)^2)r\,\mathrm dr \\ & =\sqrt L\,\sigma(S^1)\sum_{k=1}^\infty\int_0^{\frac{c_k}2}\exp(-Lr^2)(r+a_k)\,\mathrm dr 
\\ & \geq\sqrt L\,\sigma(S^1)\sum_{k=1}^\infty a_k\int_0^{\frac{c_k}2}\exp(-Lr^2)\,\mathrm dr \\ & =\sigma(S^1)
\sum_{k=1}^\infty a_k\int_0^{\frac{c_k\sqrt L}2}\exp(-r^2)\,\mathrm dr.
\end{align*}
When $L\to\infty$, monotone convergence ensures that the latter lower bound goes to $$\left(\sigma(S^1)
\int_0^\infty\exp(-r^2)\,\mathrm dr\right)\sum_{k=1}^\infty a_k=\infty,$$
which shows that Condition \ref{cond:regA} is not verified.
Now we show that $\X$ satisfies Condition \ref{cond:reg2}. Let $x\in\X$, suppose first that $\norm x\leq 2
\varepsilon$. Setting for $k\ge 1,\,\delta_k:=a_{k-1}-b_k$, we have the bound $$\va{B(x,\varepsilon)\cap\X}\geq
\pi\varepsilon^2-C\varepsilon\sum_{k=1}^\infty\delta_k\1{b_k<\norm x+\varepsilon}.$$
Indeed, we have to subtract to the area of the disc the areas of all its intersections with some rings with
thickness $\delta_k$. Any intersection of this type has an area bounded by $C\delta_k\varepsilon$ with some
$C>0$ not depending on $x,k,\varepsilon$. Because $\delta_k\leq\delta/(k+1)^4$ and $1/(k+1)=a_k\leq b_k$, we then
obtain $$\va{B(x,\varepsilon)\cap\X}\geq\pi\varepsilon^2-C\delta\varepsilon\sum_{k=1}^\infty\frac{\norm x+
\varepsilon}{(k+1)^3}\geq\pi\varepsilon^2-3C\delta\varepsilon^2\sum_{k=1}^\infty\frac 1{(k+1)^3},$$ and the lower
bound follows if $\delta$ is small enough. 

We next assume that $\norm x>2\varepsilon$ and we consider three cases. \begin{itemize}
\item Assume first that there is no circle with radius $a_l$ that intersects $B(x,\varepsilon)$. Let $k\ge 0$ such
that $x\in\mathcal C_k$. Then the ball of center $x-\varepsilon u/2$, where $u=x/\norm x$ and of radius
$\varepsilon/2$ is included in $\mathcal C_k\cap B(x,\varepsilon)$ and $\va{B(x,\varepsilon)\cap\X}\geq\pi
\varepsilon^2/4$.
\item Suppose next that there is only one circle with radius $a_\ell$ that intersects the ball $B(x,\varepsilon)$.
Since $x\in\X$, we only have two possibilities. The first one is $$a_k\leq\norm x-\epsilon<\norm x\leq b_k<
a_{k-1}\leq\norm x+\varepsilon$$ for some $k\ge 1$. In this case, a ball of radius $\varepsilon/2$ can be included
in $B(x,\varepsilon)\cap\mathcal C_k$ as above. The second case is $$a_{k+1}<\norm x-\varepsilon\leq a_k\leq 
\norm x<\norm x+\varepsilon<a_{k-1}.$$
In this case, since $\varepsilon\leq a_{k-1}-a_k=b_k-a_k+\delta_k\leq 2(b_k-a_k)$, we get $b_k-a_k\geq
\varepsilon/2$ and we deduce that $B(x,\varepsilon)\cap\mathcal C_k$ contains a ball of radius $\varepsilon/4$.
\item Finally, we consider the case for which $j\geq 2$ circles with radius $a_\ell$ intersect the ball 
$B(x,\varepsilon)$. Suppose that $a_k,\dots,a_{k+j-1}$ are the radii of the circles that intersect 
$B(x,\varepsilon)$. It automatically means that $$\frac{2\varepsilon}{\norm x^2-\varepsilon^2}=(\norm x -
\varepsilon)^{-1}-(\varepsilon+\norm x)^{-1}\geq a_{k+j-1}^{-1}-a_k^{-1}=j-1.$$
We then get $\norm x^2\leq\varepsilon^2+2\varepsilon/(j-1)$. However, $\va{B(x,\varepsilon)\cap\X}\geq\pi
\varepsilon^2-\eta\varepsilon\sum_{s=0}^j\delta_{k+s}$ for some universal constant $\eta>0$. Indeed, we have to 
subtract from the area of the ball some of its intersections with rings of thickness $\delta_s$ for some values of
$s$. Note that the case $s=k$ has to be considered. Besides $$\sum_{s=0}^j\delta_{k+s}\leq\frac{\delta(j+1)}{
(k+1)^4}\leq 2\delta(\norm x+\varepsilon)^4\left(\frac\varepsilon{\norm x^2-\varepsilon^2}+1\right).$$
Using our inequalities $4\varepsilon^2\leq\norm x^2\leq\varepsilon^2+2\varepsilon$, we obtain that
$$\va{B(x,\varepsilon)\cap\X}\geq\pi\varepsilon^2-D\delta\varepsilon^2$$ for some $D>0$ not depending on $x$
and $\varepsilon$. This shows \ref{cond:reg2} if $\delta>0$ is small enough.
\end{itemize}
\end{enumerate}

\subsection{Proof of Theorem \ref{geometry_cases}}

\begin{enumerate}
\item Suppose that $\X$ is a compact convex subset of $\R^d$ with a non-empty interior. We show that $\X$ 
satisfies Condition \ref{cond:reg2} by constructing the cone $\mathfrak C^1$ depicted in
Figure~\ref{fig:illust_cvx_implies_X2} as follows. Let $B:=B(z_0,\varepsilon)$ be an open ball contained in $\X$ 
and let $x\in\X$. There exists $z\in B$ such that $B(z,\varepsilon/2)\subset B$ and $\norm{z-x}\geq\varepsilon/2$.
Let $H$ be the hyperplane of $\R^d$ containing $z$ and orthogonal to the line segment $[x;z]$ and let $B':=B\cap
H$. Then the cone $\mathfrak C^1=\operatorname{Conv}(x,B')$ is included in $\X$ by convexity.
It means that for all $0\le r\le\diam(\X),\,\va{B(x,r)\cap\X}\geq\va{B(x,r)\cap\mathfrak C^1}$.

However, if $r<\varepsilon/2$, then $r$ is smaller than the height of the cone.
The proportion of the ball $B(x,r)$ contained in the cone as a volume is an increasing function of the cone's 
angle with vertex $x$. This angle is always greater than $\arctan(\varepsilon/(2\diam(\X)))$. Indeed, if $u\in B'$
with $\norm{u-z_0}=\varepsilon$, the tangent of this angle equals the ratio $\norm{u-z}/\norm{z-x}$, with 
$\norm{z-x}\leq\diam(\X)$ and $\norm{u-z}\geq\varepsilon/2$, since $B(z,\varepsilon/2)\subset B$. 

It means that there exists a constant $c_0\in(0,1)$ determined by $d,\,\varepsilon$, and $\diam(\X)$, such that 
for all $0\le r\le\varepsilon/2,$ we have $\va{B(x,r)\cap\X}\geq c_0\va{B(x,r)}$. \\ Then Assumption 
\ref{cond:reg2} is satisfied for $c=\left(\frac\varepsilon{2\diam(\X)}\right)^dc_0>0$.

\begin{figure}[H]
        \centering
        \includegraphics[height=5.8cm]{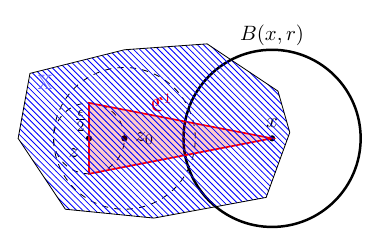}  
        \caption{Illustration for the proof of Theorem~\ref{geometry_cases}-1. Inside $\X$, we construct a cone $\mathfrak C^1$ and a sufficiently large angle with vertex $x$ so that its intersection with $B(x,r)$ will have a volume that grows proportionally to $\va{B(x,r)}$ as $r$ increases. Such a cone can be constructed by taking the convex hull of the ball $B(z,\varepsilon/2)$ and $x$.}
        \label{fig:illust_cvx_implies_X2}
    \end{figure}
    
    \begin{figure}[H]
        \centering
        \includegraphics[width=0.8\linewidth]{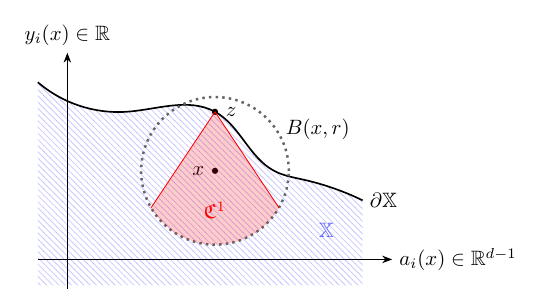}
        \caption{Illustration for the proof of Theorem~\ref{geometry_cases}-3. when $x\in\X$ is close to the border.
        Take $z \in \partial X$ by up-shifting the coordinate $y_z$ of $z$.
        The ball $B(x,r)$ has a larger intersection with $\X$ than $B(z_s,r)$.
        To show $\va{B(x,r)\cap\X}\ge c\va{B(x,r)}$ for some constant $c>0$, we construct a cone $\mathfrak C^1$ that clips at least a constant proportion of $B(x,r)$ inside $\X$, similarly to Figure~\ref{fig:illust_cvx_implies_X2}.
        We can find such a cone with a positive vertex angle using the fact that a coordinate (the vertical axis in the figure) of the graph of $\partial\X$ as a function of the other coordinates cannot deviate at a rate larger than a Lipschitz constant.
        We can ensure this by taking $\mathfrak C^1$ in a way that its surface does not drop vertically, thanks to the Lipschitz continuity of the boundary. The vertex angle and thus the clipped volume can depend on the location $z$, but its minimum is still positive thanks to the compactness.}
        \label{fig:illust_c1manifold_implies_X2}
 \end{figure}
    
    \begin{figure}[H]
       \centering
       \includegraphics[width=9cm,height=7cm]{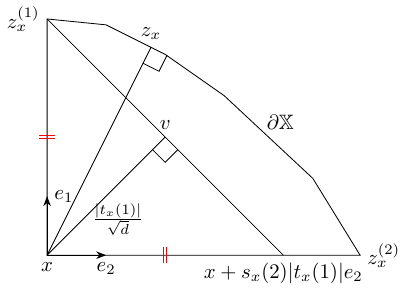}
       \caption{Illustration for the proof of Theorem~\ref{geometry_cases}-2 with the Euclidean norm $\norm\cdot$. Without loss of generality, suppose that $\va{t_x(1)}=\min_{1\le j\le d}\va{t_x(j)}$. The norm $\norm{
       z_x-x}$ is larger than that of the perpendicular line between $x$ and $v$: $\norm{z_x-x}\ge\norm{x-v}= \frac{\va{t_x(1)}}{\sqrt d}$.}
       \label{fig:illust_cvx_implies_A}
    \end{figure}
    
    \begin{figure}[H]
        \centering
        \includegraphics[width=\linewidth]{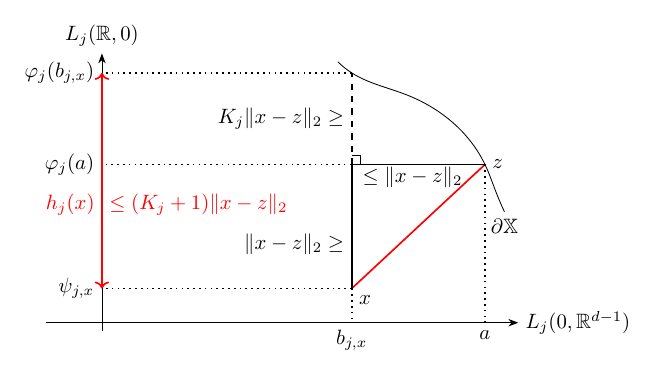}
        \caption{Illustration for the proof of Theorem~\ref{geometry_cases}-4. The distance $\norm{x-z}$ is lower bounded by $h_j(x):=\va{\varphi_j(b_{j,x})-\psi_{j,x}}$, up a constant factor. $h_j(x)$ is simply the distance between $x$ and a point on the graph having the same $d-1$ first coordinates $b_{j,x}$. The exponential function $x\mapsto\exp\left(-L\delta(x,\X^c)^d\right)$ will be then first integrated with respect to the last coordinate.}
        \label{fig:illust_c1manifold_implies_A}
 \end{figure}

\item Suppose that $\X$ is a compact convex subset of $\R^d$ with a non-empty interior. We show that $\X$ 
satisfies Condition \ref{cond:regA}. First we remind that the boundary of a convex subset has a vanishing
Lebesgue measure. For simplicity, we give a short proof of this fact. Without loss of generality, we assume that
$0\in\X^\circ$. If $x\in\partial\X$, then for any $\varepsilon\in(0,1),\,(1-\varepsilon)x\in\X^\circ$. Indeed,
$(1-\varepsilon)x$ is located in the interior of the convex cone $\{(1-\lambda)x+\lambda u\,:\,\lambda\in[0,1],
u\in B(0,r)\}$ which is contained in $\X$ as soon as $B(0,r)\subset\X$. We then deduce that $\partial\X\subset
(1-\varepsilon)^{-1}\X^\circ\backslash\X^\circ$ and $$\Leb(\partial\X)\leq\Leb((1-\varepsilon)^{-1}\X^\circ)-
\Leb(\X^\circ)=\left(\frac 1{(1-\epsilon)^d}-1\right)\Leb(\X^\circ).$$ Letting $\varepsilon\to 0$, we get the
result.

Now, let $x\in\X^\circ$. We denote by $z_x$ a point of $\partial\X$ that minimises the distance between $x$ and 
the compact set $\partial\X$. For a real number $y$, we denote by $s(y)$ the sign of $y,\,s(y):=1$ if $y\geq 0,\,
s(y):=-1$ if $y<0$. We also set $s_x(i):=\sign(e_i^\top z_x-e_i^\top x)$, where $e_1,\dots,e_d$ are the canonical 
basis vectors of $\R^d$.
Define $$t_x(i):=s_x(i)\inf\{t>0\,:\,x+ts_x(i)e_i\in\partial\X\}.$$ We will prove that $\underset{1\le i\le d}
\min\va{t_x(i)}\le\sqrt d\norm{x-z_x}$ (see Figure~\ref{fig:illust_cvx_implies_A}). 
Suppose that $\underset{1\le i\le d}\min\va{t_x(i)}\le\sqrt d\norm{x-z_x}$. We will obtain a contradiction by 
showing that $z_x$ can be only located in the interior of $\X$. Indeed, if $\lambda_1,\dots,\lambda_d$ are
non-negative real numbers such that $\sum_{i=1}^d\lambda_i=1$, then $\sum_{i=1}^d\lambda_iz_x^{(i)}=x+\sum_{i=1}^d
\lambda_it_x(i)e_i$ are elements of the convex set $\X$, where $z_x^{(i)}=x+t_x(i)e_i$.
Suppose that $z_x=x+\sum_{i=1}^d\beta_ie_i$, where $\beta_i$ has the same sign as $t_x(i)$ for $i=1,\dots,d$.
Setting for some $t>0$, $$\lambda_i:=\frac{t\beta_i}{t_x(i)}=\frac{t\va{\beta_i}}{\va{t_x(i)}},\,i=1,\dots,d,$$ 
the condition $\sum_{i=1}^d\lambda_i=1$ is equivalent to $$t=\left(\sum_{i=1}^d\va{\beta_i}/\va{t_x(i)}
\right)^{-1}.$$ Note that $t>1$ because
$$\sum_{i=1}^d\frac{\va{\beta_i}}{\va{t_x(i)}}<\frac{\sum_{i=1}^d\va{\beta_i}}{\sqrt d\norm{x-z_x}}=\frac{
\sum_{i=1}^d\va{\beta_i}}{\sqrt d\sqrt{\sum_{i=1}^d\beta_i^2}}\leq 1,$$
where the last inequality follows from the Cauchy-Schwarz inequality. We deduce that $x+t(z_x-x)\in\X$ for some
$t>1$ and then $z_x$ is the interior of $\X$, which contradicts the fact that it is a point on the boundary. 
We conclude that $\min_{1\le i\le d}\va{t_x(i)}\le\sqrt d\norm{x-z_x}$.

By the equivalence of the norms, there exists $\eta>0$ such that $$\norm{x-z_x}\geq\eta\underset{1\le i\le d}\min
\va{t_x(i)}.$$ But from its definition, we have the lower bound $$\va{t_x(i)}\geq g_i(x):=\inf\{\norm{x-z}\,:\,z
\in\partial\X,\,z_j=x_j,\,j\neq i\}.$$
Next for $i=1,\dots,d$, let $$I_i((x_j)_{j\neq i}):=\{z_i\in\R\,:\,z\in\X\text{ if }z_j=x_j,\,j\neq i\}$$ be the
$i$th section of the convex set $\X$ at point $x$. By convexity, these sections are closed intervals, that is
$$I_i((x_j)_{j\neq i})=\left[a(x_j)_{j\neq i},b(x_j)_{j\neq i}\right]$$ for some real numbers $a(x_j)_{j\neq i}<
b(x_j)_{j\neq i}$. We now have the bound $$\exp\left(-L\delta(x,\X^c)^d\right)\leq\sum_{i=1}^d\exp\left(-L\eta^d
g_i(x)^d\right).$$ Moreover, for $i=1,\dots,d$, we have \begin{align*}
\int_{a(x_j)_{j\neq i}}^{b(x_j)_{j\neq i}}\exp(-L\eta^dg_i(x)^d)\,\mathrm dx_i & =\int_{a(x_j)_{j\neq i}}^{b(x_j)_{
j\neq i}}\exp(-L\eta^d\min\{x_i-a(x_j)_{j\neq i}\;,\;b(x_j)_{j\neq i}-x_i)\}\,\mathrm dx_i \\ & =2L^{-1/d}\int_0^{
\left(b(x_j)_{j\neq i}-a(x_j)_{j\neq i}\right)L^{1/d}}\exp(-\eta^dv^d)\,\mathrm dv.   
\end{align*}

Assuming that $\X\subset[r,R]^d$ for some real numbers $r<R$, from the Fubini therorem, we then deduce that
\begin{align*}
\fbox{1-1} & :=L^{1/d}\int_\X\exp(-L\delta(x,\X^c)^d)\,\mathrm dQ(x) \\ & \leq L^{1/d}\qsup\sum_{i=1}^d\int_{x_j\in
[r,R],j\neq i}\left\{\int_{a(x_j)_{j\neq i}}^{b(x_j)_{j\neq i}}\exp(-L\eta^dg_i(x)^d)\,\mathrm dx_i\right\}\prod_{
j\neq i}\mathrm dx_j \\ & \leq 2d(R-r)^{d-1}\qsup\int_0^\infty\exp(-\eta^dv^d)\,\mathrm dv.
\end{align*}
The last upper bound is finite and the proof follows.

\item Suppose that $\X$ is the closure of a bounded open set whose boundary $\partial\X$ is a $(d-1)$-dimensional
submanifold of $\R^d$ of class $C^1$. We show that $\X$ satisfies Condition \ref{cond:reg2} by constructing the
truncated cone $\mathfrak C^1$ depicted in Figure~\ref{fig:illust_c1manifold_implies_X2}. We use the description 
of submanifolds of $\R^d$ through graphs; see for instance \citet[Theorem~1.21]{lafontaine2015introduction}: for 
all $z\in\partial\X$, there exist an open neighbourhood $V_z$ of $z$ in $\R^d$, an permutation of the coordinates
$L_z\in\mathcal O_d(\R)$, an open subset $A_z$ of $\R^{d-1}$, and a map $\varphi_z\in C^1(B_z,\R)$ such that
$V_z\cap\partial\X=\{L_z(a,\varphi_z(a))\,:\,a\in A_z\}$. 
If we reduce $V_z$ and $A_z$, we can assume that $V_z$ and $A_z$ are convex, that $V_z\subset L_z(A_z\times\R)$, 
and that the gradient $\nabla\varphi_z$ is bounded by $K_z\geq 0$ on $A_z$. By connectedness, we can show that 
$V_z\cap\X$ is either included in the subgraph or the supergraph of $\varphi_z$. We give a short proof below. \\
We partition $V_z\subset L_z(A_z\times\R)$ into the boundary $\partial\X\cap V_z$, the supergraph $V_z^+:=\{
L_z(a,y)\in V_z\,:\,y>\varphi_z(a)\}$, and the subgraph $V_z^-:=\{L_z(a,y)\in V_z\,:\,y<\varphi_z(a)\}$. Because
$\varphi_z$ is continuous, $V_z^+$ is a connected open subset of $\R^d$, which is covered by the two open sets
$\X^\circ$ and $(\bar\X)^c$, so that one of the two intersections $V_z^+\cap\X^\circ$ and $V_z^+\cap(\bar\X)^c$ is 
empty. The same applies to $V_z^-$. We cannot have $V_z^+\cap(\bar\X)^c=V_z^-\cap(\bar\X)^c=\emptyset$ since it 
would imply that $z\in\X^\circ$. Then we either have $V_z\cap\X^\circ\subset V_z^-$ or $V_z\cap\X^\circ\subset 
V_z^+$. Because $z\in\partial\X$ is in the closure of $\X^\circ$, reducing further $V_z$, we can assume that 
$V_z\cap\X^\circ$ is exactly either the subgraph or the supergraph of $\varphi_z$. 
\bigskip

Because $\partial\X$ is compact, there exist $z_1,\dots,z_k\in\partial\X$ for some $k\in\N$ such that $V_{z_1},
\dots,V_{z_k}$ is a subcover of $\partial\X$. In what follows, we simply write $V_i:=V_{z_i},\,A_i:=A_{z_i},\,
L_i:=L_{z_i},\,\varphi_i:=\varphi_{z_i}$, and $K_i:=K_{z_i}$. First, by compactness, we can find $\varepsilon_b>0$
such that for all $x\in\X\backslash\cup_{i=1}^kV_i,\,\delta(x,\partial\X)\geq\varepsilon_b$. We set $S:=\overline{
\cup_{z\in\partial\X}B(z,\varepsilon_b/2)}$ a compact subset of $\cup_{i=1}^kV_i$. Let $x\in\X$, if $x\not\in S$,
then $x$ is sufficiently in the interior of $\X$, so that $B(x,\varepsilon_b/2)\subset\X$. It implies that for all
$r\in[0;\varepsilon_b/2],\,\va{B(x,r)\cap\X}=\va{B(x,r)}$, and for $r\in[\varepsilon_b/2;\diam(\X)]$, we can write
$\va{B(x,r)\cap\X}\geq\frac{\va{B(x,\varepsilon_b/2)\cap\X}}{\va{B(x,\diam(\X))}}\times\va{B(x,\diam(\X))}\geq
\left(\frac{\varepsilon_b}{2\diam(\X)}\right)^d\va{B(x,r)}$. \\
We now consider the case where $x\in S\subset\cup_{i=1}^kV_i$ is close to the boundary. Then there exist $1\le 
i\le k$ and $\varepsilon_A>0$ such that $B(x,\varepsilon_A)\subset V_i$. By compactness of $S$, $\varepsilon_A$ 
can be chosen independently from $x$. We set $a_i(x)\in A_i$ and $y_i(x)\in\R$ such that $x=L_i(a_i(x),y_i(x))$. 
We will assume that is $V_i\cap\X$ is the subgraph of $\varphi_i$, the supergraph case being dealt with in the 
same way. \\
We now consider the point $z:=L_i(a_i(x),\varphi_i(a_i(x)))$ on the boundary $\partial\X$ that is right above $x$.
By the mean value inequality, for all $a\in A_i$, $\va{\varphi_i(a_i(x))-\varphi_i(a)}\leq K_i\norm{a_i(x)-a}$. 
This means that the slope of the boundary does not drop vertically around $z$. We construct the set $\mathfrak C^1$
depicted in Figure~\ref{fig:illust_c1manifold_implies_X2} as the intersection of the cone with vertex $z$,
directed downside or towards $x$, with slope $K_i$, and the ball $B(x,\varepsilon_A)$ that is included in $V_i$.
Formally, $$\mathfrak C^1:=\operatorname{Conv}\{L_i(a,\varphi_i(a_i(x))-K_i\norm{a_i(x)-a})\mid a\in A_i\}\cap
B(x,\varepsilon_A).$$ It is included in $V_i$ and in the subgraph of $\varphi_i$, and therefore it is included in
$\X$. We can then control the proportion of balls centered in $x$ that are inside $\X$ as in the convex case. If
$r\in[0;\varepsilon_A]$, then $$\va{B(x,r)\cap\X}\geq\va{B(x,r)\cap\mathfrak C^1}\geq c(d,K_i,\varepsilon_A,
\diam(\X))\va{B(x,r)}.$$ If $r\in[\varepsilon_A;\diam(\X)]$, then $\va{B(x,r)\cap\X}\geq c(d,K_i,\varepsilon_A,
\diam(\X))\left(\frac{\varepsilon_A}{\diam(\X)}\right)^d\va{B(x,r)}$. \\ In the end, we showed that Condition 
\ref{cond:reg2} was satisfied for $$c:=\left(\frac{\varepsilon_b}{2\diam(\X)}\right)^d\wedge\underset{1\le i\le k}
\min c(d,K_i,\varepsilon_A,\diam(\X))\times\left(\frac{\varepsilon_A}{\diam(\X)}\right)^d>0.$$

\bigskip

\item We keep the same notation as in the previous point and still consider a covering $\cup_{i=1}^kV_i$ of the 
boundary $\partial \X$ such that on each $V_i$, the boundary $\partial\X$ can be represented as a graph of a 
Lipschitz mapping $\varphi_i$. Now, let $x\in\X$ and $z\in\partial\X$, we want to lower bound the distance between 
$x$ and $z$. Without loss of generality, we assume that $V_i=B\left(z_i,\varepsilon_i\right)$ for some
$z_i\in\partial\X$ and $\varepsilon_i>0$. Suppose first that $x\notin\cup_{i=1}^kV_i$. Then, for $z\in\partial\X$
and $1\le i\le k$, $$\norm{x-z}\geq\norm{x-z_i}-\norm{z-z_i}\geq\varepsilon_i-\norm{z-z_i}.$$
We deduce that $$\norm{x-z}\geq\underset{1\leq i\leq k}\max\left\{\varepsilon_i-\norm{z-z_i}\right\}.$$
Note that the mapping $$z\mapsto g(z):=\underset{1\leq i\leq k}\max\left\{\varepsilon_i-\norm{z-z_i}\right\}$$
is continuous (as a maximum of a finite number of continuous mappings) and positive on the compact set 
$\partial \X$, we deduce that $\delta(x,\X^c)\geq\inf_{z\in\partial\X}g(z)>0$.

Next, suppose that for a non empty subset $J$ of $\ieu[k],\,x\in V_j$ for $j\in J$ and $x\notin V_\ell$ for
$\ell\notin J$. Suppose first that there is $j\in J$ such that $z\in V_j$, and set $z:=L_j(a,\varphi_j(a))$ for 
some $a\in A_j$ and $x=L_j(b_{j,x},\psi_{j,x})$ for some $(b_{j,x},\psi_{j,x})\in A_j\times\R$.
Our aim here is to lower bound the distance from $x$ to $z$ by its distance with respect to the point 
$L_j(b_{j,x},\phi_j(b_{j,x}))$ located on the graph of $\phi_j$. See Figure~\ref{fig:illust_c1manifold_implies_A} 
for an illustration. We have \begin{align*}
\va{\psi_{j,x}-\varphi_j(b_{j,x})} & \le\va{\psi_{j,x}-\varphi_j(a)}+\va{\varphi_j(a)-\varphi_j(b_{j,x})} \\ &
\le\norm{x-z}+K_j\norm{a-b_{j,x}} \\ & \le(K_j+1)\norm{x-z}.
\end{align*}
Now if $z$ is in the compact set $\partial\X\setminus\cup_{j\in J}V_j$, we still have $$\norm{x-z}\geq c_J:=
\underset{z\in\partial\X\setminus\cup_{j\in J}V_j}\inf\underset{\ell\notin J}\max\left\{\varepsilon_{\ell}-
\norm{z-z_\ell}\right\}>0,$$ as we obtained before when $J=\emptyset$. Note that $c_J$ coincides with 
$\underset{z\in\partial\X}\inf g(z)$ in the latter case.
Setting $\eta:=\underset{1\leq i\leq k}\min(1+K_i)^{-1}$, we conclude that $\delta(x,\partial\X)\geq\eta\min\{c_J\;,
\;h_J(x)\}$, where $$h_J(x):=\underset{j\in J}\min\va{L_j^{-1}(x)_d-\varphi_j\left(\pi_{d-1}\circ L_j^{-1}(z)
\right)},$$ $L_j^{-1}(x)_d$ denotes the last coordinates of the vector $L_j^{-1}(x)$, and $\pi_{d-1}(x)=
(x_1,\ldots,x_{d-1})$. Setting for $J\subset\ieu[k],\,V_J:=\cap_{j\in J}V_j\cap\cap_{\ell\in\ieu[k]\setminus J}
V_\ell^c$, the set $\left\{V_J:J\subset\ieu[k]\right\}$ is a partition of $\R^d$ and we obtain 
\begin{align*}
\fbox{1-2} & =\int_\X\exp(-L\,\delta(x,\X^c)^d)\,\mathrm dQ(x) \\ & \leq\sum_{J\subset\ieu[k]}\int_{B_J}\exp(-L
(\eta c_J)^d)\,\mathrm dQ(x)+\sum_{J\subset\ieu[k],J\neq\emptyset}\int_{B_J}\exp(-L(\eta h_J(x))^d)\,\mathrm dQ(x) 
\\ & \leq 2^k\left(CL^{-1/d}+\sum_{i=1}^k\int_{B_j}\exp\left(-L\eta^d\va{L_j^{-1}(x)_d-\varphi_j\left(\pi_{d-1}
\circ L_j^{-1}(x)\right)}^d\right)\,\mathrm dx\right) \\ & \leq C\left(L^{-1/d}+\sum_{j=1}^k\int_{L_j^{-1}\X}\exp(-L
\eta^d\va{x_d-\varphi_j(x_1,\ldots,x_{d-1})}^d)\;\mathrm dx\right) \\ & \leq C\left(L^{-1/d}+\int_{\R_+}\exp(-L
\eta^d y^d)\,\mathrm dy\right) \\ & \leq CL^{-1/d}.
\end{align*}
For the third and the fourth inequality given above, we first use a change of variables involving the orthogonal 
transformation $L_j$ and we then apply  Fubini's therorem using the fact that $L_j^{-1}\X$ and its sections are 
bounded sets and thus have a finite Lebesgue measure.

\item If $\X_1$ and $\X_2$ satisfy Condition \ref{cond:reg2}, then it follows directly from the definition that
$\X_1\cup\X_1$ satisfies Condition \ref{cond:reg2}.

\item Suppose that $\X_1$ and $\X_2$ satisfy Condition \ref{cond:regA}. We have the inequality, for all $x\in
\R^d$, $\delta(x,(\X_1\cup\X_2)^c)\geq\max\{\delta(x,\X_1^c),\delta(x,\X_2^c)\}$. Therefore if $$\mathbb L_\X:=
\underset{L>0}\sup\left\{L^{1/d}\int_\X\exp(-L\,\delta(x,\X^c)^d)\;\mathrm dQ(x)\right\},$$ then we have 
$\mathbb L_{\X_1\cup\X_2}\leq\mathbb L_{\X_1}+\mathbb L_{\X_2}<\infty$, and $\X_1\cup\X_2$ satisfies Condition
\ref{cond:regA}.
\end{enumerate}

\subsection{Proof of Proposition \ref{equivalence}}

Set $\delta(x)$ the distance between $x$ and the complement $\X^c$ of $\X$. Suppose first that 
$Q\left(A_{\epsilon}\right)=O(\epsilon)$. Note that $$Q\left(\X\cap\{\delta=0\}\right)=\underset{r\to\infty}\lim
Q\left(\X\cap\{\delta\leq 1/r\}\right)=0.$$
It is sufficient to show that $$\underset{L>0}\sup\,L^{1/d}\int_{\X\cap\{0<\delta\leq 1\}}\exp
\left(-L\delta(x)^d\right)\,\mathrm dQ(x)<\infty.$$
We have \begin{align*}
\fbox{1-1} & =L^{1/d}\int_{\X\cap\{0<\delta\leq 1\}}\exp\left(-L\delta(x)^d\right)\,\mathrm dQ(x) \\
& =L^{1/d}\sum_{r=0}^{\infty}\int_{\X\cap\{2^{-r-1}<\delta\leq 2^{-r}\}}\exp\left(-L\delta(x)^d
\right)\,\mathrm dQ(x) \\ & \leq L^{1/d}\sum_{r=0}^{\infty}Q\left(\left\{2^{-r-1}<\delta\leq 
2^{-r}\right\}\right)\exp\left(-L2^{-d(r+1)}\right) \\ & \leq CL^{1/d}\sum_{r=0}^{\infty}2^{-r}\exp\left(-L
2^{-d(r+1)}\right) \\ & \leq 4CL^{1/d} \sum_{r=0}^{\infty}\int_{2^{-r-2}}^{2^{-r-1}}\exp\left(-Ly^d\right)\,
\mathrm dy \\ & \leq 4C\int_0^\infty\exp\left(-y^d\right)\,\mathrm dy,
\end{align*}
which is finite. 

Conversely, suppose that condition \ref{cond:regA} is satisfied. Setting $L=\epsilon^{-d}$ for some $\epsilon>0$,
we have $$L^{1/d}\int_\X\exp\left(-L\delta(x)^d\right)\,\mathrm dQ(x)\geq L^{1/d}\exp\left(-L\epsilon^d
\right)Q(A_{\epsilon})=e^{-1}Q(A_{\epsilon})/\epsilon.$$
This shows that $Q(A_{\epsilon})=O(\epsilon)$ which is equivalent to the desired assertion.

\subsection{Proof of Theorem \ref{local_bias_control}}

Applying Taylor's formula to the regression function $g(x)=\E[h(X,Y)\mid X=x]$ yields $g(x)=\zeta(x,z)^\top G(x)+
R(x,z)$, where $G_j(x):=\frac{\partial_{\Psi(j)}g(x)}{\Psi(j)!}$ and $$R(x,z):=\frac 1{(l-1)!}\int_0^1(1-t)^{l-1}
(\nabla^lg(x+tz)-\nabla^lg(x))(z,\dots,z)\,\mathrm dt.$$
Then \begin{align*}
\hat\gamma & =\underset{\gamma\in\R^{K^*}}\argmin\sum_{i=1}^n[(G(x)-\gamma)^\top\zeta(x,X_i)+(h-g)(X_i,Y_i)+
R(x,X_i)]^2\1{\norm{X_i-x}\leq\ta(x)} \\ & =G(x)+M^{-1}(x)\sum_{i=1}^n(R(x,X_i)+(h-g)(X_i,Y_i))\zeta(x,X_i)
\,\1{\norm{X_i-x}\leq\ta(x)}. 
\end{align*}
Thus we have \begin{align*} 
B_n & =\int_\X(\E[e_1^\top\hat\gamma\mid X]-g(x))\,\mathrm dQ(x) \\ & =\int_\X\sum_{i=1}^nR(x,X_i)\,e_1^\top 
M^{-1}(x)\zeta(x,X_i)\1{\norm{X_i-x}\leq\ta(x)}\,\mathrm dQ(x), \\ B_n^2 & \leq\int_\X k\sum_{i=1}^nR(x,X_i)^2
\left(e_1^\top M^{-1}(x)\zeta(x,X_i)\1{\norm{X_i-x}\leq\ta(x)}\right)^2\,\mathrm dQ(x) \\ & \leq\frac{
\Lambda^2}{l!}k\int_\X\ta(x)^{2(l+\beta)}\,e_1^\top\left(\sum_{i=1}^nM^{-1}(x)\zeta(x,X_i)\zeta(x,X_i)^\top
M^{-1}(x)\1{\norm{X_i-x}\leq\ta(x)}\right)e_1\,\mathrm dQ(x) \\ & =\frac{\Lambda^2}{l!}k\int_\X\ta(x)^{2
(l+\beta)}\,e_1^\top M^{-1}(x)e_1\,\mathrm dQ(x) \\ \E[B_n^2] & \leq\frac{\Lambda^2}{l!}k\int_\X\E[\ta(x)^{2
(l+\beta)}\E[e_1^\top M^{-1}(x)e_1\mid\ta(x)]]\,\mathrm dQ(x) \\ & \leq C_\X\frac{\Lambda^2}{l!}\left(\frac kn
\right)^{\frac{2(l+\beta)}d},\quad\text{using the following key lemma}.
\end{align*}

\begin{lemma}\label{local_key_lemma}
Suppose that Assumptions \ref{cond:reg1} to \ref{cond:reg3} hold true, then there exists a constant $C_\X$ 
depending on $d,l$, and $P$ such that almost surely, for all $k\geq(2D+1)K^*+1$ and $Q$-almost all $x\in\X$, 
$$\E[e_1^\top M^{-1}(x)e_1\mid\ta(x)]\leq C_\X/k.$$
\end{lemma}
\begin{proof} If $L=0$, then $M^{-1}(x)=1/k$, so now we suppose that $L\geq 1$. We follow the same proof structure
as in \citep{holzmann2024multivariate}. Fix $x\in\X$. Let $u$ be the first column of the random matrix $M^{-1}(x)$,
then $M(x)u=e_1$ and $e_1^\top M^{-1}(x)e_1=u_1=u^\top e_1=u^\top M(x)u$, so that
\begin{align*}
u_1 & =\sum_{i=1}^n\left(\sum_{j=1}^{K^*}u_j\,\zeta_j(x,X_i)\right)^2\1{\norm{X_i-x}\leq\ta(x)}, \\ 1/u_1 &
=\sum_{i=1}^n\left(\sum_{j=1}^{K^*}\frac{u_j}{u_1}\,\zeta_j(x,X_i)\right)^2\1{\norm{X_i-x}\leq\ta(x)} \\ &
\geq\underset{\tilde P\in P_1}\inf\sum_{i=1}^n\tilde P(x-X_i)^2\,\1{\norm{X_i-x}\leq\ta(x)} \\ & \geq\underset{
\tilde P\in P_1}\inf\sum_{i=1}^n\tilde P(x-X_i)^2\,\1{\norm{X_i-x}<\ta(x)},
\end{align*}
where we have set $P_1$ the set of polynomials in $d$ variables with total degree at most $L$ taking the value $1$ 
at $0$. Now write $\{i\in\ieu\mid\norm{X_i-x}<\ta(x)\}=\{i_1<\dots<i_{k-1}\}$ and for $a=1,\dots,k-1,\,V_a:=
\frac{x-X_{i_a}}{\ta(x)}$. It means that we perform a normalisation, changing $B(x,\ta(x))$ into $B(0,1)$. We can
therefore write \begin{align*}
1/u_1 & \geq\underset{\tilde P\in P_1}\inf\sum_{a=1}^{k-1}\tilde P(V_a)^2, \\ \E[u_1\mid\ta(x)] & \leq\int_{\R_+}
\Pb\left(\underset{\tilde P\in P_1}\inf\sum_{a=1}^{k-1}\tilde P(V_a)^2\leq t^{-1}\mid\ta(x)\right)\,\mathrm dt.   
\end{align*}

We want to divide the conditional sample $(V_a)_{a=1,\dots,k-1}$ into $\nu$ smaller samples all of size $K^*$. 
Let $\nu:=\floor{k-1/K^*}$, we can write \begin{align*}
\E[u_1\mid\ta(x)]\leq\int_{\R_+}\Pb\left(\sum_{j=0}^{\nu-1}\underset{\tilde P\in P_1}\inf\sum_{a=1}^{K^*}
\tilde P(V_{jK^*+a})^2\leq t^{-1}\mid\ta(x)\right)\,\mathrm dt.    
\end{align*}
We are now interested in bounding the quantity $$\Pb\left(\underset{\tilde P\in P_1}\inf\sum_{a=1}^{K^*}
\tilde P(V_a)^2\leq t^{-1}\mid\ta(x)\right).$$

As in \citet{holzmann2024multivariate}, for $w\in(\R^d)^{K^*}$, let $\Xi(w)$ be the matrix of size $K^*$ defined 
by $\Xi(w)_{a,j}:=w_a^{\Psi(j)}$. Then setting $V:=(V_a)_{a=1,\dots,K^*}\in(\R^d)^{K^*}$, we have \begin{align*}
\underset{\tilde P\in P_1}\inf\sum_{a=1}^{K^*}\tilde P(V_a)^2 & \geq\underset{\gamma\in\R^{K^*},\gamma_1=1}\inf
\norm{\Xi(V)\gamma}^2 \\ & \geq\underset{\gamma\in\R^{K^*},\norm\gamma\geq 1}\inf\norm{\Xi(V)\gamma}^2 \\ & 
\geq\la_{\min{}}(\Xi(V))^2 \\ & \geq\frac{\det(\Xi(V))^2}{\rho(\Xi(V))^{2(K^*-1)}} \\ & \geq\frac{\det(\Xi(V))^2}{
\norm{\Xi(V)}_F^{2(K^*-1)}},
\end{align*}
where $\la_{\min{}}$ denotes the smallest eigenvalue, $\rho$ the spectral radius and $\norm{\,\cdot\,}_F$ the
Frobenius norm. Thus, \begin{align*}
\norm{\Xi(V)}_F^2 & =\sum_{a,j=1}^{K^*}V_a^{2\Psi(j)} \\ & =\sum_{a=1}^{K^*}\sum_{l=0}^L\sum_{\underset{\va\la=l}{
\la\in\N^d}}V_a^{2\la} \\ & \leq\sum_{a=1}^{K^*}\sum_{l=0}^L\sum_{\underset{\va\la=l}{\la\in\N^d}}\binom l{\la_1,
\dots,\la_d}V_a^{2\la} \\ & =\sum_{a=1}^{K^*}\sum_{l=0}^L\norm{V_a}^{2l} \\ & \leq K^*(L+1).
\end{align*}

It yields that for all $t>0$, \begin{align*} 
\Pb\left(\underset{\tilde P\in P_1}\inf\sum_{a=1}^{K^*}\tilde P(V_a)^2\leq t^{-1}\mid\ta(x)\right) & \leq\Pb(\det(
\Xi(V))^2\leq\norm{\Xi(V)}_F^{2(K^*-1)}t^{-1}\mid\ta(x)) \\ & \leq\Pb(\va{\det(\Xi(V))}\leq t^{-1/2}(K^*(L+1))^{
(K^*-1)}\mid\ta(x)).
\end{align*}

Let $\varepsilon>0$, we are now interested in bounding the quantity $\Pb(\va{\det(\Xi(V))}\leq\varepsilon\mid
\ta(x))$. Let $\mu:=\underset{w\in B_d(0,1)^{K^*}}\max\va{\det\Xi(w)}>0$ and let $w$ be a point on which this
maximum is attained.
For $\theta\in S^{dK^*-1}$, we can apply the fundamental theorem of algebra to the polynomial $\det(\Xi(w+s\theta))
\in\C[s]$. Writing its roots $v_1^\theta,\dots,v_D^\theta$, we have $$\va{\det(\Xi(w+s\theta))}=\mu\prod_{i=1}^D
\va{1-s/v_i^\theta}.$$

If $s\ge 0$, then the distance between $s$ and the complex circle $(\va{\,\cdot\,}=\va{v_i^\theta})$ is reached 
at $\va{v_i^\theta}$ so that $$\va{\det(\Xi(w+s\theta))}\geq\mu\prod_{i=1}^D\va{1-\frac s{\va{v_i^\theta}}}.$$

Introduce $\Theta\in S^{dK^*-1}$ the direction of $[w,V]$, that is $\Theta_a=\frac{V_a-w_a}{\norm{V-w}},\,
a=1,\dots,K^*$ which is almost surely well-defined since $V$ has a density. We then have $V=w+\norm{V-w}\Theta$ 
and $$\va{\det(\Xi(V))}\geq\mu\prod_{i=1}^D\va{1-\frac{\norm{V-w}}{\va{v_i^\Theta}}}.$$

But $B_d(0,1)^{K^*}\subset B_{dK^*}(0,\sqrt{K^*})$, so that $\norm{V-w}\leq 2\sqrt{K^*}$ and $$\va{\det(\Xi(V))}
\geq\mu\prod_{i=1}^D\va{1-\frac{\norm{V-w}}{\va{v_i^\Theta}\wedge 2\sqrt{K^*}}},$$ which means we may replace
$\sva{v_i^\theta}$ by $\sva{v_i^\theta}\wedge 2\sqrt{K^*}$. Thus,\begin{align*} 
(\va{\det(\Xi(V))}\leq\varepsilon) & \subset\left(\prod_{i=1}^D\va{1-\frac{\norm{V-w}}{\va{v_i^\Theta}}}\leq
\varepsilon/\mu\right) \\ & \subset\bigcup_{i=1}^D\left(\va{1-\frac{\norm{V-w}}{\va{v_i^\Theta}}}\leq(\varepsilon/
\mu)^{1/D}\right) \\ & \subset\bigcup_{i=1}^D\left(1-(\varepsilon/\mu)^{1/D}\leq\frac{\norm{V-w}}{\va{v_i^\Theta}}
\leq 1+(\varepsilon/\mu)^{1/D}\right), \\ \Pb\left(\va{\det(\Xi(V))}\leq\epsilon\mid\ta(x)\right) & \leq\sum_{
i=1}^D\Pb(V-w\in\mathcal C_i(\varepsilon)\mid\ta(x)),
\end{align*}
where $$\mathcal C_i(\varepsilon):=\{r\theta:\,\theta\in S^{dK^*-1},\,0\vee(1-(\varepsilon/\mu)^{1/D})\sva{
v_i^\theta}\le r\le(1+(\varepsilon/\mu)^{1/D})\sva{v_i^\theta}\}\cap B_{dK^*}\left(0,2\sqrt{K^*}\right).$$

This is where the geometrical condition \ref{cond:reg2} plays a role. \citet{holzmann2024multivariate} worked
conditionally on the direction $\Theta$. It is justified by the fact that if we assume that the target support is
included in the interior of the source support $\X$, then there exists $\epsilon>0$ such that for all $z$ in the 
target support and all directions $\theta\in S^{d-1},[z-\epsilon\theta,z+\epsilon\theta]\subset\X$. This is not 
the case in our setting, which is why we have to let the direction free. It is only by averaging on all possible
directions with Condition \ref{cond:reg2} that we will get the same bound. First, we can upper bound the volume of
the set $\mathcal C_i(\varepsilon)$ by \begin{align*}
\va{\mathcal C_i(\varepsilon)} & =\int_0^{2\sqrt{K^*}}\left(\int_{S^{dK^*-1}}\1{r\in [(1-(\varepsilon/\mu)^{1/D})
\sva{v_i^\theta};(1+(\varepsilon/\mu)^{1/D})\sva{v_i^\theta}]}\,\mathrm d\s(\theta)\right)r^{dK^*-1}\,\mathrm dr 
\\ & \leq 2\left(2\sqrt{K^*}\right)^{dK^*-1}(\varepsilon/\mu)^{1/D}\int_{S^{dK^*-1}}\sva{v_i^\theta}\,
\mathrm d\sigma(\theta),    
\end{align*} 
with $\sva{v_i^\theta}\leq 2\sqrt{K^*}$. Therefore, using the conditional density from Lemma \ref{density_V}, we
obtain $$\Pb\left(\va{\det(\Xi(V))}\leq\varepsilon\mid\ta(x)\right)\leq\sum_{i=1}^D\va{\mathcal C_i(\varepsilon)}
\left(\frac{\psup\ta(x)^d}{\int_{\R^d}f^*(z)\,\mathrm dz}\right)^{K^*}\text{, where}$$ $$f^*(z):=\1{z\in B_d(0,1)}
\;p(x-z\ta(x))\,\ta(x)^d.$$ The geometric condition \ref{cond:reg2} now appears in the normalisation constant
$\int_{\R^d}f^*(z)\,\mathrm dz$. Using the lower bound from Lemma \ref{density_V}, we obtain $$\Pb\left(
\va{\det(\Xi(V))}\leq\varepsilon\mid\ta(x)\right)\leq C\,\varepsilon^{1/D},$$ where $C:=2D\left(\frac{(2
\sqrt{K^*})^d\psup}{c\pinf\vab}\right)^{K^*}\mu^{-1/D}\,\sigma(S^{dK^*-1})<\infty$.

We have proved that for all $t>0,\,\Pb\left(\underset{\tilde P\in P_1}\inf\sum_{a=1}^{K^*}\tilde P(V_a)^2\leq 
t^{-1}\mid\ta(x)\right)\leq C\,t^{-1/2D}$. Now recall that $\E[u_1\mid\ta(x)]\leq\int_{\R_+}\Pb\left(\sum_{j=0}^{
\nu-1}\underset{\tilde P\in P_1}\inf\sum_{a=1}^{K^*}\tilde P(V_{jK^*+a})^2\leq t^{-1}\mid\ta(x)\right)\,
\mathrm dt$, where $\nu=\floor{k-1/K^*}>2D$. To obtain the statement of Lemma \ref{local_key_lemma}, because we
work with variable $k$, we will need an anti-concentration result. Applying Lemma \ref{anti_concentration} with 
$Z_j:=\underset{\tilde P\in P_1}\inf\sum_{a=1}^{K^*}\tilde P(V_{jK^*+a})^2$ and $b:=1/2D$, we obtain, for all 
$C'>0$, $$\E[u_1\mid\ta(x)]\leq\frac{C'}\nu+\frac 1\nu\,\Gamma(b)\,C^\nu\,\frac{\Gamma(1+b)^{\nu-1}}{\Gamma(b\nu)}
\int_{C'/\nu}^\infty t^{-b\nu}\,\mathrm dt.$$
But since $b\nu>1$, \begin{align*}
C^\nu\,\frac{\Gamma(1+b)^\nu}{\Gamma(b\nu)}\int_{C'/\nu}^\infty t^{-b\nu}\,\mathrm dt & =\frac 1{b\nu-1}
\frac{(C\,\Gamma(1+b))^\nu}{\Gamma(b\nu)}\left(\frac\nu{C'}\right)^{b\nu-1} \\ & \leq\frac{C'C''\sqrt b}{(b\nu-1)
\sqrt{2\pi\nu}}\left(\frac{C\Gamma(1+b)e^bb^{-b}}{C'^b}\right)^\nu,
\end{align*}
using Stirling's approximation. If we choose $C'$ large enough, the above expression remains bounded as
$\nu\to\infty$. Since $\nu=\floor{k-1/K^*}$, we have shown that $\E[u_1\mid\ta(x)]\leq C_\X/k.$
\end{proof}

\bibliographystyle{plainnat} %
\bibliography{references}

\begin{thebibliography}{28}
\providecommand{\natexlab}[1]{#1}
\providecommand{\url}[1]{\texttt{#1}}
\expandafter\ifx\csname urlstyle\endcsname\relax
  \providecommand{\doi}[1]{doi: #1}\else
  \providecommand{\doi}{doi: \begingroup \urlstyle{rm}\Url}\fi

\bibitem[Abadie and Imbens(2006)]{abadie2006large}
Alberto Abadie and Guido~W. Imbens.
\newblock Large sample properties of matching estimators for average treatment
  effects.
\newblock \emph{Econometrica}, 74:\penalty0 235--267, 2006.

\bibitem[Abadie and Imbens(2008)]{abadie2008failure}
Alberto Abadie and Guido~W Imbens.
\newblock On the failure of the bootstrap for matching estimators.
\newblock \emph{Econometrica}, 76\penalty0 (6):\penalty0 1537--1557, 2008.

\bibitem[Abadie and Imbens(2011)]{abadie2011bias}
Alberto Abadie and Guido~W. Imbens.
\newblock Bias-corrected matching estimators for average treatment effects.
\newblock \emph{Journal of Business \& Economic Statistics}, 29\penalty0
  (1):\penalty0 1--11, 2011.

\bibitem[Abadie and Imbens(2012)]{abadie2012martingale}
Alberto Abadie and Guido~W Imbens.
\newblock A martingale representation for matching estimators.
\newblock \emph{Journal of the American Statistical Association}, 107\penalty0
  (498):\penalty0 833--843, 2012.

\bibitem[Bickel and Scheffer(2006)]{bickel2006dirichlet}
Steffen Bickel and Tobias Scheffer.
\newblock Dirichlet-enhanced spam filtering based on biased samples.
\newblock \emph{Advances in Neural Information Processing Systems}, 19, 2006.

\bibitem[Boucheron et~al.(2013)Boucheron, Lugosi, and
  Massart]{boucheron2013concentration}
St{\'{e}}phane Boucheron, G{\'{a}}bor Lugosi, and Pascal Massart.
\newblock \emph{Concentration Inequalities - {A} Nonasymptotic Theory of
  Independence}.
\newblock Oxford University Press, 2013.
\newblock ISBN 978-0-19-953525-5.
\newblock URL \url{https://doi.org/10.1093/acprof:oso/9780199535255.001.0001}.

\bibitem[Devroye et~al.(2013)Devroye, Ferrario, Gy{\"o}rfi, and Walk]{Dev2}
Luc Devroye, Paola~G Ferrario, L{\'a}szl{\'o} Gy{\"o}rfi, and Harro Walk.
\newblock Strong universal consistent estimate of the minimum mean squared
  error.
\newblock \emph{Empirical Inference: Festschrift in Honor of Vladimir N.
  Vapnik}, pages 143--160, 2013.

\bibitem[Devroye et~al.(2018)Devroye, Gy{\"o}rfi, Lugosi, and Walk]{Dev1}
Luc Devroye, L{\'a}szl{\'o} Gy{\"o}rfi, G{\'a}bor Lugosi, and Harro Walk.
\newblock {A nearest neighbor estimate of the residual variance}.
\newblock \emph{Electronic Journal of Statistics}, 12\penalty0 (1):\penalty0
  1752 -- 1778, 2018.

\bibitem[Gadat et~al.(2016)Gadat, Klein, and Marteau]{Gadat}
S{\'e}bastien Gadat, Thierry Klein, and Cl{\'e}ment Marteau.
\newblock {Classification in general finite dimensional spaces with the
  k-nearest neighbor rule}.
\newblock \emph{The Annals of Statistics}, 44\penalty0 (3):\penalty0 982 --
  1009, 2016.

\bibitem[Gretton et~al.(2009)Gretton, Smola, Huang, Schmittfull, Borgwardt,
  Sch{\"o}lkopf, et~al.]{gretton2009covariate}
Arthur Gretton, Alex Smola, Jiayuan Huang, Marcel Schmittfull, Karsten
  Borgwardt, Bernhard Sch{\"o}lkopf, et~al.
\newblock Covariate shift by kernel mean matching.
\newblock \emph{Dataset shift in machine learning}, 3\penalty0 (4):\penalty0 5,
  2009.

\bibitem[Hahn(1998)]{hahn1998role}
Jinyong Hahn.
\newblock On the role of the propensity score in efficient semiparametric
  estimation of average treatment effects.
\newblock \emph{Econometrica}, 66\penalty0 (2):\penalty0 315--331, 1998.

\bibitem[Heckman(1979)]{heckman1979sample}
James~J Heckman.
\newblock Sample selection bias as a specification error.
\newblock \emph{Econometrica: Journal of the Econometric Society}, pages
  153--161, 1979.

\bibitem[Holzmann and Meister(2024)]{holzmann2024multivariate}
Hajo Holzmann and Alexander Meister.
\newblock Multivariate root-n-consistent smoothing parameter free matching
  estimators and estimators of inverse density weighted expectations.
\newblock \emph{arXiv preprint arXiv:2407.08494}, 2024.
\newblock URL \url{https://arxiv.org/abs/2407.08494}.

\bibitem[Jiang and Zhai(2007)]{jiang2007instance}
Jing Jiang and ChengXiang Zhai.
\newblock Instance weighting for domain adaptation in nlp.
\newblock In \emph{Proceedings of the 45th Annual Meeting of the Association
  Computational Linguistics}. ACL, 2007.

\bibitem[Kanamori et~al.(2009)Kanamori, Hido, and Sugiyama]{kanamori2009least}
Takafumi Kanamori, Shohei Hido, and Masashi Sugiyama.
\newblock A least-squares approach to direct importance estimation.
\newblock \emph{The Journal of Machine Learning Research}, 10:\penalty0
  1391--1445, 2009.

\bibitem[Klarner et~al.(2023)Klarner, Rudner, Reutlinger, Schindler, Morris,
  Deane, and Teh]{klarner2023drug}
Leo Klarner, Tim~GJ Rudner, Michael Reutlinger, Torsten Schindler, Garrett~M
  Morris, Charlotte Deane, and Yee~Whye Teh.
\newblock Drug discovery under covariate shift with domain-informed prior
  distributions over functions.
\newblock In \emph{International Conference on Machine Learning}, pages
  17176--17197. PMLR, 2023.

\bibitem[Lafontaine et~al.(2015)]{lafontaine2015introduction}
Jacques Lafontaine et~al.
\newblock \emph{An introduction to differential manifolds}.
\newblock Springer, 2015.

\bibitem[Li et~al.(2010)Li, Kambara, Koike, and Sugiyama]{li2010application}
Yan Li, Hiroyuki Kambara, Yasuharu Koike, and Masashi Sugiyama.
\newblock Application of covariate shift adaptation techniques in
  brain--computer interfaces.
\newblock \emph{IEEE Transactions on Biomedical Engineering}, 57\penalty0
  (6):\penalty0 1318--1324, 2010.

\bibitem[Lin et~al.(2023)Lin, Ding, and Han]{lin2023estimation}
Zhexiao Lin, Peng Ding, and Fang Han.
\newblock Estimation based on nearest neighbor matching: From density ratio to
  average treatment effect.
\newblock \emph{Econometrica}, 91\penalty0 (6):\penalty0 2187--2217, 2023.

\bibitem[Loog(2012)]{loog2012nearest}
Marco Loog.
\newblock Nearest neighbor-based importance weighting.
\newblock In \emph{2012 IEEE international workshop on machine learning for
  signal processing}, pages 1--6. IEEE, 2012.

\bibitem[Portier et~al.(2024)Portier, Truquet, and Yamane]{portier2023scalable}
Fran{\c{c}}ois Portier, Lionel Truquet, and Ikko Yamane.
\newblock Nearest neighbor sampling for covariate shift adaptation.
\newblock \emph{Journal of Machine Learning Research}, 25\penalty0
  (410):\penalty0 1--42, 2024.

\bibitem[Shimodaira(2000)]{shimodaira2000improving}
Hidetoshi Shimodaira.
\newblock Improving predictive inference under covariate shift by weighting the
  log-likelihood function.
\newblock \emph{Journal of Statistical Planning and Inference}, 90\penalty0
  (2):\penalty0 227--244, 2000.

\bibitem[Singh and P{\'o}czos(2016)]{singh2016finite}
Shashank Singh and Barnab{\'a}s P{\'o}czos.
\newblock Finite-sample analysis of fixed-k nearest neighbor density functional
  estimators.
\newblock \emph{Advances in neural information processing systems}, 29, 2016.

\bibitem[Sricharan et~al.(2012)Sricharan, Raich, and
  Hero]{sricharan2012estimation}
Kumar Sricharan, Raviv Raich, and Alfred~O Hero.
\newblock Estimation of nonlinear functionals of densities with confidence.
\newblock \emph{IEEE Transactions on Information Theory}, 58\penalty0
  (7):\penalty0 4135--4159, 2012.

\bibitem[Sugiyama et~al.(2007)Sugiyama, Krauledat, and
  M{\"u}ller]{sugiyama2007covariate}
Masashi Sugiyama, Matthias Krauledat, and Klaus-Robert M{\"u}ller.
\newblock Covariate shift adaptation by importance weighted cross validation.
\newblock \emph{Journal of Machine Learning Research}, 8\penalty0 (5), 2007.

\bibitem[Sugiyama et~al.(2008)Sugiyama, Suzuki, Nakajima, Kashima, von B\"unau,
  and Kawanabe]{sugiyama2008direct}
Masashi Sugiyama, Taiji Suzuki, Shinichi Nakajima, Hisashi Kashima, Paul von
  B\"unau, and Motoaki Kawanabe.
\newblock Direct importance estimation for covariate shift adaptation.
\newblock \emph{Annals of the Institute of Statistical Mathematics},
  60\penalty0 (4):\penalty0 699--746, 2008.

\bibitem[Weyl(1939)]{weyl1939volume}
Hermann Weyl.
\newblock On the volume of tubes.
\newblock \emph{American Journal of Mathematics}, 61\penalty0 (2):\penalty0
  461--472, 1939.

\bibitem[Whitney(1934)]{boundary}
Hassler Whitney.
\newblock Functions differentiable on the boundaries of regions.
\newblock \emph{Annals of Mathematics}, 35\penalty0 (3):\penalty0 482--485,
  1934.

\end{thebibliography}

\appendix

\begin{dftn}
Notation: \begin{itemize}
\item $P$ and $Q$ distributions on $\R^d$ with respective densities $p$ and $q$,
\item $\X$ a subset of $\R^d$ containing the supports of $P$ and $Q$,
\item $((X_1,Y_1),\dots,(X_n,Y_n))$ an $n$-sample of $P_{X,Y}=P_{Y\mid X}\cdot P$,
\item $((X_1^*,\cdot),\dots,(X_m^*,\cdot))$ an $m$-sample of $Q_{X,Y}=Q_{Y\mid X}\cdot Q$,
\item $h$ a measurable function from $\X\times\R$ to $\R$,
\item $g\colon x\mapsto\E[h(X,Y)\mid X=x]$,
\item $\1 A$ the indicator function of the event $A$,
\item $\norm\cdot$ the Euclidean norm on $\R^d$,
\item $\ik$ the index of the $k$th-nearest neighbour to $x$ among $(X_1,\dots,X_n)$,
\item $\ta(x)$ the $k$th-order statistic of $(\norm{X_i-x})_{i=1,\dots,n}$,
\item $A_k(z):=\{x\in\R^d\,:\,\norm{z-x}\leq\ta(x)\}$ the catchment area of a point $x$,
\item $M_k^*(x):=\sum_{j=1}^m\1{\norm{x-X_j^*}\leq\ta(X_j^*)}$ the matching estimator,
\item $\ieu$ the integer interval $i=1,\dots,n$,
\item $V^d$ the $\norm\cdot$-unit ball in $\R^d$,
\item $B(x,r)$ the closed ball for $\norm\cdot$ with center $x$ and radius $r$,
\item $S(x,r)$ the sphere for $\norm\cdot$ with center $x$ and radius $r$,
\item $S^{d-1}$ the $\norm\cdot$-unit sphere in $\R^d$,
\item $\s$ the surface measure on $S^{d-1}$,
\item $\va B$ the Lebesgue measure of a Borel set $B$ in $\R^d$,
\item $\diam(\X):=\sup\{\norm{y-x}\,:\,x,y\in\X\}$,
\item $\Gamma\colon s>0\mapsto\int_{\R_+}t^{s-1}e^{-t}\,\mathrm dt$ Euler's gamma function,
\item $\beta\colon(s,t)>0\mapsto\int_0^1u^{s-1}(1-u)^{t-1}\,\mathrm du=\frac{\Gamma(s)\Gamma(t)}{\Gamma(s+t)}$
Euler's beta function.
\end{itemize}
\end{dftn}

\bigskip

\begin{lemma}\label{order_statistics}
Let $Z_1,\dots,Z_n$ be iid real-valued random variables having density $f_Z$ with respect to the Lebesgue measure.
\begin{itemize}
\item Almost surely, there is a unique permutation $\hat\pi$ of $\ieu$ such that $Z_{\hat\pi(1)}<\dots<Z_{
\hat\pi(n)}$. From now on, we write $Z_{(i)}$ for $Z_{\hat\pi(i)}$.
\item The random variable $\hat\pi$ is uniformly distributed over the set of permutations of $\ieu$.
\item The random vector $(Z_{(1)},\dots,Z_{(n)})$ has density $$f(x_1,\dots,x_n)=n!\,f_Z(x_1)\dots f_Z(x_n)\,
\1{x_1<\dots<x_n}$$ with respect to the Lebesgue measure on $\R^n$.
\item The vector $(Z_{(1)},\dots,Z_{(n)})$ and the permutation $\hat\pi$ are independent.
\end{itemize} 
\end{lemma}
\begin{proof}
For the first item, recall that because the variables $Z_i$ have density, the event of having two observations 
with the same value has probability zero. We will prove the last three items at once. Let $\varphi\colon\R\to\R$ 
be a bounded and measurable function and let $\pi$ be a permutation of $\ieu$, then \begin{align*}
\E[\varphi(Z_{(1)},\dots,Z_{(n)})\1{\hat\pi=\pi}] & =\E[\varphi(Z_{\pi(1)},\dots,Z_{\pi(n)})\,\1{Z_{\pi(1)}<\dots<
Z_{\pi(n)}}] \\ & =\int_{\R^n}\varphi(x)\1{x_1<\dots<x_n}f_Z(x_1)\dots f_Z(x_n)\,\mathrm dx,
\end{align*}
since the vector $(Z_{\pi(1)},\dots,Z_{\pi(n)})$ has the same distribution as the vector $(Z_1,\dots,Z_n)$. 
Because the last expression above does not depend on the permutation $\pi$, we deduce that $\hat\pi$ follows a
uniform distribution on the set of permutations of $\ieu$, that is $\Pb(\hat\pi=\pi)=\frac 1{n!}$, and the last 
two points are valid.
\end{proof}

\begin{lemma}\label{beta}
Define Euler's beta function $\beta\colon(s,t)>0\mapsto\int_0^1u^{s-1}(1-u)^{t-1}\,\mathrm du=\frac{\Gamma(s)
\Gamma(t)}{\Gamma(s+t)}$. Let $U_1^n,\dots,U_n^n$ be an $n$-sample of uniformly distributed random variables over
$[0,1]$ and let $1\le k\le n$, then the $k$th-order statistic of this sample has density with respect to the 
Lebesgue measure on $[0,1]$, $$u\mapsto\frac 1{\beta(k,n-k+1)}\,u^{k-1}(1-u)^{n-k}.$$

Let $I(x;s,t):=\frac 1{\beta(s,t)}\,\int_0^xu^{s-1}(1-u)^{t-1}\,\mathrm du$ be the regularised incomplete beta
function. Then for all $x\in[0,1]$, we have $$I(x;k,n-k+1)=\Pb(U_{(k)}^n\leq x)=\Pb(B_{n,x}\geq k),$$
where $B_{n,x}$ is a random variable following the binomial distribution with parameters $n$ and $x$.
\end{lemma}
\begin{proof}
Setting $x^{(\bar k)}:=x_1,\dots,x_{k-1},x_{k+1},\dots,x_n$, we know from Lemma \ref{order_statistics} that the
density of the $k$th-order statistic of $(U_i)_{1\le i\le n}$ is equal to \begin{align*}
f_{U_{(k)}^n}(u) & =\int_{[0,1]^{n-1}}n!\,\1{x_1<\dots<x_{k-1}<u<x_{k+1}<\dots<x_n}\,\mathrm dx^{(\bar k)} \\ &
=n!\,\frac{u^{k-1}}{(k-1)!}\,\frac{(1-u)^{n-k}}{(n-k)!} \\ & =\frac 1{\beta(k,n-k+1)}\,u^{k-1}(1-u)^{n-k}.
\end{align*}    
From this, we can deduce the first equality $$I(x;k,n-k+1)=\Pb(U_{(k)}^n\leq x).$$ Observing that the events
$(U_{(k)}^n\leq x)$ and $\left(\sum_{i=1}^n\1{U_i^n\leq x}\geq k\right)$ are the same gives the second equality.
\end{proof}

\begin{lemma}\label{moments_tau}
Suppose that Assumptions \ref{cond:reg1} to \ref{cond:reg3} hold true. Let $\Gamma\colon s>0\mapsto\int_0^\infty
t^{s-1}e^{-s}\,\mathrm dt$ denote Euler's gamma function and let $\floor\cdot$ denote the floor function. Let
$\la\ge 0$, then for all $x\in\X$, we have $$\E[\ta(x)^\la]\leq C_{\la,d,P}\,\left(\frac k{n+1}\right)^{\la/d},$$
where $C_{\la,d,P}:=2\,\Gamma(2+\floor{\la/d})\,(c\pinf\,\vab)^{-\la/d}$.
\end{lemma}
\begin{proof}
The proof can be found in \citet{portier2023scalable}. We give a proof below for the sake of completeness.

Let $F_x(r):=\Pb(X\in B(x,r)),\,F_x^{-1}(u):=\inf\{r\in\R\,:\,F_x(r)\geq u\}$ its generalised inverse, and let 
$(U_1,\dots,U_n)$ be an $n$-sample of the uniform distribution on $[0,1]$. Then the two random vectors
$(\norm{X_1-x},\dots,\norm{X_n-x})$ and $(F_x^{-1}(U_1),\dots,F_x^{-1}(U_n))$ have the same distribution. 
Because the function $F_x^{-1}$ is non-decreasing, the two vectors $(\hat\tau_1(x),\dots,\hat\tau_n(x))$ and 
$(F_x^{-1}(U_{(1)}),\dots,F_x^{-1}(U_{(n)}))$ also have the same distribution. In particular, by Lemma 
\ref{order_statistics}, we can write $$\E[\ta(x)^\la]=\frac 1{\beta(k,n-k+1)}\int_0^1(F_x^{-1}(u))^\la\,u^{k-1}
(1-u)^{n-k}\,\mathrm du.$$

However, using Assumptions \ref{cond:reg2} and \ref{cond:reg3}, we know that for all $0\le r\le\diam(\X)$, we have 
$F_x(r)\geq c\pinf\,\vab\,r^d$, so that $F_x^{-1}(u)\leq\left(\frac u{c\pinf\,\vab}\right)^{1/d}$, which also 
holds for $u>c\pinf\,\vab\,\diam(\X)^d$. It means that \begin{align*}
\E[\ta(x)^\la] & \leq\frac 1{\beta(k,n-k+1)}(c\pinf\,\vab)^{-\la/d}\beta(k+\la/d,n-k+1) \\ & \leq(c\pinf\,\vab)^{
-\la/d}\frac{\Gamma(n+1)}{\Gamma(n+\la/d+1)}\frac{\Gamma(k+\la/d)}{\Gamma(k)} \\ & \leq (c\pinf\,\vab)^{
-\la/d}\frac 2{(n+1)^{\la/d}}\Gamma(2+\floor{\la/d})k^{\la/d},
\end{align*}
using the inequalities $\frac{\Gamma(a+s)}{\Gamma(a)}\leq\Gamma(2+\floor s)\,a^s$ and $\frac{\Gamma(a)}{
\Gamma(a+s)}\leq 2\,a^{-s}$, as shown in \citet{portier2023scalable}.    
\end{proof}

\begin{lemma}\label{tail_tau}
Suppose that Assumptions \ref{cond:reg1} to \ref{cond:reg3} hold true. Let $a>0$, then for all $x\in\X$,
$$\Pb(\ta(x)>a)\leq e^{1/4}\exp\left(-\frac nk\,\frac{c\pinf\,\vab}8\,a^d\right).$$
\end{lemma}
\begin{proof}
Let $(U_1^n,\dots,U_n^n)$ be an $n$-sample of the uniform distribution on $[0,1]$, then applying Lemma \ref{beta},
we know that $$\Pb(\ta(x)>a)=\Pb(U_{(k)}^n>F_x(a))=\Pb(B\leq k-1),$$ where $B$ has a binomial distribution with
parameters $n$ and $F_x(a)$, with $F_x(a)$ being defined in the proof of Lemma \ref{moments_tau}.
We know from concentration results on binomial random variables, as can be seen in 
\citet{boucheron2013concentration}, that for all $\delta\in(0,1)$, we have the Chernoff bound 
$$\Pb(B\leq(1-\delta)\E[B])\leq\exp\left(-\delta^2\,\frac{\E[B]}2\right).$$
With $\delta:=1/2$, we find $\Pb(B\leq\E[B]/2)\leq\exp(-\frac{\E[B]}8)$. It means that if $t\leq\E[B]/2$, then we
can write $\Pb(B\leq t)\leq\exp(-\frac{\E[B]}8)$, while if $t\geq\E[B]/2$, we simply write $\Pb(B\leq t)\leq 1\leq
e^{1/4}\exp(-\frac{\E[B]}{8t})$.
In either case, we showed that for all $t\ge 1$, $$\Pb(B\leq t)\leq e^{1/4}\exp\left(-\frac{\E[B]}{8t}\right).$$
In our context, it means that \begin{align*}
\Pb(B\leq k-1) & \leq\Pb(B\leq k) \\ & \leq e^{1/4}\exp\left(-\frac n{8k}\,F_x(a)\right) \\ & \leq e^{1/4}\exp
\left(-\frac nk\,\frac{c\pinf\,\vab}8\,a^d\right).   
\end{align*}
\end{proof}

\begin{lemma}\label{censored_tau}
Suppose that Assumptions \ref{cond:reg1} to \ref{cond:reg3} hold true. Let $\la\ge 0$ and $a>0$, then for all
$x\in\X$, $$\E[\ta(x)^\la\,\1{\ta(x)>a}]\leq C_{\la,d,P}\,\left(\frac k{n+1}\right)^{\la/d}\exp\left(-\frac{n+1}k
\,L_{\la,d,P}\,a^d\right),$$ where $C_{\la,d,P}:=2e^{1/4}\,\Gamma(2+\floor{\la/d})\,(c\pinf\,\vab)^{-\la/d}$ and
$L_{\la,d,P}:=\frac{c\pinf\,\vab}{8(1+\floor{\la/d})}$.
\end{lemma}
\begin{proof}
Exactly as in the proof of Lemma \ref{moments_tau}, we can write \begin{align*}
\E[\ta(x)^\la\,\1{\ta(x)>a}] & =\frac 1{\beta(k,n-k+1)}\int_{F_x(a)}^1(F_x^{-1}(u))^\la\,u^{k-1}(1-u)^{n-k}\,
\mathrm du \\ & \leq\frac 1{\beta(k,n-k+1)}(c\pinf\,\vab)^{-\la/d}\int_{F_x(a)}^1u^{k+\la/d-1}(1-u)^{n-k}\,
\mathrm du,
\end{align*}
where $$\int_{F_x(a)}^1u^{k+\la/d-1}(1-u)^{n-k}\,\mathrm du=\beta(k+\la/d,n-k+1)(1-I(F_x(a);k+\la/d,n-k+1)).$$
Now \begin{align*}
\beta(s,t)^2\partial_sI(\alpha;s,t) & =\int_{[0,1]^2}\1{u\leq \alpha}\,u^{s-1}(1-u)^{t-1}v^{s-1}(1-v)^{t-1}
(\log\,u-\log\,v)\,\mathrm d(u,v) \\ & =\int_{[0,1]^2}\1{u\leq\alpha\leq v}\,u^{s-1}(1-u)^{t-1}v^{s-1}(1-v)^{t-1}
(\log\,u-\log\,v)\,\mathrm d(u,v) \\ & \leq 0.    
\end{align*}
In the second equality above, we used the fact that for an integrable function $h:[0,1]^2\to\R$ such that 
$h(u,v)=-h(v,u)$ for each $(u,v)\in[0,1]^2$, we have $\int_{[0,1]^2}\1{u,v\leq\alpha}h(u,v)\,\mathrm d(u,v)=0$.
It means that the function $s\mapsto I(\alpha;s,t)$ is non-increasing, so that $$1-I(F_x(a);k+\la/d,n-k+1)\leq 1-
I(F_x(a);k+\floor{\la/d}+1,n-k+1).$$ Then, because $k+\floor{\la/d}+1$ is a positive integer, applying Lemma
\ref{beta} yields $$1-I(F_x(a);k+\floor{\la/d}+1,n-k+1)\leq\Pb(B<k+\floor{\la/d}+1),$$ 
where $B$ has a binomial distribution with parameters $n+\floor{\la/d}+1$ and $F_x(a)$.
Thanks to the binomial inequality proved in Lemma \ref{tail_tau}, we get \begin{align*}
\Pb(B\leq k+\floor{\la/d}) & \leq e^{1/4}\exp\left(-\frac{n+\floor{\la/d}+1}{8(k+\floor{\la/d})}\;F_x(a)\right) 
\\ & \leq e^{1/4}\exp\left(-\frac{n+1}k\;\frac{c\pinf\,\vab}{8(1+\floor{\la/d})}\;a^d\right).   
\end{align*}
In then end, we proved that \begin{align*}
\E[\ta(x)^\la\,\1{\ta(x)>a}] & \leq\frac{\beta(k+\la/d,n-k+1)}{\beta(k,n-k+1)}(c\pinf\,\vab)^{-\la/d}e^{1/4}\exp
\left(-\frac{n+1}k\;\frac{c\pinf\,\vab}{8(1+\floor{\la/d})}\;a^d\right), 
\end{align*}
where $$\frac{\beta(k+\la/d,n-k+1)}{\beta(k,n-k+1)}\leq 2\Gamma(2+\floor{\la/d})\,\left(\frac k{n+1}\right)^{
\la/d},$$ as we have seen in the proof of Lemma \ref{moments_tau}.
\end{proof}

\begin{lemma}\label{pair_ind_distribution}
Let $x,y\in\X,\,1\le\ell,\ell'\le n$ with $\ell+\ell'\leq n$, and let $\varphi_1,\varphi_2$ be measurable bounded
functions from $\R^d$ to $\R$, then setting $\Phi:=\E[\varphi_1(X_{\ik[\ell]}-x)\varphi_2(X_{\hat i_{\ell'}(y)}-y)
\1{\tal+\hat\tau_{\ell'}(y)<\norm{y-x}}]$, we have \begin{multline*}
\Phi=C_{n,\ell,\ell'}\int_{\R_+^2}\1{r_1+r_2<\norm{y-x}}F_x(r_1)^{\ell-1}F_y(r_2)^{\ell'-1}(1-F_x(r_1)-F_y(r_2))^{
n-\ell-\ell'} \\ \times(r_1r_2)^{d-1}P_x(\varphi_1,r_1)P_y(\varphi_2,r_2)\,\mathrm d(r_1,r_2),
\end{multline*}
where we have set $C_{n,\ell,\ell'}:=n(n-1)\binom{n-2}{\ell-1}\binom{n-\ell-1}{\ell'-1}$, $P_a(\varphi,r):=\int_{
S^{d-1}}\varphi(r\theta)p(a+r\theta)\,\mathrm d\s(\theta)$, and $F_a(r):=P(B(a,r))$.
\end{lemma}
\begin{proof}
We can show that the event $\mathcal D:=(\tal+\hat\tau_{\ell'}(y)<\norm{y-x})$ occurs if and only if there exists
a unique couple of integers $(i_1,i_2)\in\ieu^2$ with $i_1\neq i_2$ and a unique couple of subsets $A_1,A_2
\subset\ieu\backslash\{i_1,i_2\}$ with $\Card(A_1)=\ell-1,\,\Card(A_2)=\ell'-1$, and $A_1\cap A_2=\emptyset$ such 
that the event $\mathcal C_{A_1,A_2,i_1,i_2}$, defined by the following three conditions, occurs:
\begin{itemize}
\item (i) $\norm{X_{i_1}-x}+\norm{X_{i_2}-y}<\norm{y-x}$,
\item (ii) for all $(a_1,a_2)\in A_1\times A_2,\,\norm{X_{a_1}-x}<\norm{X_{i_1}-x},\,\norm{X_{a_2}-y}<\norm{
X_{i_2}-y}$, and 
\item (iii) for all $a\not\in A_1\cup A_2\cup\{i_1,i_2\},\,\norm{X_a-x}>\norm{X_{i_1}-x}$ and $\norm{X_a-y}>
\norm{X_{i_2}-y}$.
\end{itemize}
We can see $\mathcal D\subseteq\cup_{(i_1,i_2,A_1;A_2)}\mathcal C_{A_1, A_2, i_1, i_2}$ by letting $i_1$ be the 
index of the $\ell$-th nearest neighbour and $A_1$ be the set of indexes of the $\ell-1$ nearest neighbours of $x$;
and similarly $i_2$ be the index of the $\ell'$-th nearest neighbour and $A_2$ be the set of indexes of the
$\ell'-1$ nearest neighbours of $y$.
To see $\mathcal C_{A_1,A_2,i_1,i_2}\subseteq\mathcal D$, we can show that $a\in A_1\sqcup\{i_1\}\iff\norm{
X_a-x}\le\norm{X_{i_1}-x}$ and $a\in A_2\sqcup\{i_2\}\iff\norm{X_a-y}\le\norm{X_{i_2}-y}$ using Conditions (i-iii)
and the triangular inequality. $A_1$ thus characterizes the $\ell-1$ nearest neighbours of $x$ and $\norm{X_{i_1}
-y}=\tal$. Likewise, $A_2$ is the set of the $\ell'-1$ nearest neighbours of $y$ and $\norm{X_{i_2}-y}=
\hat\tau_{\ell'}(y)$. By Condition (i), we get the condition of $\mathcal D$. Here, $(A_1,A_2,\{i_1, i_2\})$
satisfying Conditions (i-iii) is unique because of the uniqueness of those nearest neighbours.

One can then write the event $\mathcal D$ as a finite union of the disjoint events $\mathcal C_{A_1,A_2,i_1,i_2}$.
The number of all possible triplets $(A_1,A_2,\{i_1,i_2\})$ is given by $C_{n,\ell,\ell'}$.
Computing the restriction of the expectation to one of these events, applying the change of variables $z_1:=
X_{i_1}-x$ and $z_2:=X_{i_2}-y$, we get the expression $$\Phi=C_{n,\ell,\ell'}\int\1{\norm{z_1}+\norm{z_2}<
\norm{y-x}}\varphi_1(z_1)\varphi_2(z_2)f(z_1,z_2)\,\mathrm d(z_1,z_2),\text{ where}$$ \begin{multline*}
f(z_1,z_2):=p(x+z_1)p(y+z_2)P(B(x,\norm{z_1}))^{\ell-1}P(B(y,\norm{z_2}))^{\ell'-1} \\ \times P(B(x,\norm{z_1})^c
\cap B(y,\norm{z_2})^c)^{n-\ell-\ell'}.
\end{multline*}
Further switching to polar coordinates yields the result.
\end{proof}

\begin{coro}\label{ind_cond}
In the context of Lemma \ref{pair_ind_distribution}, setting this time $$\Phi_c:=\E[\varphi_1(X_{\ik[\ell]}-x)
\varphi_2(X_{\hat i_{\ell'}(y)}-y)\mid\tal,\hat\tau_{\ell'}(y)]\,\1{\tal+\hat\tau_{\ell'}(y)<\norm{y-x}},$$ we 
have the conditional independence $$\Phi_c=\E[\varphi_1(X_{\ik[\ell]}-x)\mid\tal]\,\E[\varphi_2(X_{\hat i_{
\ell'}(y)}-y)\mid\hat\tau_{\ell'}(y)]\,\1{\tal+\hat\tau_{\ell'}(y)<\norm{y-x}},\text{ where}$$ 
$$\E[\varphi_1(X_{\ik[\ell]}-x)\mid\tal=r]=\frac 1{\int_{S^{d-1}}p(x+r\theta)\,\mathrm d\s(\theta)}\,\int_{S^{d-1}}
\varphi_1(r\theta)p(x+r\theta)\,\mathrm d\s(\theta),$$ and similarly for $y$.
\end{coro}
\begin{proof}
We use the expression with polar coordinates from Lemma \ref{pair_ind_distribution}, where the only quantity 
that is not radial is $P_x(\varphi_1,r_1)P_y(\varphi_2,r_2)$, hence the conditional independence and the
distribution after normalisation.
\end{proof}

\begin{lemma}\label{single_distribution}
Let $x\in\X$ and $1\le\ell\le k$, then the random vector $X_{\ik[\ell]}-x$ has density $f_\ell$ with respect 
to the Lebesgue measure of $\R^d$, where, setting $F_x(r)=P(B(x,r))$, we have 
$$f_\ell(z)=n\binom{n-1}{\ell-1}p(x+z)F_x(\norm z)^{\ell-1}(1-F_x(\norm z))^{n-\ell}.$$
\end{lemma}
\begin{proof}
One can show that for a bounded and measurable mapping $\varphi\colon\R^d\to\R$, $$\E\left[\varphi\left(X_{ 
\ik[\ell]}-x\right)\right]=\int_\X\varphi(z)f_\ell(z)\,\mathrm dx,$$ using some arguments very similar to that 
used for proving Lemma \ref{pair_ind_distribution}. Details are omitted.
\end{proof}

\begin{lemma}\label{negative_correlation}
Let $A$ and $B$ be disjoint Borel sets of $\X$, and $0\le\ell,\ell'\le n$, then we have the negative correlation
$$\Pb\left(\sum_{i=1}^n\mathbf 1_A(X_i)\leq\ell,\,\sum_{i=1}^n\mathbf 1_B(X_i)\leq\ell'\right)\leq\Pb\left(
\sum_{i=1}^n\mathbf 1_A(X_i)\leq\ell\right)\,\Pb\left(\sum_{i=1}^n\mathbf 1_B(X_i)\leq\ell'\right).$$
\end{lemma}
\begin{proof}
We prove the result by induction on the number of observations $n$. Assuming that $\sum_{i=1}^0(\cdot)=0$, the
result is trivially true when $n=0$. We define the events \begin{align*}
I_{\ell,\ell'}^{(n)} & :=\left(\sum_{i=1}^n\mathbf 1_A(X_i)\leq\ell,\,\sum_{i=1}^n\mathbf 1_B(X_i)\le\ell'\right),
\\ A_\ell^{(n)} & :=\left(\sum_{i=1}^n\mathbf 1_A(X_i)\leq\ell\right),\quad\text{and} \\ B_{\ell'}^{(n)} & :=
\left(\sum_{i=1}^n\mathbf 1_B(X_i)\leq\ell'\right),
\end{align*}
as well as the two quantities $a:=P(A)$ and $b:=P(B)$. To apply our induction hypothesis, we apply the law of 
total probability to the random variable $X_n$. It yields the three equations \begin{align*}
I_{\ell,\ell'}^{(n)} & =aI_{\ell-1,\ell'}^{(n-1)}+bI_{\ell,\ell'-1}^{(n-1)}+(1-a-b)I_{\ell,\ell'}^{(n-1)}, \\
A_\ell^{(n)} & =aA_{\ell-1}^{(n-1)}+(1-a)A_\ell^{(n-1)},\quad\text{and} \\ B_{\ell'}^{(n)} & =bB_{\ell'-1}^{(n-1)}+
(1-b)B_{\ell'}^{(n-1)}.
\end{align*}
Therefore, \begin{multline*}
A_\ell^{(n)}B_{\ell'}^{(n)}=abA_{\ell-1}^{(n-1)}B_{\ell'-1}^{(n-1)}+a(1-b)A_{\ell-1}^{(n-1)}B_{\ell'}^{(n-1)}+
b(1-a)A_\ell^{(n-1)}B_{\ell'-1}^{(n-1)} \\ +(1-a)(1-b)A_\ell^{(n-1)}B_{\ell'}^{(n-1)}.    
\end{multline*}
We recognise the terms $aA_{\ell-1}^{(n-1)}B_{\ell'}^{(n-1)},\,bA_\ell^{(n-1)}B_{\ell'-1}^{(n-1)}$, and 
$(1-a-b)A_\ell^{(n-1)}B_{\ell'}^{(n-1)}$, so that applying the induction hypothesis, we find that 
\begin{align*}
A_\ell^{(n)}B_{\ell'}^{(n)}-I_{\ell,\ell'}^{(n)} & \geq ab\,(A_{\ell-1}^{(n-1)}B_{\ell'-1}^{(n-1)}-A_{\ell-1}^{
(n-1)}B_{\ell'}^{(n-1)}-A_\ell^{(n-1)}B_{\ell'-1}^{(n-1)}+A_\ell^{(n-1)}B_{\ell'}^{(n-1)}) \\ & =ab(A_\ell^{
(n-1)}-A_{\ell-1}^{(n-1)})(B_{\ell'}^{(n-1)}-B_{\ell'-1}^{(n-1)}) \\ & \geq 0,   
\end{align*}
which concludes the proof by induction.
\end{proof}

\begin{coro}\label{negative_correlation_NN}
Let $x,y\in\X,\,1\le\ell,\ell'\le k$ and let $\chi_x,\chi_y\colon\R_+\to\R_+$ be non-decreasing functions.
Then $$\E[\chi_x(\tal)\,\chi_y(\hat\tau_{\ell'}(y))\,\1{\tal+\hat\tau_{\ell'}(y)<\norm{y-x}}]\leq\E[\chi_x(\tal)]
\,\E[\chi_y(\hat\tau_{\ell'}(y))].$$
\end{coro}
\begin{proof}
Let $R_x,\,R_y\ge 0$ such that $R_x+R_y<\norm{y-x}$, then the event $(\tal>R_x,\,\hat\tau_{\ell'}(y)>R_y)$
coincides with the event $\left(\sum_{i=1}^n\1{X_i\in B(x,R_x)}\leq\ell-1,\,\sum_{i=1}^n\1{X_i\in B(y,R_y)}\leq
\ell'-1\right)$. But because $R_x+R_y<\norm{y-x}$, the two Borel sets $B(x,R_x)$ and $B(y,R_y)$ are disjoint, and
we may apply Lemma \ref{negative_correlation}, which yields
$$\Pb(\tal>R_x,\,\hat\tau_{\ell'}(y)>R_y)\leq\Pb(\tal>R_x)\,\Pb(\hat\tau_{\ell'}(y)>R_y).$$
Now, we can write the expectation of nonnegative random variables thanks to tail distribution functions as 
follows. Let $\chi^+$ denote the right-continuous generalised inverse of a non-decreasing function $\chi$, then
\begin{align*}
\fbox{1-1} & :=\E[\chi_x(\tal)\,\chi_y(\hat\tau_{\ell'}(y))\,\1{\tal+\hat\tau_{\ell'}(y)<\norm{y-x}}] \\ & 
=\int_{\R_+^2}\Pb(\chi_x(\tal)>t,\,\chi_y(\hat\tau_{\ell'}(y))>s,\,\tal+\hat\tau_{\ell'}(y)<\norm{y-x})\,
\mathrm d(t,s) \\ & \leq\int_{\R_+^2}\1{\chi_x^+(t)+\chi_y^+(s)<\norm{y-x}}\Pb(\tal>\chi_x^+(t),\,\hat\tau_{\ell'}
(y)>\chi_y^+(s))\,\mathrm d(t,s) \\ & \leq\int_{\R_+^2}\Pb(\tal>\chi_x^+(t))\,\Pb(\hat\tau_{\ell'}(y)>\chi_y^+(s))
\,\mathrm d(t,s) \\ & =\E[\chi_x(\tal)]\,\E[\chi_y(\hat\tau_{\ell'}(y))].   
\end{align*}
\end{proof}

\begin{lemma}\label{density_V}
Suppose that Assumptions \ref{cond:reg1} to \ref{cond:reg3} hold true, and let $\psi:(\R^d)^{k-1}\to\R$ be a
measurable and bounded mapping, invariant with respect to any permutation of its coordinates. Then with the
notation of Lemma \ref{local_key_lemma}, conditionally on $\ta(x)$ and the set of indexes $I_{k-1}(x)$ of the 
$k-1$ nearest neighbours of $x$, we have the expression 
$$\E\left[\psi\left(V_1,\ldots,V_{k-1}\right)\mid\ta(x),I_{k-1}(x)\right]=\int_{B(0,1)^{k-1}}\psi(z_1,\ldots,z_{k-1})
\prod_{i=1}^{k-1}\frac{f_x(z_i)}{F_x(\tal)}\,\mathrm dz_1\cdots\,\mathrm dz_{k-1},$$
where $$f_x(z)=\1{x\in B(0,1)}\;p(x-z\ta(x))\,\ta(x)^d,$$
and $F_x(r):=P(B(x,z))$ denotes the cdf of the variable $\norm{X_1-x}$.
Furthermore, the normalisation constant is lower bounded by 
$$F_x\left(\ta(x)\right)=\int_{\R^d}f_x(z)\,\mathrm dz\geq c\pinf\vab\,\ta(x)^d.$$
\end{lemma}
\begin{proof}
For conciseness of notation, we set $Z_i:=\norm{X_i-x}$ for $1\leq i\leq n$. Let $I:=\{i_1,\ldots,i_{k-1}\}$ be a
subset of $\ieu$ with cardinality $k-1$. Let $\check g:\R_+\to\R$ and $\psi:\R^{d(k-1)}\to\R$ be two
measurable and bounded mappings, with $\psi$ being invariant with respect to any permutation of its coordinates. 
We also denote for $t>0$, $$\overline\psi(t):=\int_{B(0,1)^{k-1}}\psi(z_1,\ldots,z_{k-1})\prod_{i=1}^{k-1}p(x-
tz_i)t^{d(k-1)}\,\mathrm dz_1\cdots\mathrm dz_{k-1}.$$
Finally, we denote by $F_x$ the cdf of the variable $Z_1$. Using the independence properties between 
$X_1,\dots,X_n$, we have \begin{eqnarray*}
&& \E\left[\psi\left(V_1,\ldots,V_{k-1}\right) \check g\left(\ta(x)\right)\1{I_{k-1}(x)=I}\right] \\
&=& \sum_{i_k\notin I}\E\left[\psi\left(\frac{X_{i_1}-x}{Z_k},\ldots,\frac{X_{i_{k-1}}-x}{Z_k}\right)\prod_{
j\in I}\1{Z_j<Z_{i_k}}\check g\left(Z_{i_k}\right)\prod_{\ell\notin I\cup\{i_k\}}\1{Z_\ell>Z_{i_k}}\right] \\
&=& \sum_{i_k\notin I}\E\left[\overline\psi\left(Z_{i_k}\right)\check g\left(Z_{i_k}\right)\prod_{\ell\notin I\cup
\{i_k\}}\1{Z_\ell>Z_{i_k}}\right] \\
&=& \sum_{i_k\notin I}\E\left[\frac{\overline\psi\left(Z_{i_k}\right)}{F_x\left(Z_{i_k}\right)^{k-1}}\check g
\left(Z_{i_k}\right)\prod_{j\in I}\1{Z_j<Z_{i_k}}\prod_{\ell\notin I\cup\{i_k\}}\1{Z_\ell>Z_{i_k}}\right] \\
&=& \E\left[\frac{\overline\psi\left(\ta(x)\right)}{F_z\left(\ta(x)\right)^{k-1}}\check g\left(\ta(x)\right)
\1{I_{k-1}(x)=I}\right].
\end{eqnarray*}
From the definition of the conditional expectation, we deduce that $$\E\left[\psi\left(V_1,\ldots,V_{k-1}\right)
\mid\ta(x),I_{k-1}(x)\right]=\frac{\overline\psi\left(\ta(x)\right)}{F_x\left(\ta(x)\right)^{k-1}},$$
which is the required expression.
\end{proof}

\begin{lemma}\label{anti_concentration}
Let $(Z_j)_{1\le j\le\nu}$ be independent positive random variables such that there exist $C>0$ and $b\in(0,1)$
such that for all $1\le j\le\nu$ and all $\varepsilon>0,\,\Pb(Z_j\leq\varepsilon)\leq C\varepsilon^b$. Then for 
all $\varepsilon>0$, we have $$\Pb\left(\sum_{j=1}^\nu Z_j\leq\varepsilon\right)\leq\frac 1\nu\,\Gamma(b)\,\frac{
\Gamma(1+b)^{\nu-1}}{\Gamma(b\nu)}\left(C\varepsilon^b\right)^\nu.$$
\end{lemma}
\begin{proof}
We prove by induction on $\nu$ that $$\Pb\left(\sum_{j=1}^\nu Z_j\leq\varepsilon\right)\leq K_\nu\left(C
\varepsilon^b\right)^\nu=b^{\nu-1}\,(\nu-1)!\left(\prod_{j=1}^{\nu-1}\beta(1+b,bj)\right)\left(C\varepsilon^b
\right)^\nu.$$ For $\nu=1$, the statement writes $P(Z_1\leq\varepsilon)\leq C\varepsilon^b$, which is true 
by hypothesis. Suppose that the result is true for $\nu$ random variables, we can write \begin{align*}
\Pb\left(\sum_{j=1}^{\nu+1}Z_j\leq\varepsilon\right) & =\E\left[\Pb\left(\sum_{j=1}^\nu Z_j\leq\varepsilon-Z_{
\nu+1}\mid Z_{\nu+1}\right)\1{Z_{\nu+1}\leq\varepsilon}\right] \\ & \leq K_\nu C^\nu\E[(\varepsilon-Z_{\nu+1})^{
b\nu}\1{Z_{\nu+1}\leq\varepsilon}] \\ & =K_\nu C^\nu\int_{\R_+}\Pb(Z_{\nu+1}\leq\varepsilon,\,\varepsilon-Z_{
\nu+1}>s^{1/b\nu})\,\mathrm ds \\ & \leq K_\nu C^{\nu+1}\int_0^{\varepsilon^{b\nu}}(\varepsilon-s^{1/b\nu})^b\,
\mathrm ds \\ & =K_\nu  C^{\nu+1}\int_0^\varepsilon(\varepsilon-s)^b s^{-1+b\nu}\,b\nu\,\mathrm ds \\ & =K_\nu
C^{\nu+1}\,b\nu\,\varepsilon^{b(\nu+1)}\int_0^\varepsilon\left(1-\frac s\varepsilon\right)^b\left(
\frac s\varepsilon\right)^{-1+b\nu}\,\frac{\mathrm ds}\varepsilon \\ & =K_\nu\,b\nu\,\beta(1+b,b\nu)\;\left(C
\varepsilon^b\right)^{\nu+1} \\ & =K_{\nu+1}\left(C\varepsilon^b\right)^{\nu+1},
\end{align*}
which concludes the induction. To obtain the expression in the lemma's statement, observe that \begin{align*}
K_\nu & =b^{\nu-1}\,(\nu-1)!\prod_{j=1}^{\nu-1}\frac{\Gamma(1+b)\Gamma(bj)}{\Gamma(1+b(j+1))}\\ & =b^{\nu-1}\,
(\nu-1)!\,\Gamma(1+b)^{\nu-1}\frac{\Gamma(b)}{\Gamma(b\nu)}\prod_{j=1}^{\nu-1}\frac 1{b(j+1)} \\ & =\frac 1\nu
\,\Gamma(1+b)^{\nu-1}\frac{\Gamma(b)}{\Gamma(b\nu)}.
\end{align*}
\end{proof}

\newcommand{\ellmax}{\overline{\ell}}
\newcommand{\ellmin}{\underline{\ell}}

\begin{lemma}\label{lem:negligible_integral_simple_simplified}
Let $\kappa\in[0,\infty)$. Suppose that $k^2=O(n)$ and $\kappa=O(k)$. Then, for all $\ell\in\ieu[k]$, denoting 
$L_k=1-(2k)^{-1}$ and $C_{\ell,n}:=n\binom{n-1}{\ell-1}$, we have
\[
C_{\ell,n}\int_{\R_+}(p(x)\,\vab\,r^d)^{\ell-1}r^\kappa\exp(-nL_kp(x)\,\vab\,r^d)^{n-\ell}\,r^{d+1}\,\mathrm dr=
O(k^{\frac{2+\kappa}d}n^{-\frac{2+\kappa}d}).
\]
\end{lemma}
\begin{proof}
Using the bound $C_{\ell,n}\le\frac{n^\ell}{\Gamma(\ell)}$ and making the change of variables $s:=nL_kp(x)\,\vab\,
r^d$, we can write \begin{align*}
\fbox{1-1} & :=C_{\ell,n}\int_{\R_+}(p(x)\,\vab\,r^d)^{\ell-1}r^{2+\kappa}\exp(-nL_kp(x)\,\vab\,r^d)\,r^{d-1}\,
\mathrm dr \\ & \le\frac{n^\ell}{\Gamma(\ell)}\int_{\R_+}\left(\frac s{nL_k}\right)^{\ell-1}\left(\frac s{nL_kp(x)\,
\vab}\right)^{\frac{2+\kappa}d}\exp(-s)\,(dnL_kp(x)\,\vab)^{-1}\mathrm ds \\ & \le C n^{-\frac{2+\kappa}d}L_k^{
-\ell-\frac{2+\kappa}d}\frac{\Gamma\mathopen{}\left(\frac{2+\kappa}d+\ell\right)}{\Gamma(\ell)},
\end{align*}
where $C:=d^{-1}(\pinf\,\vab)^{-\frac{2+\kappa}d-1}$. Because $\frac{\Gamma\mathopen{}\left(\frac{2+\kappa}d+
\ell\right)}{\Gamma(\ell)}=O\left(\ell^{\frac{2+\kappa}d}\right)$, it remains to show that the factor 
$$L_k^{-\ell-\frac{2+\kappa}d}\leq L_k^{-k-\frac{2+\kappa}d}$$ is bounded in $k$, which is straightforward from the
convergence $e^{-c}=\lim_{k\rightarrow\infty}(1-c/k)^k$ for any $c>0$.
\end{proof}

\begin{lemma}\label{lem:PolyOneMinusFNearExpLnExtended}
Suppose that Assumptions~\ref{cond:reg1} and \ref{cond:reg3} hold true and that the support of $Q$, which is denoted 
by $\widetilde\X$, lies in the interior of $\X$. There exists $\varepsilon>0$ such that for all points $x$ and $y$ 
in $\widetilde\X$, all $n\in\N$, and all $(r_1,r_2)\in[0,\varepsilon/k]^2$,
\begin{align*}
& \va{(1-F_x(r_1)-F_y(r_2))^n-\exp(-np(x)\,\vab\,r_1^d-np(y)\,\vab\,r_2^d)} \\ &\le 2n\overline p\,\vab\,O(r_1^{d+1}
+r_2^{d+1})\exp(-n\,\vab\,\left(p(x)r_1^d +p(y)r_2^d\right)L_k),
\end{align*}
with $L_k:=1-(2k)^{-1}$.
\end{lemma}
\begin{proof}
We can write \begin{equation}
(1-F_x(r_1)-F_y(r_2))^n=\chi(\delta_{r_1,r_2}),\label{eq:ExpWithDeltaInLemma}
\end{equation}
where $\chi:\xi\mapsto\exp(-n(p(x)\,\vab\,r_1^d\,+p(y)\,\vab\,r_2^d)\,(1+\xi))$, and \begin{equation*}
\delta_{r_1,r_2}:=\begin{cases}
\displaystyle\frac{{(-p(x)\,\vab\,r_1^d-p(y)\,\vab\,r_2^d)-\ln(1-F_x(r_1)-F_y(r_2))}}{{p(x)\,\vab\,r_1^d+ p(y)\,
\vab\,r_2^d}} & \text{if $r_1+r_2>0$}, \\ 0 & \text{if $r_1+r_2=0$.}
\end{cases}
\end{equation*}
By taking $\varepsilon$ sufficiently small, specifically $\varepsilon\le(2\overline p\vab)^{-1/d}$, we can ensure 
$1-F_x(r_1)-F_y(r_2)\ge 1/2$ for all $(r_1,r_2)\in[0,\varepsilon/k]^2$.
Then, applying the Taylor theorem to the function $\xi\mapsto\ln(1-\xi)$, with the bound $\sup_{\xi\in[0,1/2]}
\norm{\nabla^2\ln(1-\xi)}\le 2$, we have for $(r_1,r_2)\in[0,\varepsilon/k]^2$,
\begin{align}
\va{\ln (1-F_x(r_1)-F_y(r_2))-(-F_x(r_1)-F_y(r_2))} & \le 2(F_x(r_1)+F_y(r_2))^2\nonumber \\ & \le 4\overline p\,
\vab\,(r_1^{2d}+r_2^{2d}).\label{eq:lnOneMinusFToMinusF}
\end{align}
By the compactness of $\X$, choosing sufficiently small $\varepsilon$ also ensures that $B(x,\varepsilon)\subset\X$
and $B(y,\varepsilon)$ for all points $x$ and $y$ in $\widetilde\X$.
By the triangular inequality, Eq.~\eqref{eq:lnOneMinusFToMinusF}, the Lipschitz-continuity of $p$, and the inclusion
$B(x,\varepsilon/k)\subset\X$ and $B(y,\varepsilon/k)\subset\X$, we can bound the numerator of $\delta_{r_1,r_2}$ as
\begin{align*}
& (p(x)\,\vab\,r_1^d+p(y)\,\vab\,r_2^d)\times\va{\delta_{r_1,r_2}} \\ & \le\va{\ln(1-F_x(r_1)-F_y(r_2))-(- F_x(r_1)-
F_y(r_2))} \\ & \phantom{\le{}}+\va{-F_x(r_1)-F_y(r_2)-(-p(x)\,\vab\,r_1^d-p(y)\,\vab\,r_2^d)} \\ & =O(r_1^{2d}+
r_2^{2d})+\va{\int_{\norm{z-x}\le r_1}\lr(){p(z)-p(x)}\mathrm d(z)}+\va{\int_{\norm{z-y}\le r_2}\lr(){p(z)-p(y)}\,
\mathrm dz} \\ & \le O(r_1^{2d}+r_2^{2d})+K\,\vab\,(r_1^{d+1}+r_2^{d+1})=O(r_1^{d+1}+r_2^{d+1}),
\end{align*}
where $K$ is the Lipschitz constant of $p$. Since $p(x)$ and $p(y)$ are bounded below by $\pinf>0$, we deduce that
$\va{\delta_{r_1,r_2}}=O(r_1+r_2)=O(2\varepsilon/k)$ and that taking $\varepsilon$ sufficiently small ensures 
$\va{\delta_{r_1,r_2}}\le (2k)^{-1}$ for all $(r_1,r_2)\in[0,\varepsilon/k]^2$ and all $n\in\N$.
By applying the mean value theorem to the function $\chi$ (Eq.~\ref{eq:ExpWithDeltaInLemma})
on the interval $[-\mathopen{}\va{\delta_{r_1,r_2}},\va{\delta_{r_1,r_2}}]$, we get
\begin{align*}
& \va{\chi(\delta_{r_1,r_2})-\chi(0)}\le\va{\chi(\delta_{r_1,r_2})-\chi(-\delta_{r_1,r_2})}\le\sup_{\xi\in
[-\vert{\delta_{r_1,r_2}}\vert,\vert{\delta_{r_1,r_2}}\vert]}\va{\chi'(\xi)}\times 2\va{\delta_{r_1,r_2}} \\
& \le 2\sup_{\xi\in[-(2k)^{-1},(2k)^{-1}]}\exp(-n\,\vab\,(p(x)r_1^d+p(y)r_2^d)(1+\xi)) \\ & \phantom{\le{}}\times n
\,\vab\,(p(x)r_1^d+p(y)r_2^d)\times\va{\delta_{r_1,r_2}} \\ & \le 2n\overline p\,\vab\,O(r_1^{d+1}+r_2^{d+1}) 
\exp(-n\,\vab\,(p(x)r_1^d+p(y)r_2^d)L_k),
\end{align*}
where $L_k=1-(2k)^{-1}$. Combining this result with \eqref{eq:ExpWithDeltaInLemma}, we obtain the lemma.
\end{proof}

\begin{lemma}\label{lem:CanChangeExponentOfOneMinusFExtendedCorrected}
Suppose that Assumptions \ref{cond:reg1} and \ref{cond:reg3} hold true.
Then, there exists $\varepsilon>0$ such that for all points $x$ and $y$ in $\widetilde\X$, the support of $Q$
included in the interior of $\X$, all $n\in\N$, all $\ell,\ell'\in\ieu[k]^2$, with $k^2=o(n)$, and all $(r_1,r_2) 
\in[0,\varepsilon/k]^2$, we have \begin{align*}
& \va{(1-F_x(r_1)-F_y(r_2))^{n-\ell-\ell'}-(1-F_x(r_1)-F_y(r_2))^n} \\ & =O\mtlr(){(\ell+\ell')\exp(-n\,\vab\,
(p(x)L_kr_1^d+p(y)L_kr_2^d))(r_1^d+r_2^d)},
\end{align*}
where $L_k:=1-(2k)^{-1}$.
\end{lemma}
\begin{proof}[Proof of Lemma~\ref{lem:CanChangeExponentOfOneMinusFExtendedCorrected}]
As in the proof of Lemma \ref{lem:PolyOneMinusFNearExpLnExtended}, one can take $\varepsilon>0$ small enough such 
that $F_x(r_1)+F_y(r_2)\leq 1/2$. By the mean value theorem, we have \begin{align*}
& \va{(1-F_x(r_1)-F_y(r_2))^{n-\ell-\ell'}-(1-F_x(r_1)-F_y(r_2))^n} \\ & =(1-F_x(r_1)-F_y(r_2))^{n-\ell-\ell'}
\va{1-(1-F_x(r_1)-F_y(r_2))^{\ell+\ell'}} \\ & \le(1-F_x(r_1)-F_y(r_2))^{n-\ell-\ell'}\sup_{\xi\in[0,F_x(r_1)+F_y
(r_2)]}(\ell+\ell')(1-\xi)^{{\ell+\ell'}-1}\,(F_x(r_1)+F_y(r_2)) \\ & =\lr(){\ell+\ell'} (1-F_x(r_1)-F_y(r_2))^{
n-{\ell-\ell'}}\overline p\,\vab\,(r_1^d+r_2^d) \\ & \le\lr(){\ell+\ell'}\exp(-(n-{\ell-\ell'})(F_x(r_1)+F_y(r_2)))
O(r_1^d+r_2^d).
\end{align*}
By the compactness of $\X$, there exists $\varepsilon>0$ such that for all points in $\widetilde\X$, we have 
$B(x,\varepsilon)\subset\X$ and $B(y,\varepsilon)\subset\X$. Then, we can give a lower bound on $F_x(r_1)+F_y(r_2)$
as \begin{align*}
F_x(r_1) & \ge\int_{B(x,r_1)}(p(x)-\va{p(z)-p(x)})\,\mathrm dz \\ & \ge p(x)\,\vab\,r_1^d(1-\norm{\nabla p}_\infty
r_1)\ge p(x)\,\vab\,r_1^d(1-C_p\varepsilon/k) \\ & \ge p(x)\,\vab\,r_1^d(1-(4k)^{-1})
\end{align*}
with $C_p:=\norm{\nabla p}_\infty/\pinf$, by making the constant $\varepsilon$ sufficiently small so that
$C_p\varepsilon \le 1/4$ will hold. Likewise, $F_y(r_2)\ge p(y)\,\vab\,r_2^d(1-(4k)^{-1})$.
Putting the results together, we get \begin{align*}
& \va{(1-F_x(r_1)-F_y(r_2))^{n-\ell-\ell'}-(1-F_x(r_1)-F_y(r_2))^n} \\ & \le\va{{\ell+\ell'}}\exp(-n\,\vab\,(p(x)
L_{k,\ell+\ell',n}r_1^d+p(y)L_{k,\ell+\ell',n}r_2^d))O(r_1^d+r_2^d),
\end{align*}
where $L_{k,\ell+\ell',n}=\frac{n-\ell-\ell'}n(1-(4k)^{-1})$.
Since $L_{k,\ell+\ell',n}\ge L_k:=(1-(2k)^{-1})$ for sufficiently large $n$, we can further bound the last 
expression from above by
\[
O\mtlr(){\lr(){\ell+\ell'}\exp(-n\,\vab\,(p(x)L_kr_1^d+p(y)L_kr_2^d))(r_1^d+r_2^d)}.
\]
\end{proof}

\end{document}